\documentclass[oneside]{amsart}
\usepackage{geometry}
\usepackage[shortlabels]{enumitem}
\usepackage[utf8]{inputenc}
\usepackage{amsmath}
\usepackage{amssymb}
\usepackage{amsthm}
\usepackage{tikz}
\usepackage{graphicx}
\usepackage{float}
\usepackage{svg}
\usepackage{verbatim}
\usepackage{url}
\usepackage{mathtools}
\usepackage{microtype}
\usepackage{thm-restate}
\usepackage{enumitem}
\usepackage[normalem]{ulem}
\usepackage[hidelinks]{hyperref}
\usepackage[nameinlink]{cleveref}
\usepackage{caption}
\theoremstyle{plain}
\newtheorem{theorem}{Theorem}[section]
\theoremstyle{definition}
\newtheorem{definition}[theorem]{Definition}

\theoremstyle{plain}
\newtheorem{lemma}[theorem]{Lemma}

\theoremstyle{definition}
\newtheorem{example}[theorem]{Example}

\theoremstyle{plain}
\newtheorem{proposition}[theorem]{Proposition}

\theoremstyle{remark}
\newtheorem*{remark}{Remark}

\theoremstyle{plain}

\title{Detecting right-veering diffeormorphisms}
\author{Miguel {Orbegozo Rodriguez}}
\email{miguel.orbegozo@unine.ch}
\date{}

\newcommand{\qua}{\hskip 0.4em \ignorespaces}
\def\arxiv#1{\relax\ifhmode\unskip\qua\fi
\href{http://arxiv.org/abs/#1}%
{\tt arXiv:\penalty -100\unskip#1}}

\def\MR#1{\relax\ifhmode\unskip\qua\fi
\href{https://mathscinet.ams.org/mathscinet-getitem?mr=#1}{\tt MR#1}}
\def\ZB#1{\relax\ifhmode\unskip\qua\fi
\href{https://zbmath.org/?q=an:#1}{\tt Zbl\:#1}}
\def\xox#1{\csname xx#1\endcsname}

\renewenvironment{thebibliography}[1]{
  \begin{oldthebibliography}{#1}\small
    \setlength{\itemsep}{.5ex}
    \setlength{\parskip}{0em}
}
{
  \end{oldthebibliography}
}

\begin{document}

\begin{abstract}
    A result of Honda, Kazez, and Mati\'{c} states that a contact structure is tight if and only if all its supporting open books are right-veering. We show a combinatorial way of detecting the left-veering arcs in open books, implying the existence of an algorithm that detects the right-veering property for compact surfaces with boundary.
\end{abstract}

\maketitle

\section{Introduction}

Let $M$ be a closed, oriented $3$-manifold. A \emph{contact structure} on $M$ is a plane field $\xi$ that is maximally non-integrable; equivalently, it can be expressed as the kernel of a $1$-form $\alpha$ such that $\alpha \wedge d \alpha > 0$. An \emph{open book decomposition} is a pair $(B, \pi)$, where $B$, called the \emph{binding}, is an oriented link in $M$; and $\pi : M \setminus B \rightarrow S^1$ is a fibration of the complement of the binding over the circle such that the fibres, called the \emph{pages}, are (diffeomorphic) compact surfaces with common boundary $B$. This is determined up to diffeomorphism by the pair $(\Sigma, \varphi)$, where $\Sigma$ is a page and $\varphi$ is the mapping class of the return map of $\pi$. Thanks to the following result of Giroux, contact structures can be studied by studying open book decompositions: 
there exists a 1-1 correspondence between contact structures up to isotopy and open book decompositions up to positive stabilisation~\cite{Giroux}.

Contact structures exhibit a fundamental dichotomy; those that admit an embedding of a disc $D^2$ such that the boundary of the disc is tangent to the contact structure are called \emph{overtwisted}, while those that do not are called \emph{tight}. Overtwisted contact structures were classified by Eliashberg, who showed in \cite{Eliashberg} that there is a unique isotopy class of overtwisted contact structures in each homotopy class of plane fields. Tight contact structures, on the other hand, are not as well understood. Thus being able to distinguish between tight and overtwisted contact structures becomes an important problem. \medskip

A result of Honda, Kazez, and Mati\'{c} \cite{Right-veering} states that a contact structure is tight if and only if all its supporting open book decompositions are \emph{right-veering}. An open book is right-veering if its monodromy sends every properly embedded arc to the right (a precise definition will be given in Section \ref{Preliminaries}). Therefore existence of a single arc that is not sent to to the right (a \emph{left-veering arc}) is enough to guarantee overtwistedness of the contact structure. However, it is often difficult to determine if such an arc exists. Indeed, if we take a \emph{basis} of the surface (a collection of arcs which cut the surface into a disc), the images of these arcs under a diffeomorphism determine it up to isotopy; however it is possible to find monodromies that are not right-veering and yet send each arc of some basis to the right. Looking at more arcs does not necessarily solve the problem, as similar counterexamples can be found, for example, when the collection of arcs we look at is a \emph{complete set}, i.e a maximal collection of pairwise disjoint and non-isotopic arcs.\medskip

Thus the usual approach to showing whether an open book is right-veering is to either exhibit an arc that goes to the left, which is found in a non-systematic way, or divide all arcs into different classes and then show that each class can only contain arcs sent to the right. The main issue with this is that in both cases the argument is case dependent. \medskip

Our main theorem shows that a left-veering arc can be detected combinatorially from a basis of arcs and its image.

\begin{restatable}%
{theorem}{main}  \label{TowerCollection1}
Let $(\Sigma, \varphi)$ be an open book, and $\Gamma$ a basis for $\Sigma$ with all arcs duplicated. Suppose that there exists a left-veering arc $\gamma$, which we can assume to be minimal. Then there exists a collection of extended towers $\{\mathcal{T}_i\}_{i = 1}^{N}$ (where $N$ is the number of intersections between $\gamma$ and the basis) supported in (subcollections of) $\Gamma$ such that: 
\begin{itemize}
    \item $\mathcal{T}_1$ is a completed extended tower.
    \item $\mathcal{T}_i$ is a completed partial extended tower, whose starting point is the adjacent point to the connecting vertex of $\mathcal{T}_{i-1}$.
    \item $\mathcal{T}_N$ is a incomplete partial extended tower, whose starting point is the adjacent point to the connecting vertex of $\mathcal{T}_{N-1}$.
\end{itemize}

Conversely, if we have such a collection, then there exists a left-veering arc $\gamma$.

\end{restatable}

The non-standard terminology will be defined precisely in Sections \ref{Preliminaries} and \ref{Towers}. For now, we provide an overview of the result. We show that, given an open book $(\Sigma, \varphi)$ and a basis $\mathcal{B}$ of $\Sigma$, a left-veering arc induces a collection of objects, called \emph{extended towers}, that are constructed using $\mathcal{B}$. Conversely, the existence of such a collection implies the existence of a left-veering arc, which moreover can be constructed from the extended towers.

We first briefly introduce what extended towers are. Let $(\Sigma, \phi)$ be an open book, and $\Gamma$ a collection of arcs such that the arcs of $\Gamma$, together with an extra arc $\gamma$, cut out a disc $D$ from the surface. An \emph{extended tower supported in $\Gamma$} is a collection of regions (i.e. $n$-gons) with boundaries on $\Gamma \cup \varphi(\Gamma)$, with certain properties that give information on whether the arc $\gamma$ is left-veering, or isotopic to $\varphi(\gamma)$.

\subsection{Strategy of the construction}

First we show that using \emph{regions} we can detect a left-veering arc and a fixable arc segment in a simple case, that is, when they are contained in a $6$-gon cut out by $3$ arcs, two of which belong to our chosen arc collection. Regions are $n$-gons whose boundary lies on the arc collection, and have as vertices intersection points of the arc collection, labelled alternatively by $\bullet$-points and $\circ$-points. Then, a region is \emph{completed} if there exists another region with the same $\circ$-points but opposite orientation, and \emph{incomplete} otherwise --see Section \ref{Preliminaries} for a more precise definition, and the examples below for an illustration. Proposition \ref{InitialFixed} says that completed regions correspond to fixable arcs, and Proposition \ref{InitialLV} says that incomplete regions correspond to left-veering arcs. 

In Subsection \ref{Minimal Arcs} we show that existence of a left-veering arc implies the existence of a \emph{minimal left-veering arc}. This is a left-veering arc which can be divided, by intersections with a chosen arc collection, into arc segments, with disjoint interior with the basis, that are \emph{fixable} (i.e they have endpoints which coincide with their images, and relative to their endpoints they are isotopic to their image), and a unique arc segment that is left-veering.

Then in Subsection \ref{Base Case} we prove that the regions in Propositions \ref{InitialLV} and \ref{InitialFixed} form extended towers. This is our base case. The idea is that \emph{completed} (respectively \emph{incomplete}) extended towers generalise the completed (respectively incomplete) regions found in Section \ref{Preliminaries}, when the arc collections are complicated and feature multiple regions. We remark that the definitions from Section \ref{Towers} are so that the objects we consider completely characterise fixable and left-veering arcs.

Our inductive step is then to show that an extended tower supported in an arc collection with $n$ arcs induces an extended tower supported in an arc collection with $n+1$ arcs, and conversely. Moreover, we show that the properties of the extended tower are preserved. The proof of this inductive step is cumbersome and is broken down into several Lemmas in Subsection \ref{Inductive Step} to account for the different cases. 

Using this we show that if we have an arc collection that, together with a left-veering arc, cuts out a disc, then the left-veering arc is detected by an incomplete extended tower supported in the arc collection. Similarly, a fixable arc segment is detected by a completed extended tower. This is the content of Theorems \ref{TowerLV} and \ref{TowerFixedArc}.

Finally, we show that in the case of the minimal left-veering arc not being disjoint from our chosen basis, we can detect it with a collection of completed extended towers and one incomplete extended tower. This constitutes our main result, Theorem \ref{TowerCollection1}. 

The number $N$ of extended towers is the number of intersections of the left-veering arc with the chosen basis. Moreover, this corresponds to a point $\alpha \cap \varphi(\alpha)$ for some arc $\alpha$ in the basis. The number of such points is finite and gives an upper bound for $N$. Moreover, each extended tower is also a finite collection of regions. Therefore, this implies the existence of the following --crude-- algorithm. Take a basis of arcs $\mathcal{B}$ and consider all collections of $n$ extended towers, for all $n\leq N$, where $N$ is the total number of intersections of arcs of the basis with their images. This is a finite procedure as there is a finite number of objects to consider, so we either find a collection of extended towers that construct a left-veering arc, or the algorithm stops and concludes that the monodromy is right-veering. In Section \ref{Example}, we explain why this is, in general, not very efficient.

We can see the strategy in Figure \ref{fig:Strategy}. Once we have a basis with arcs duplicated, they cut $\Sigma$ into a disc. Moreover, the arcs from the basis divide the minimal left-veering arc $\gamma$ into segments which are either fixable or left-veering, and disjoint from the basis except at endpoints. First, Theorem~\ref{TowerFixedArc} shows that the segment $\gamma_1$ being fixable is detected by a completed extended tower supported in the arc collection that, together with $\gamma_1$, cuts out the disc $P_1$ (which exists because the arcs in this arc collection belong to our chosen basis, which cuts the entire surface into a disc), and similarly for $\gamma_2$. Then, Theorem \ref{TowerLV} shows that $\gamma_3$ being left-veering is detected by an incomplete extended tower supported in the arc collection that together with $\gamma_3$ cuts out the disc $P_3$.

\begin{figure} htp
    \centering
    \includegraphics[width=7cm]{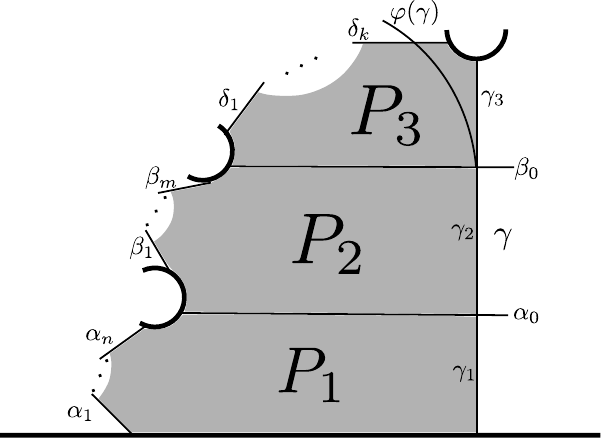}
    \caption[Algorithm strategy.]{A left-veering arc $\gamma$, which we want to detect using extended towers. The arcs $\alpha_i$, $\beta_j$, $\delta_k$ are arcs from a chosen basis $\mathcal{B}$. The arc is minimal with respect to $\mathcal{B}$, which means that its image is itself up until the point $x$ --that is, all segments up to $x$ are fixable-- and after that it goes to the left --that is, the last segment is left-veering. Finally, we will use the convention that the thicker lines represent the boundary of the surface.}
    \label{fig:Strategy}
\end{figure}

\subsection{Examples}

We illustrate the strategy outlined in the previous subsection with a couple of concrete examples. First, we have a torus with a right-veering basis, which supports an extended tower consisting of a single region $R_1$, see Figure \ref{fig:torus}. This is \emph{incomplete}, that is, we cannot find another region which shares its $\circ$-points with $R$. By Proposition \ref{InitialLV}, this implies that the arc $\beta$ given by the sum of the arcs of the basis is left-veering.

\begin{figure}htp
    \centering
    \includegraphics[width=4cm]{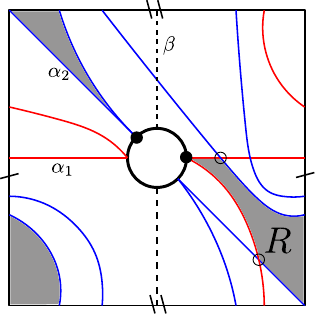}
    \caption[A non-right-veering monodromy on a punctured torus.]{The extended tower $\{R\}$ is incomplete, showing that the arc $\beta = \alpha_1 + \alpha_2$ is left-veering.}
    \label{fig:torus}
\end{figure}

Second, in Figure \ref{fig:Two-steps} we have an example where we construct the left-veering arc using two extended towers. First, the arcs $\alpha_1$ and $\alpha_2$ support a region $R_1$ which is completed by the region $R'_1$, since they have the same $\circ$-points. This implies by Proposition \ref{InitialFixed} that the segment $\gamma_1$ is fixable, that is, we can perform an isotopy of $\varphi$ that fixes not only the boundary, but also the point $x$, such that $\varphi(\gamma) = \gamma$. Then the arcs $\alpha_2$ and $\alpha_3$ support an incomplete extended tower consisting of the single incomplete region $R_2$. (To see that it is incomplete, observe that the shaded part $A$ of the surface, where the completion would have to be, is not a disc). This implies that the segment $\gamma_2$ is left-veering. Note that it makes sense to talk about left-veeringness in this context since we have fixed the point $x$. Putting everything together, we get that the arc $\gamma = \gamma_1 \cup \gamma_2$ is left-veering.

\begin{figure} [htp]
    \centering
    \includegraphics[width=10cm]{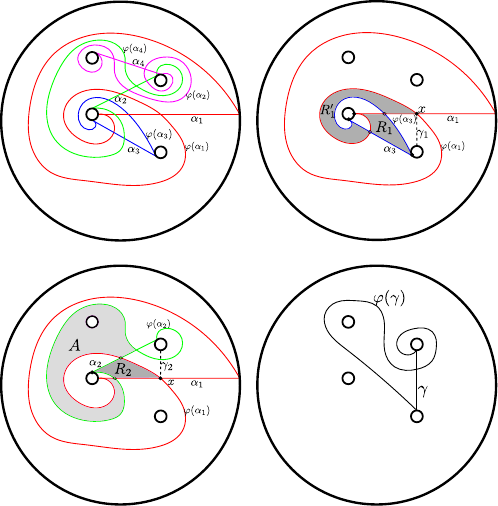}
    \caption[Two extended towers.]{On the top left, a basis of arcs and their images, determining the monodromy. On the top right, the extended tower $\{R_1, R'_1\}$ is completed, so $\gamma_1$ is fixable. On the bottom left, $\{R_2\}$ is incomplete, so $\gamma_2$ is left-veering. This implies that $\gamma = \gamma_1 \cup \gamma_2$ is left-veering, as we can see on the bottom right.}
    \label{fig:Two-steps}
\end{figure}

\subsection{Relationship to knot Floer homology}

A recent result of Baldwin, Ni, and Sivek \cite{BNS} characterises right-veering diffeomorphisms for open books with connected boundary in terms of an invariant defined using knot Floer homology called $b(K)$ (where $K$ is the binding), originally defined in \cite{Baldwin-Vela-Vick}. This invariant counts the first level of the filtration induced by the knot where the contact class vanishes. Baldwin, Ni, and Sivek's result then says that the monodromy of a fibred knot is not right-veering if and only if $b(K) = 1$, that is, the contact class vanishes in the lowest possible filtration level. It seems possible to link this result with ours, this is the subject of work in progress. 

Indeed, take for example the open book in Figure \ref{fig:torus}. This is given in \cite{BNS} as an example of an open book that is not right-veering but has a right-veering basis.
We saw that the extended tower given by the region $R$ constructs the left-veering arc. However, after a suitable isotopy of the arcs to obtain the knot Floer complex, it also realises a differential killing the contact class in the lowest filtration level, showing that $b(K) = 1$ as predicted by \cite{BNS}. This suggests that the regions appearing in extended towers constructing a left-veering arc could be related to differentials realising $b(K) = 1$. The problem in establishing a connection is that with more complicated extended towers the relationship with knot Floer differentials is not so clear. We also stress that our methods directly construct the left-veering arc, while the ones in \cite{BNS} merely show that such an arc exists, as well as the fact that our methods work for multiple boundary components. Indeed, this is the case with most of the examples presented in this paper.   

\subsection{Outline}

The outline of the paper is as follows. In Section \ref{Preliminaries} we give some preliminary definitions and we introduce regions, which will be the building blocks of our extended towers, and we show that they detect left-veering arcs and fixable arc segments in simple cases. In Section \ref{Towers} we define extended towers and their properties. The main results are proved in Section \ref{Results}. We first show the existence of minimal left-veering arcs, and we prove by induction that each segment of such an arc is detected in a basis by an extended tower. Finally, putting all the arc segments together gives the collection of extended towers, which are connected by the fixed points that are the endpoints of the arc segments. To conclude, in Section \ref{Example} we show that the collection of extended towers needed to detect a left-veering arc can be arbitrarily large, but that in some cases this systematic way of detecting a left-veering arc is not the most efficient. We also include a detailed example of a left-veering arc being detected by a collection of more than one extended tower. \medskip

\subsection{Acknowledgements}

I would like to thank my supervisor Andy Wand for invaluable advice and support. I would also like to thank the anonymous referee for many useful comments and suggestions.
\section{Preliminaries} \label{Preliminaries}
We start with some definitions regarding open books, and we introduce the concept of a region.

\begin{definition}
    Let $\Sigma $ be a compact surface with nonempty boundary. A \emph{properly embedded arc} is the image of an embedding $\alpha: [0,1] \hookrightarrow \Sigma$ such that $\alpha(0), \alpha(1) \in \partial \Sigma$. An \emph{arc segment} is the image of an embedding of the unit interval that is not necessarily proper, i.e we do not require that $\alpha(0), \alpha(1) \in \partial \Sigma$.
\end{definition}

\begin{definition}
    Let $(\Sigma, \varphi)$ be an open book. An \emph{arc collection} in $\Sigma$ is a set of pairwise disjoint properly embedded arcs. An arc collection $\mathcal{B}$ such that $\Sigma \setminus \mathcal{B}$ is a disc is called a \emph{basis}. 
\end{definition}

The significance of bases comes from the fact that a mapping class is uniquely determined by its action on any basis.

\begin{definition}
    Let $(\Sigma, \varphi)$ be an open book, and  $\Gamma$ an arc collection in $\Sigma$. If $\Sigma \setminus \Gamma$ contains an $n$-gon component with exactly one edge on each element of $\Gamma$, we say $\Gamma$ \emph{cuts out an $n$-gon}.
\end{definition}

\begin{definition}
    Let $\alpha_1, \alpha_2$ be disjoint properly embedded oriented arcs in a compact surface with boundary $\Sigma$, such that there is a boundary arc going from $\alpha_1(1)$ to $\alpha_2(0)$. The \emph{arc-slide} of $\alpha_1$ and $\alpha_2$ is (the isotopy class of) the arc $\beta$ that starts at $\alpha_1(0)$ and ends at $\alpha_2(1)$, such that $\alpha_1, \alpha_2$, and $\beta$ cut out a $6$-gon from $\Sigma$ whose standard boundary orientation coincides with the orientation from $\alpha_1$, $\alpha_2$, and $-\beta$.
\end{definition}

\begin{figure}[htp]
    \centering
    \includegraphics[width=4cm]{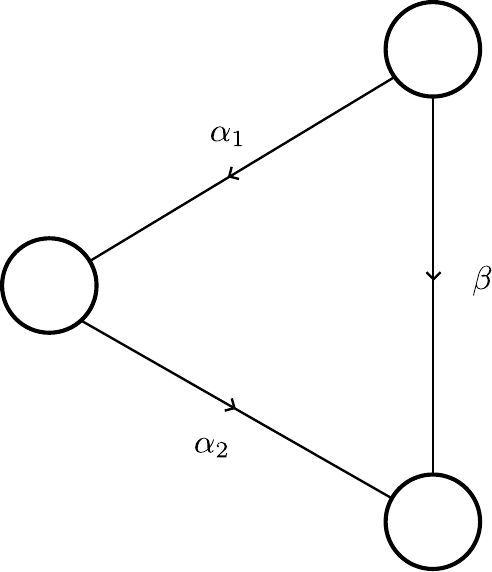}
    \caption[Arc-slide of two arcs.]{The arc-slide $\beta$ of two arcs $\alpha_1$ and $\alpha_2$. Observe we need to reverse the orientation of $\beta$ to obtain an orientation of the boundary of the $6$-gon.}
    \label{fig:Arc-slide}
\end{figure}

 Any two bases for $\Sigma $ can be related by a sequence of arc-slides (see \cite{Contact-class} for a proof of this fact). We can also extend this definition to larger families of arcs; if a collection of pairwise disjoint arcs $\{ \alpha_i \}_{i=1}^n$ is such that there is a boundary arc going from $\alpha_i(1)$ to $\alpha_{i+1}(0)$, then the arc $\beta$ that starts at $\alpha_1 (0)$ and ends at $\alpha_n(1)$ and cuts out a disc from $\Sigma$ (whose standard boundary orientation coincides with the orientation from $\alpha_i$ and $-\beta$) is called the \emph{arc-sum} of $\{ \alpha_i \}_{i=1}^n$.

 We now give the notion of right-veering arcs as introduced in \cite{Right-veering}. Our definition is phrased in a slightly different way, in order to be consistent with the orientation convention that we will use, but is equivalent to the one in \cite{Right-veering}.

\begin{definition} 
Let $(\Sigma, \varphi)$ be an open book decomposition, and let $\alpha$ be an oriented properly embedded arc with starting point $x$. We will adopt the convention that its image $\varphi(\alpha)$ is given the opposite orientation to $\alpha$. We then say that $\alpha$ is \emph{right-veering} (with respect to $\varphi$) if $\varphi(\alpha)$ is isotopic to $ \alpha$ or, after isotoping $\alpha$ and $\varphi(\alpha)$ so that they intersect transversely with the fewest possible number of intersections,  $( \alpha'(0),\varphi(\alpha)'(1))$ define the orientation of $\Sigma$ at $x$. In this latter case we will say that $\alpha$ is \emph{strictly right-veering}. If $\alpha$ is not right-veering we say it is \emph{left-veering}.
\end{definition} 

In Figure \ref{fig:RVArc} we can see that intuitively a right-veering arc $\alpha$ is such that $\varphi(\alpha)$ is to the right of $\alpha$ near the starting point once we have isotoped them so that they have the fewest possible number of intersections.

\begin{figure}[htp]
    \centering
    \includegraphics[width=5cm]{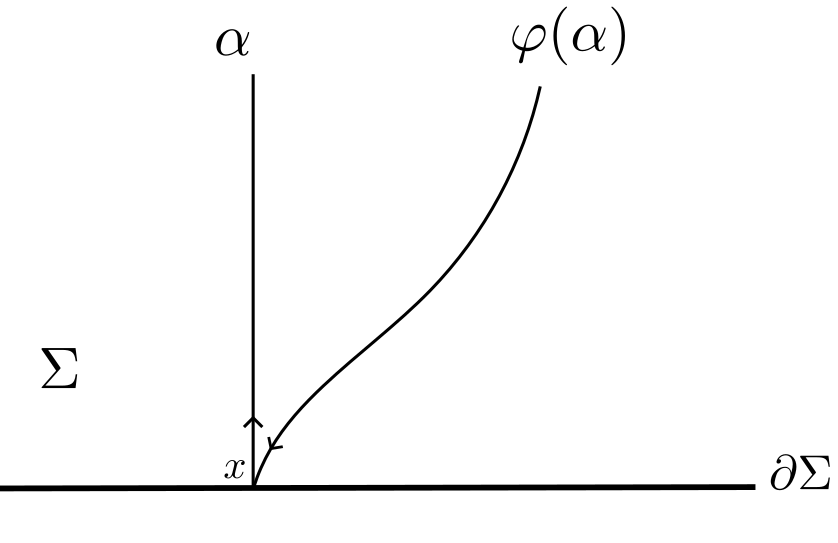}
    \caption[Right-veering arc.]{A (strictly) right-veering arc $\alpha$.}
    \label{fig:RVArc}
\end{figure}

Although this definition only refers to the starting point of the oriented arc, we will usually say that an arc is right-veering to mean that both itself and the arc with opposite orientation are right-veering (thus referring to both endpoints). However, when we say that an arc is left-veering we only refer to its starting point. This is because one oriented left-veering arc is enough for an open book to support an overtwisted contact structure, by \cite[Theorem 1.1]{Right-veering}, so we only need to detect one.

\begin{definition}
Let $(\Sigma, \varphi)$ be an open book decomposition. We say that $\varphi$ is \emph{right-veering} if every oriented properly embedded arc in $\Sigma$ is right-veering.
\end{definition}

Sometimes we will say that the open book itself is right-veering when the monodromy is right-veering.

Figure \ref{fig:RVBasis} shows that, to determine if an open book is right-veering, it is not enough to check that every arc of a basis is right-veering. The page is a planar surface with $4$ boundary components, and the monodromy is determined by the images of the arcs from the basis. As can be seen in Figure \ref{fig:RVBasis}, the dotted arc is left-veering, and yet the monodromy sends each arc from the basis to the right.

\begin{figure}[htp]
    \centering
    \includegraphics[width=6cm]{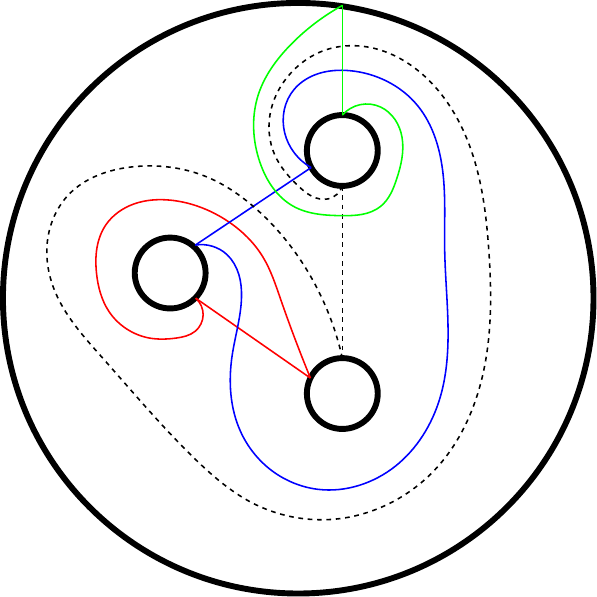}
    \caption[Right-veering basis in an open book that is not right-veering]{A basis of right-veering arcs, and a (dotted) left-veering arc.}
    \label{fig:RVBasis}
\end{figure}

Let $(\Sigma, \varphi)$ be an open book. Given an arc collection $\Gamma$, by an innermost disk argument we can always isotope $\varphi(\Gamma)$ so that or every $\alpha, \beta \in \Gamma$, $\alpha$ does not form any bigons with $\varphi (\beta)$. Therefore we will always assume this is the case, and thus that arcs and arc images in our arc collections intersect minimally. 

We now recall some definitions and notation from \cite{DetectingTightness}.

\begin{definition}
Let $(\Sigma, \varphi)$ be an open book, $\Gamma = \{\alpha_i\}_{i = 1}^n$ an arc collection, and $\varphi(\Gamma) = \{\varphi(\alpha_i) \}_{i=1}^n$ its image under the mapping class $\varphi$ (we orient $\varphi(\alpha_i)$ with the opposite orientation to the one induced by $\alpha_i$). A \emph{region $R$ in $(\Sigma, \varphi,\Gamma)$} (or \emph{supported in $(\Sigma, \varphi,\Gamma)$}) is the image of an immersed $2k$-gon such that:

\begin{itemize}
    \item The edges are mapped to $\Gamma$ and $\varphi (\Gamma)$ alternatively.
    \item The orientations of the arcs and their images orient $\partial R$.
    \item Every corner is acute, i.e for every vertex $ x = \alpha_i \cap \varphi(\alpha_j)$, in a neighbourhood of $x$, $R$ only intersects one of the $4$ quadrants defined by $\alpha_i$ and $\varphi(\alpha_j)$ at $x$.
    \item The immersion restricted to the vertices of the $2k$-gon is injective.
\end{itemize}

A point $\alpha_i \cap \varphi(\alpha_j)$ is \emph{positive} if the tangent vectors of $\alpha_i$ and $\varphi(\alpha_j)$ (in that order) determine the orientation of $\Sigma$ at the intersection point, and \emph{negative} otherwise. Moreover, we will say a region is \emph{positive} if the boundary orientation given by the orientation of the arcs from $\Gamma$ and $\varphi(\Gamma)$ coincides with the usual counterclockwise orientation, and \emph{negative} otherwise.

We will denote positive intersection points by $\bullet$-points and negative intersection points by $\circ$-points. See Figure \ref{fig:Region} for an example of a region with its positive and negative points labelled. We will also denote the set of $\bullet$-points (respectively $\circ$-points) of a region $A$ by Dot($A$) (respectively Circ($A$)), and the set of vertices $\textrm{Dot}(A) \cup \textrm{Circ}(A)$ as $\textrm{V}(A)$. 
\end{definition}

 \begin{figure}[htp]
    \centering
    \includegraphics[width=5cm]{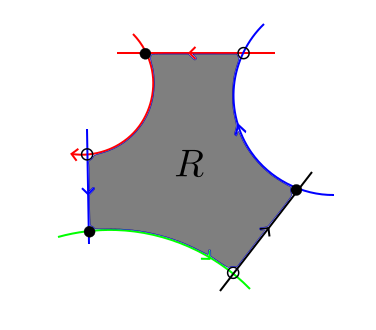}
    \caption[A region.]{A region $R$, which is positive because the arcs orient $\partial R$ counterclockwise. We will use the convention that straight lines represent arcs and curved lines represent their images under $\varphi$.}
    \label{fig:Region}
\end{figure}

\begin{definition}
Let $R$ be a  region in $(\Sigma, \varphi,\Gamma)$. We say that $R$ is \emph{completed} if there exists another region $R'$, such that $\textrm{Circ}(R') \subset \textrm{Circ}(R)$, but the induced orientation of $\partial R'$ is the opposite orientation to the one in $R$. We will call this region the \emph{completion} of $R$.
If no such region exists we say $R$ is \emph{not completed}.
\end{definition}

\begin{remark}
A completion of a region need not use all of the arcs used in the region.
\end{remark}

 \begin{figure}[htp]
    \centering
    \includegraphics[width=7cm]{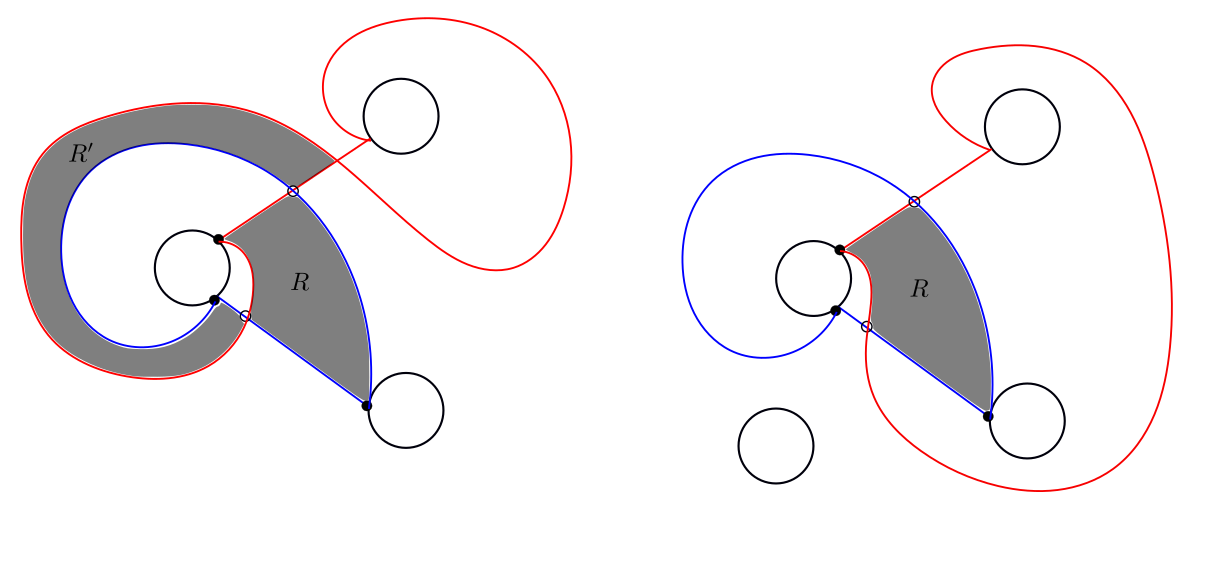}
    \caption[Completion of a region.]{On the left, a region $R$ and its completion $R'$. On the right, a region $R$ which is not completed.}
    \label{fig:Completion}
\end{figure}

Now our setup differs slightly from the standard version of consistency as defined in \cite{LegendrianSurgery} and \cite{DetectingTightness}. The reason for this is that, while Honda, Kazez, and Mati\'{c}'s result establishes a relationship between the right-veering property and tightness, the two concepts are not equivalent, since for any overtwisted contact manifold we can find a supporting open book decomposition that is right-veering \cite[Proposition 6.1]{Right-veering}. The technology in \cite{LegendrianSurgery} and \cite{DetectingTightness} aims to detect tightness, while we aim to detect the right-veering property, and so it is natural that there will be similarities as well as differences. In particular, we pay special attention to \emph{basepoint triangles}, defined below, and in Section \ref{Towers} the difference between our setup and the one in \cite{DetectingTightness} will become more apparent.

\begin{definition}
Let $(\Sigma, \varphi)$ be an open book, $\Gamma$ an arc collection in $\Sigma$, and $\alpha_1, \alpha_2 \in \Gamma$. Assume there exists a boundary component $B$ of $\Sigma$ which contains an endpoint of $\alpha_1$ and an endpoint of $\alpha_2$. Then assume that $\varphi(\alpha_1)$ is boundary parallel near $B$ until it intersects $\alpha_2$, and does not intersect any other $\alpha \in \Gamma$ before doing so. This creates a triangle (with sides an arc segment of $\varphi(\alpha_1)$, an arc segment of $\alpha_2$ and an arc segment of $B$). We can see such a triangle in Figure \ref{fig:BasepointTriangle}. We will refer to this triangle as a \emph{basepoint triangle}. For an arc collection $\Gamma$, we denote the set of $\circ$-points that are vertices of basepoint triangles by $\textrm{Circ}_{\partial}(\Gamma)$.
\end{definition}

 \begin{figure}[htp]
    \centering
    \includegraphics[width=3cm]{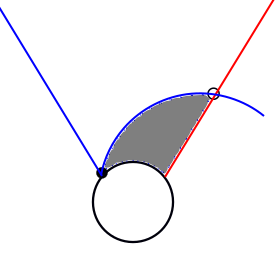}
    \caption[Basepoint triangle.]{A basepoint triangle (shaded).}
    \label{fig:BasepointTriangle}
\end{figure}

\begin{remark}
    If there is an arc image $\varphi(\alpha)$ that intersects a basepoint triangle disjoint from $\alpha$, it must do so forming a bigon. Note that since we assume that our collections are bigon free then this cannot happen. We can generalise this situation for the case where, instead of a basepoint triangle, we have a disc component of $\Sigma \setminus (\Gamma \cup \{ \alpha \})$ with a unique edge on an arc $\alpha \in \Gamma$ (a basepoint triangle is the simplest case of such a disc, with an edge on an arc, and the other two on the boundary and an arc image respectively).

\end{remark}

\begin{definition}
    Let $\Gamma$ be an arc collection in an open book $(\Sigma, \varphi)$,  $\{\varphi(\beta_i)\}_{i = 1}^n$ a subcollection of $  \varphi(\Gamma)$, and $a$ an arc segment of an arc $\alpha \in \Gamma$ such that $\{\varphi(\beta_i)\}_{i = 1}^n \cup \{a\}$ cut out a disc $D$ from $\Sigma$. Then we say $a$ is \emph{restricted in $\Gamma$}.
\end{definition}

Similarly as in the previous remark, for a restricted edge $a$ and associated disc $D$, images of arcs disjoint from $D$ cannot intersect the restricted edge because they would then have to form a bigon. Clearly the edge of a basepoint triangle that lies on an arc is restricted, with the disc $D$ being the basepoint triangle itself.

Now we want to show that, in the case where the arc-slide of a pair of arcs --with the opposite orientation-- is left-veering, regions with an edge on a basepoint triangle can detect this left-veering arc. We will see later that this is a simple example of an extended tower, and it will be the base case of our induction. 

This means that we want to understand what possibilities there are for regions when we look at three arcs $\{ \alpha_0, \alpha_1, \alpha_2 \}$ that cut out a $6$-gon from the surface. Call this $6$-gon $P$. Then $\varphi(P)$ must also be a $6$-gon. Now consider an arc collection $\mathcal{C}$ (which may also include some of the arcs cutting out $P$). When segments of two of the arc images $\varphi(\alpha_i), \varphi(\alpha_j)$ form opposite sides of a rectangle which is a connected component of $\varphi(P) \setminus \mathcal{C}$, we will say that they are \emph{parallel (with respect to $\mathcal{C}$)} along those segments.

\begin{proposition} \label{InitialLV}
Let $\alpha_0, \alpha_1, \alpha_2$ be properly embedded arcs that cut out a $6$-gon $P$ from $\Sigma$, oriented counterclockwise, and assume $\alpha_1$ and $\alpha_2$ are right-veering. Then $\alpha_0$ is left-veering if and only if there exists a positive region $R$ in $\{\alpha_1, \alpha_2\}$ contained in $P$, with $\bullet$-points on $\partial \Sigma$ and where one of the edges is the edge of a basepoint triangle, that has no completion.
\end{proposition}

\begin{proof}
First assume that $\alpha_0$ is left-veering, which means (since $\alpha_2$ is right-veering), that it leaves $P$ by intersecting $\alpha_1$ in a point $z$. Since $\alpha_2$ is right-veering, $\varphi(\alpha_2)$ must leave $P$ by intersecting $\alpha_1$ in a point $y$. This in turn means that $\varphi(\alpha_1)$ must leave $P$ by intersecting $\alpha_2$. Indeed, since $\alpha_1$ and $\alpha_2$ are disjoint, $\varphi(\alpha_1)$ cannot intersect $\varphi(\alpha_2)$. This creates the region $R$, with an edge being an edge of the basepoint triangle formed by $\alpha_2$ and $\varphi(\alpha_1)$. For a contradiction, suppose that this region can be completed with a region $R'$, which must necessarily be a rectangle (it cannot be a bigon because we are assuming our collections are bigon free, and it cannot have more than 4 vertices because $R$ only has $2$ $\circ$-points). Moreover, the edge of this rectangle on $\alpha_2$ is restricted, because it is an edge of a basepoint triangle. This in turn means that the edge of the rectangle on $\alpha_1$ is restricted. Then the $\bullet$-point of $R'$ on $\alpha_1$ cannot be between $z$ and $y$, because then $\varphi(\alpha_1)$ would have to form a bigon. This means that it would have to be between $z$ and the other endpoint of $\alpha_1$ (that is, $\alpha_1(0)$). But this means that $\varphi(\alpha_0)$ intersects the restricted edge --a contradiction. 

Conversely, assume that there exists a region $R$ satisfying the above conditions,  in particular, it has no completion. Suppose for a contradiction that $\alpha_0$ is right-veering. But then, since the image of $P$ must be a disc, there must be an arc segment on $\alpha_1$ cutting out a disc with $\varphi(\alpha_1)$ and $\varphi(\alpha_2)$, so it must be a restricted edge with respect to $\{ \alpha_1, \alpha_2 \}$. However this in turn means that $R$ has a completion --a contradiction. So $\alpha_0$ must be left-veering, see Figure \ref{fig:TriangleLV}. For a more concrete example, we can see in Figure \ref{fig:RVBasis} the positive region that has no completion because the dotted arc is left-veering.

 \begin{figure}[htp]
    \centering
    \includegraphics[width=5cm]{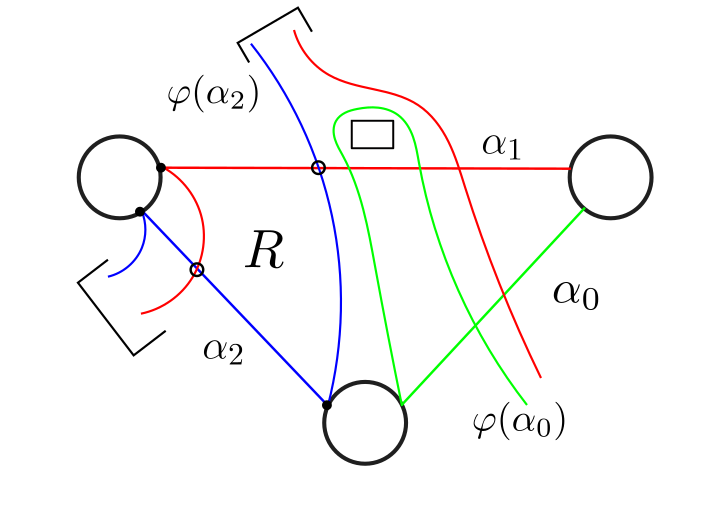}
    \caption[Incomplete region from left-veering arc.]{The incomplete region $R$ when $\alpha_0$ is left-veering.}
    \label{fig:TriangleLV}
\end{figure}
\end{proof}

For completeness, we include the case where all three arcs cutting out a $6$-gon are right-veering.

\begin{proposition} \label{RV}
Let $\alpha_0, \alpha_1, \alpha_2$ be properly embedded arcs cutting out a $6$-gon $P$, oriented counterclockwise, and assume they are right-veering. Then we can find a non-empty collection of regions in $\{ \alpha_0, \alpha_1, \alpha_2 \}$ such that every positive region is completed by a negative region, and every interior $\bullet$-point is a vertex of two regions (one positive and one negative). 
\end{proposition}

\begin{proof}
Suppose first that the image of each arc leaves $P$ by forming a basepoint triangle. Then we have an initial positive region $R$ (a $6$-gon) where the $\bullet$-points are endpoints of the arcs and the $\circ$-points are the intersection points where the arc images leave $P$. To see that this is completed, observe that the image of $P$ is again a $6$-gon which is the union of a completing region and the three basepoint triangles, and so we must have a negative region with the same $\circ$-points as $R$ and whose $\bullet$-points are the other endpoints of the arcs, see the top of Figure \ref{fig:TriangleCases}. 

Now suppose that one arc image (we can assume it is $\varphi(\alpha_2)$) leaves $P$ without forming a basepoint triangle (so in this case, by intersecting $\alpha_1$). This forces the image of $\alpha_1$ to leave $P$ by intersecting $\alpha_2$ and creating a basepoint triangle. This immediately gives a positive region $R_1$. This region has a completion $R'_1$, because if it did not, $\alpha_0$ would have to be left-veering by Proposition \ref{InitialLV}. One of the $\bullet$-points is a vertex of the basepoint triangle formed by $\varphi(\alpha_1)$ and the other one is either an endpoint of $\alpha_1$ or an interior point $\alpha_1 \cap \varphi(\alpha_1)$, let us denote it by $x$. In the first case we have that $\alpha_0$ is isotopic to $\varphi(\alpha_0)$, as it cannot be left-veering because it would intersect a restricted edge on $\alpha_1$ but it also cannot be strictly right-veering because $\alpha_0$ with the opposite orientation would also have to be strictly right-veering and it would have to intersect the restricted edge on $\alpha_1$, see the bottom left of Figure \ref{fig:TriangleCases}. 

For the second case we have that, after $x$, $\varphi(\alpha_1)$ intersects $\alpha_0$ and then it must be parallel to $\varphi(\alpha_0)$ until their other endpoint (because the image of $P$ must be a disc). This gives another two regions, a positive one $R_2$ where the $\bullet$-points are $x$ and an endpoint of $\alpha_0$, and its completion $R'_2$, the part of $\varphi(P)$ where $\varphi(\alpha_0)$ and $\varphi(\alpha_1)$ are parallel, see the bottom right of Figure \ref{fig:TriangleCases}. 

\begin{figure}[htp]
    \centering
    \includegraphics[width=10cm]{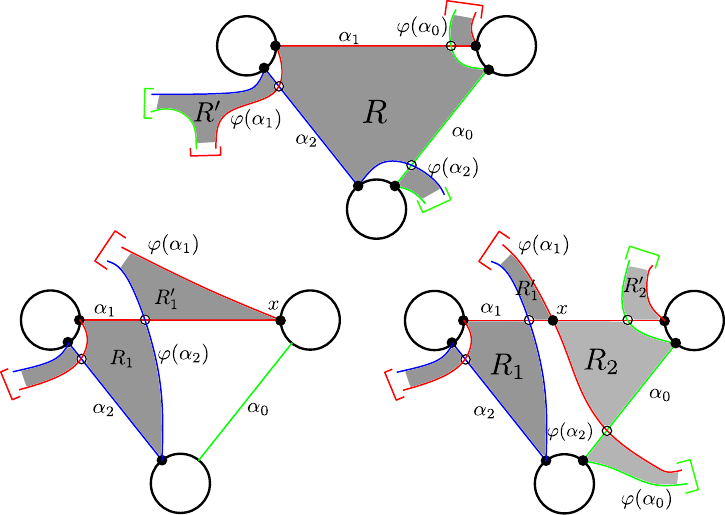}
    \caption[Splitting pair from fixable arc.]{On the top, the positive $6$-gon $R$ and its completion $R'$. On the bottom left, $\alpha_0$ must be isotopic to its image because the image of $P$ must be a disk. On the bottom left, we have two positive regions $R_1$ and $R_2$, and two negative regions $R'_1$ and $R'_2$, with every interior $\bullet$-point belonging to two regions.}
    \label{fig:TriangleCases}
\end{figure}

\end{proof}

We have seen what regions arise in a $6$-gon when all arcs are right-veering, and when one of the arcs is left-veering. Eventually we want to detect a left-veering arc by dividing it into segments that can be fixed by the monodromy and a segment that is left-veering. Thus, now we turn our attention to arc segments that can be fixed by the monodromy.

\begin{definition} \label{ContainedArcDef}
Let $(\Sigma, \varphi)$ be an open book. If two arcs  $\alpha_1, \alpha_2$ that cut out a $6$-gon $P$ with a third arc $\alpha_0$ (oriented counterclockwise) support a pair of regions $\{R_1, R'_1\}$ as in the second case of Proposition \ref{RV} (i.e there is a positive region $R_1$ in $P$ where one of the sides is a side of the basepoint triangle, and it is completed by a region $R_1'$ with a $\bullet$-point in the interior of $\alpha_1$) we say that $\alpha_2$ is \emph{$\varphi$-contained} in $\alpha_1$ and we call $\{ R_1, R'_1 \}$ a \emph{positive splitting pair}. We also call $\{R_2,R'_2\}$ a \emph{negative splitting pair}, and we also say that $\alpha_0$ is \emph{$\varphi$-contained} in $\alpha_1$.
\end{definition}

\begin{definition} \label{Fixable}
Let $\Gamma$ be an arc collection in an open book $(\Sigma, \varphi)$, and $x, y \in \Gamma \cap \varphi(\Gamma)$ be two points (which could be on the boundary or interior points) that are fixed by some representative of $\varphi$ . Let $\gamma$ be an arc segment starting in  $x$ and ending in $y$. We say $\gamma $ is \emph{fixable by $\varphi$} if there is a representative of $\varphi$ that fixes it relative to $\partial \gamma$.
\end{definition}

This means that $\gamma $ and $\varphi(\gamma)$ bound a collection of bigons that intersect only on points $\gamma \cap \varphi(\gamma)$, see Figure \ref{fig:FixableArc} for an example.

\begin{remark}
    If both endpoints of an arc segment are fixed by $\varphi$, properties like being right- or left-veering can be defined as for properly embedded arcs. 
\end{remark}

 \begin{figure}[htp]
    \centering
    \includegraphics[width=7cm]{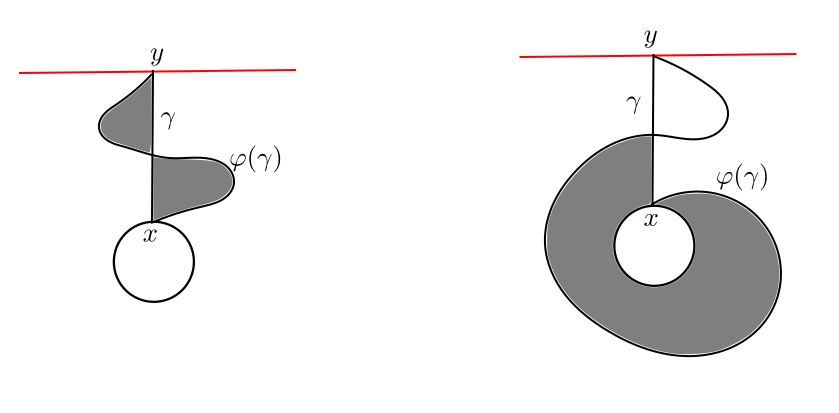}
    \caption[Fixable arcs.]{On the left, $\gamma $ is fixable because we can isotope $\varphi(\gamma)$ relative to its endpoints to coincide with $\gamma$. On the right, $\gamma$ is not fixable.} 
    \label{fig:FixableArc}
\end{figure}

In practice, we want our fixed points to come from intersections $\alpha \cap \varphi(\alpha)$ where $\alpha$ is an arc of a chosen basis.

Similarly to Proposition \ref{InitialLV}, if we have $3$ arcs cutting out a $6$-gon, we can use regions to detect a fixable arc inside this $6$-gon.

\begin{proposition} \label{InitialFixed}
Let $(\Sigma, \varphi)$ be an open book, and let $\alpha_0, \alpha_1, \alpha_2$ be properly embedded arcs cutting out a $6$-gon $P$ from $\Sigma$, oriented counterclockwise, and assume they are strictly right-veering. Let $\gamma$ be an arc segment contained in $P$ starting on $\partial \Sigma$ between $\alpha_2$ and $\alpha_0$ and ending in an intersection point $ x= \alpha_1 \cap \varphi(\alpha_1)$ in the interior of $\alpha_1$. Then $\gamma$ is fixable by $\varphi$ if and only if $\{\alpha_1, \alpha_2\}$ support a positive splitting pair $\{R,R'\}$ such that the $\bullet$-point of $R'$ in the interior of $\Sigma $ is $x$.
\end{proposition}

\begin{proof}
First assume that $\gamma$ is fixable. Then the image of $\alpha_2$ must leave $P$ by intersecting $\alpha_1$, which means that the image of $\alpha_1$ leaves $P$ by intersecting $\alpha_2$ (and forming a basepoint triangle), giving the positive region $R$. Since we are assuming that $\alpha_0$ is strictly right-veering, this region must be completed by a region $R'$, giving the positive splitting pair. 

\begin{figure}[htp]
    \centering
    \includegraphics[width=5cm]{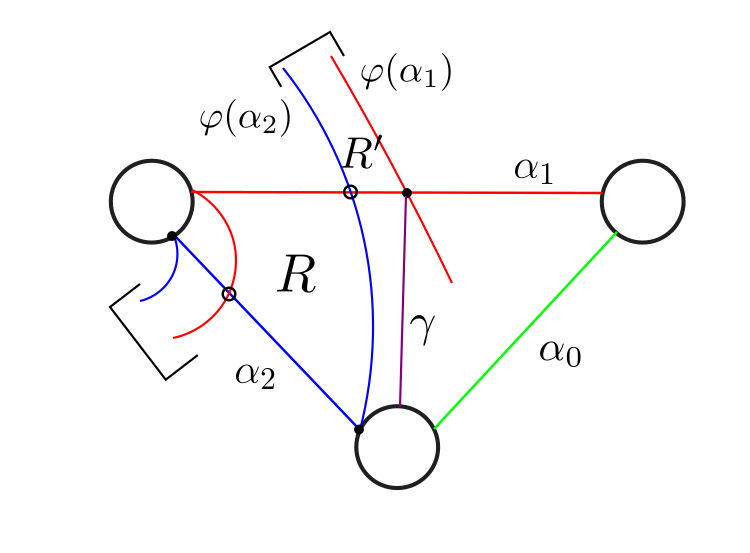}
    \caption[Splitting pair from fixable arc.]{The regions $R$ and $R'$ when $\gamma$ is fixable.}
    \label{fig:TriangleFixed}
\end{figure}

Conversely, suppose that we have a positive splitting pair $\{R, R'\}$. Then $\gamma$ cannot be left-veering, because it would have to intersect the edge of $R'$ on $\alpha_1$, which is restricted. However, by Proposition \ref{RV}, after the splitting pair (that is, after the interior $\bullet$-point of $R'$) $\varphi(\alpha_1)$ and $\varphi(\alpha_0)$ must be parallel up to the boundary, which means that they form a rectangle with $\alpha_0$ and an edge of a basepoint triangle on $\alpha_1$. Thus the edge of this rectangle on $\alpha_0$ is also restricted. If $\gamma$ were strictly right-veering in its starting point, it would have to intersect this restricted edge. Thus $\gamma$ must be fixable by $\varphi$ (strictly speaking, we might not have that $\varphi(x) = x$, however, in this case, by the same argument, $\gamma, \varphi(\gamma)$, and $\varphi(\alpha_1)$ must bound a disc, and thus we may isotope $\varphi(x)$ by sliding it along $\varphi(\alpha_1)$ to coincide with $x$ and then $\gamma $ is fixable). 

 \begin{figure}[htp]
    \centering
    \includegraphics[width=5cm]{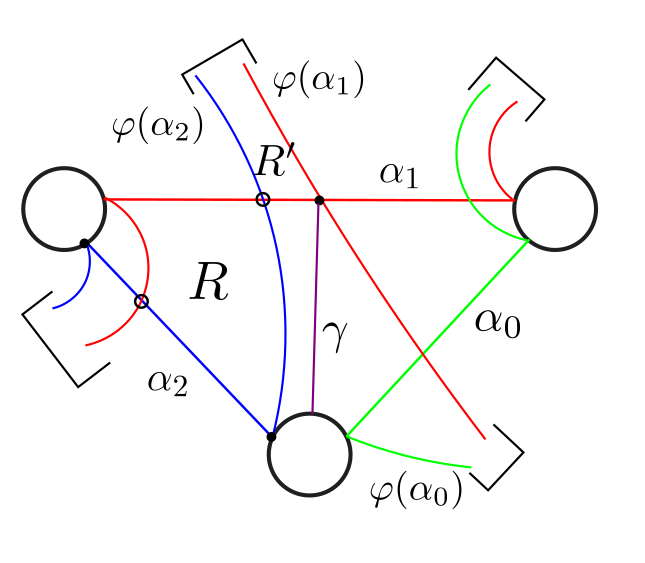}
    \caption[Fixable arc from splitting pair.]{Once we have $R$ and $R'$, the image of $\gamma$ cannot go to either side because the edges on $\alpha_0$ and $\alpha_1$ are restricted.}
    \label{fig:Split6-gon}
\end{figure}
\end{proof}

\section{Extended towers} \label{Towers}
We now introduce extended towers, which will be our main tool for detecting left-veering arcs. They are inspired by the notion of towers introduced in \cite{DetectingTightness}. However, there are some differences. Towers were defined to detect tightness, and we aim to detect left-veering arcs. 
Both concepts are related but not equivalent, and the new features of extended towers reflect this. Indeed, incomplete towers in the sense of \cite{DetectingTightness} are meant to show that the supported contact structure is overtwisted, while incomplete extended towers in our sense are meant to show that there exists a left-veering arc. Of course, by \cite[Theorem 1.1]{Right-veering} this implies overtwistedness. However, because the converse is not true \cite[Proposition 6.1]{Right-veering}, and there are right-veering open books supporting overtwisted contact structures, this means we can find open books where we can find incomplete towers but not incomplete extended towers, as we will show in Example~\ref{ex:GlobalExample}. In particular, our notions of being \emph{replete} and \emph{completed} will differ from those in \cite{DetectingTightness}, and we introduce the notion of \emph{niceness}.

From now on, let $\Gamma$ be an arc collection in a surface with boundary $\Sigma$, and $\alpha_0$ a properly embedded arc, disjoint from $\Gamma$, such that $\Gamma \cup \{ \alpha_0 \}$ cuts out a disc $P$. Moreover, orient the arcs in the standard counterclockwise orientation of $\partial (P)$.

\begin{definition}
For a given arc collection $\Gamma$ in an open book $(\Sigma, \varphi)$, we denote the set of regions supported in $\Gamma$ by $\mathcal{R}(\Sigma, \varphi, \Gamma)$. Moreover, for any collection of regions $\mathcal{A}$, the set of positive regions in $\mathcal{A}$ is denoted by $\mathcal{A}^+$, and, similarly, the set of negative regions of $\mathcal{A}$ is denoted by $\mathcal{A}^-$.
\end{definition}

\begin{definition}
An \emph{extended tower} in $(\Sigma, \varphi, \Gamma)$ is a (nonempty) collection $\mathcal{T} \subset \mathcal{R}(\Sigma, \varphi, \Gamma)$ where $\textrm{Dot}(\mathcal{T}) \subset (\textrm{Dot}(\mathcal{T}^-) \cup \partial \Sigma)$, $\textrm{Circ}(\mathcal{T}) \subset (\textrm{Circ}(\mathcal{T}^+)  \cup \textrm{Circ}_{\partial}(\Gamma))$, and for all pairs $A,B \in \mathcal{T}$, no corner of $A$ is contained in the interior of $B$. We say that $\Gamma$ \emph{supports} $\mathcal{T}$.
\end{definition}

 We can see an example of two extended towers in Figure \ref{fig:RunningExample}. These will be our primordial examples, and while we will provide other, simpler examples for some of the properties later on, we will often come back to these ones to illustrate why all the properties are needed. We will see later that the existence of the extended tower on the left implies that the arc segment $\gamma$ is fixable, while the existence of the extended tower on the right implies that the arc $\alpha_5$ is left-veering.

    \begin{figure} [htp]
        \centering
        \includegraphics[width=12cm]{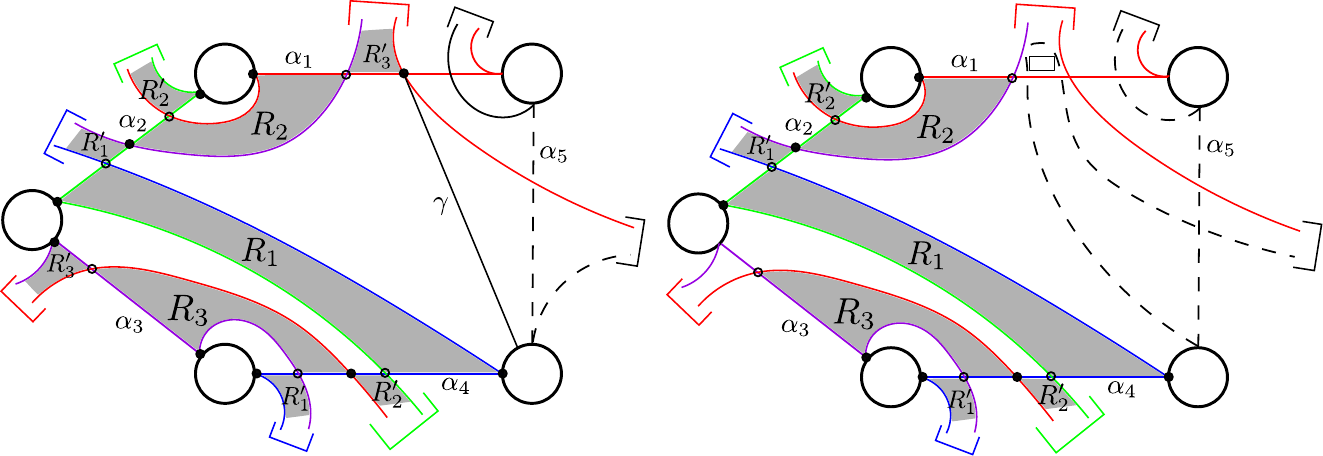}
        \caption[Primordial example.]{Two extended towers which will be our primordial examples to showcase the different properties that we will require. The black box on the right hand side represents that there is no bigon that can be used to isotope $\varphi(\alpha_5)$ so that it is right-veering (it also means that there is no region using the $\circ$-points of $R_2$ and $R_3$).}
        \label{fig:RunningExample}
    \end{figure}

\begin{definition} \label{Replete}
An extended tower $\mathcal{T}$ in $(\Sigma, \varphi, \Gamma)$ is \emph{replete} if whenever there is a region $A \in \mathcal{R}^-(\Sigma, \varphi, \Gamma)$ which satisfies $\textrm{Circ}(A) \subset \textrm{Circ}(\mathcal{T}) \cup \textrm{Circ}_{\partial}(\Gamma)$, and $\mathcal{T} \cup A$ is again an extended tower, then $A \in \mathcal{T}$. All extended towers will be assumed replete unless otherwise stated.
\end{definition}

The main difference with towers from \cite{DetectingTightness} is that here we allow negative regions with $\circ$-points in $\textrm{Circ}_{\partial}(\Gamma)$ as well as $\textrm{Circ} (\mathcal{T}^+)$. This is because we want to detect left-veering arcs rather than tightness, and we can see this difference in the following example. 

\begin{example} \label{ex:GlobalExample}

Let $(\Sigma, \varphi)$ be the open book specified by Figure \ref{fig:GlobalExample}, with $\Gamma$ the two straight arcs. This open supports an overtwisted contact structure \cite{Lisca}, and indeed this is detected by the fact that $\mathcal{T} = \{ R \}$ forms an incomplete tower in the sense of \cite{DetectingTightness}. However, as an extended tower, $\mathcal{T}$ is not replete, since $\mathcal{T} \cup \{ R'\}$ is again an extended tower, and $\textrm{Circ}(R') \subset \textrm{Circ}(\mathcal{T}) \cup \textrm{Circ}_{\partial}(\Gamma)$. Indeed, there can be no incomplete extended tower since this open book is right-veering \cite{Lisca}.
 \begin{figure}[htp]
    \centering
    \includegraphics[width=5cm]{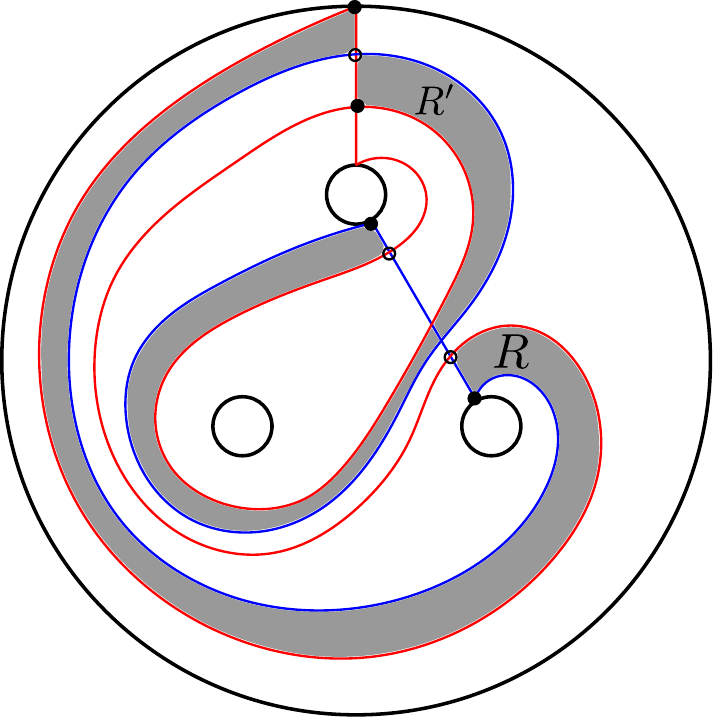}
    \caption[Difference between towers and extended towers.]{The region $R$ forms an incomplete tower but is not a replete extended tower because we can add the region $R'$.}
    \label{fig:GlobalExample}
\end{figure}

\end{example}

The importance of this definition can now be seen in the extended towers from Figure \ref{fig:RunningExample}. We wanted the extended tower $\{ R_1, R'_1, R_2, R'_2, R_3 \}$ on the right side of the figure to correspond to the arc $\alpha_5$ being left-veering. However, if we do not impose the repleteness condition, we can also consider the extended tower $\{ R_1, R'_1, R_2, R'_2, R_3 \}$ on the left side of the figure, and there the arc $\alpha_5$ is not left-veering (if it were, it would have to intersect the edges on $R'_3$ or $R'_1$ which are both restricted). Therefore we need to impose the condition that we have to add the region $R'_3$ to this extended tower if possible.

We now impose some further restrictions on our extended towers so that they completely characterise fixable and left-veering arcs.

\begin{definition} \label{Nice}
An extended tower $\mathcal{T}$ in $(\Sigma, \varphi, \Gamma)$, with $\Gamma$ an arc collection as above, is \emph{nice} if for every region $A \in \mathcal{T}^+$ we have $A \subset P$ and for every region $B \in \mathcal{T}^-$ we have that $int(B)$ is disjoint from $\varphi(\Gamma)$. We will assume all extended towers are nice.
\end{definition}

\begin{remark}
    This is not a notion that appears in the original tower technology. We want to consider only nice extended towers because extended towers that are not nice come from collections where there are no left-veering arcs and no fixable arcs. See for example the extended tower in Figure~\ref{fig:NotNice}, it is not nice but conforms to the definition of completed extended tower that we will see next (which is the one we want to identify with fixable arcs, and indeed there is no fixable arc in Figure~\ref{fig:NotNice}). The extended towers in Figure \ref{fig:RunningExample} on the other hand, are indeed nice, and they do correspond to a fixable arc segment and a left-veering arc.
\end{remark}

 \begin{figure}[htp]
    \centering
    \includegraphics[width=5cm]{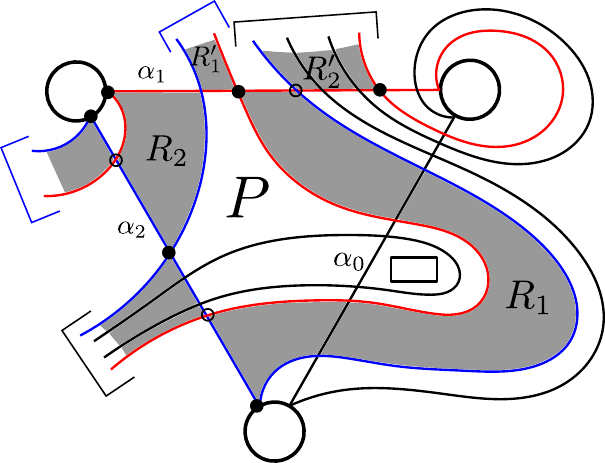}
    \caption[A non-nice extended tower.]{The extended tower $\mathcal{T} = \{ R_1, R_2, R'_1, R'_2 \}$ supported in $\alpha_1, \alpha_2$ is not nice because $R_1$ is a positive region that is not contained in $P$.}
    \label{fig:NotNice}
\end{figure}

\begin{definition} \label{Nested}
    Let $\Gamma$ be an arc collection in an open book $(\Sigma, \varphi)$ and $\mathcal{T}$ an extended tower in $\Gamma$. We say $\mathcal{T}$ is \emph{nested} if there exist nested subcollections of regions as follows.

    \begin{itemize}
        \item $\mathcal{T}^+_0 = \{ R \in \mathcal{T}^+  \mid \textrm{Dot}(R) \subset \partial \Sigma \}$.

        \item $\mathcal{T}^-_0 = \{ R' \in \mathcal{T}^-  \mid \textrm{Circ}(R') \subset \mathcal{T}^+_0 \cup \textrm{Circ}_{\partial}(\Gamma) \}$.
        
        \item $\mathcal{T}^+_i = \{ R \in \mathcal{T}^+  \mid \textrm{Dot}(R) \subset \mathcal{T}^-_{i-1} \cup \partial \Sigma \}$.

        \item $\mathcal{T}^-_i = \{ R' \in \mathcal{T}^-  \mid \textrm{Circ}(R) \subset \mathcal{T}^+_{i} \cup \textrm{Circ}_{\partial}(\Gamma)  \}$.

        \item $\mathcal{T}^+ = \bigcup_{i} \mathcal{T}^+_i $ and $\mathcal{T}^- = \bigcup_{i} \mathcal{T}^-_i $
    \end{itemize}

    We will assume all extended towers are nested, and we will refer to regions in $\mathcal{T}^{\pm}_i$ as \emph{being in level i}.
\end{definition}

\begin{remark}
    For the arc collections $\Gamma$ as above, a necessary condition for an extended tower to be nested is the existence of a level $0$ positive region. Indeed, without a level $0$ positive region the only possibility for a level $0$ negative region would be one where all $\circ$-points are on basepoint triangles. However, this implies that the arcs supporting this negative region cut out a disc, contradicting the conditions we required for $\Gamma$. 
    
    Therefore, we can see that the extended tower from Figure \ref{fig:NotNice} is not nested as there is no level $0$ positive region. Observe that in terms of Heegaard Floer homology, a level $0$ positive region corresponds to a differential to the contact class. We can also see an example of a nice extended tower that is not nested in Figure \ref{fig:WeirdCompletedTower}.

    However, the extended towers in Figure \ref{fig:RunningExample} are indeed nested. There is a unique level $0$ positive region which is $R_1$, and a unique level $0$ negative region $R'_1$, because its $\circ$-points are on $R_1$ and on a basepoint triangle. Then there is a unique level $1$ positive region $R_2$ because the $\bullet$-point that is not on the boundary belongs to $R'_1$, and a unique level $1$ negative region $R'_2$ because its $\circ$-points belong to $R_1$ and $R_2$. Finally, there is a unique level $2$ positive region $R_3$ because the $\bullet$-point that is not on the boundary belongs to $R'_2$. On the right hand side there are no more regions, but on the left hand side there is a level $2$ negative region $R'_3$ because its $\circ$-points belong to $R_2$ and $R_3$.
\end{remark}

  \begin{figure} [htp]
     \centering
     \includegraphics[width=5cm]{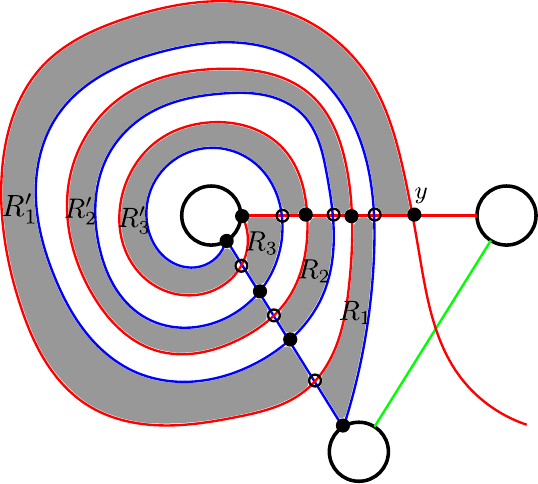}
     \caption[Non-nested extended tower.]{An extended tower that is not nested, because there is no level $0$ positive region.}
     \label{fig:WeirdCompletedTower}
 \end{figure}

\begin{definition}

Let $\mathcal{T}$ be an extended tower in $(\Sigma,\varphi,\Gamma)$, and let $x$ be an interior point of some $\alpha \in \Gamma$ that is a vertex of a region in $\mathcal{T}$. We say that $x$ is \emph{two-sided} if it is a vertex of exactly two regions of $\mathcal{T}$ (one positive and one negative).

\end{definition}

\begin{definition} \label{Completed}
Let $\mathcal{T}$ be an extended tower in $(\Sigma,\varphi, \Gamma = \{ \alpha_i \}_{i = 1} ^ n)$, where $\Gamma \cup \{\alpha_0\}$ cuts out a disc $P$ for some properly embedded arc $\alpha_0$ disjoint from $\Gamma$, and the arcs are oriented labelled counterclockwise. We say that $\mathcal{T}$ is \emph{completed} if every interior vertex of $\mathcal{T}$ is two-sided, with the exception of a single $\bullet$-point $y_0 \in \alpha_1 \cap \varphi(\alpha_1)$, which we call a \emph{connecting vertex}.
\end{definition}

\begin{figure}[htp]
    \centering
    \includegraphics[height=6cm]{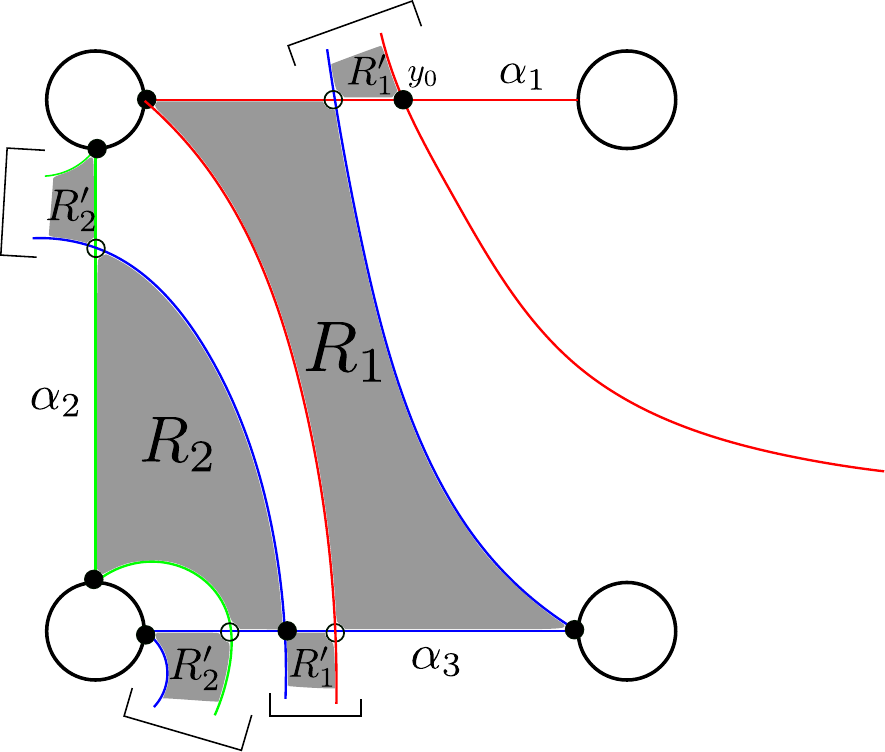}
    \caption[Completed extended tower.]{A simple example of a completed extended tower, which is also nice and replete, where the only interior vertex that is not two-sided is the $\bullet$-point $y_0 \in \alpha_1 \cap \varphi(\alpha_1)$.}
    \label{fig:CompletedTower}
\end{figure}


 This is a much more restricted notion than that of completed tower in \cite{DetectingTightness}. The reason for this is we want completed extended towers to correspond exactly to fixable arc segments, and to determine them uniquely. It is also clear that completed extended towers are those where every point of every arc $\alpha_2, \dots, \alpha_n$ (as long as they are all strictly right-veering) belongs to a region, and every point on $\alpha_1$ from  $y_0$ to $\alpha_1(1)$ also belongs to a region.

 Once again, if we return to our primordial example in Figure \ref{fig:RunningExample}, we can see that the extended tower on the left is completed. Indeed, every interior vertex is two-sided with the exception of a vertex $\alpha_1 \cap \varphi(\alpha_1)$, and the arc segment $\gamma$ is fixable. However, the extended tower on the right is not completed, as there are interior $\circ$-points on $R_2$ and $R_3$ that are not two-sided. In this case there is no fixable arc segment (and indeed the arc $\alpha_5$ is left-veering).

\begin{definition} \label{Incomplete}
Let $\mathcal{T}$ be an extended tower in $\Gamma$. We say that $\mathcal{T}$ is \emph{incomplete} if

for every negative region $A \in \mathcal{T}^-$, there exists a vertex $x \in \textrm{Dot}(A)$ that is two-sided, i.e. there exists a positive region $B \in \mathcal{T}^+$ such that $x \in \textrm{Dot}(B)$. 
\end{definition}

\begin{figure}[htp]
    \centering
    \includegraphics[height=5cm]{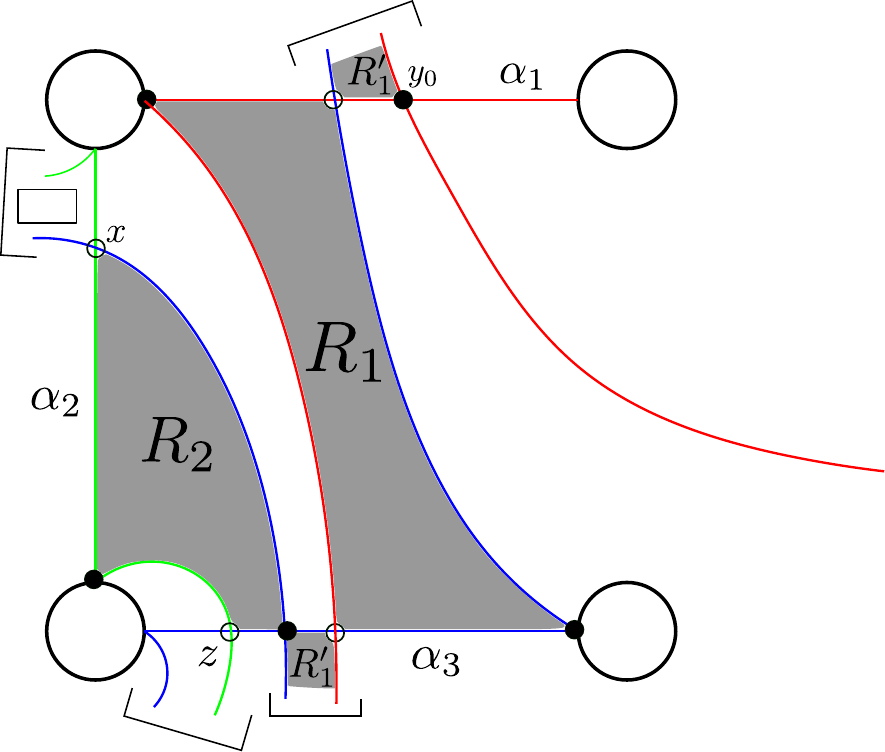}
    \caption[Incomplete extended tower.]{A simple example of an incomplete extended tower, because there is only one negative region $R'_1$, which shares a $\bullet$-point with $R_2$, and the $\circ$-points $x$ and $z$ are not two-sided.}
    \label{fig:IncompleteTower}
\end{figure}

Once again returning to our primordial example in Figure \ref{fig:RunningExample}, the extended tower on the right is incomplete, because every negative region has a $\bullet$-point which is two-sided. The one on the left, however, is not incomplete since the region $R'_3$ is a negative region without any two-sided $\bullet$-points.

This mirrors the definition of incomplete tower from \cite{DetectingTightness}, because the property it aims to detect, a left-veering arc, implies overtwistedness. The additional conditions imposed to extended towers are what distinguishes this from the definition of incomplete tower. We could see this in Example \ref{ex:GlobalExample}, where the change in the definition of repletedness implied that $\mathcal{T} = \{ R \}$ forms a replete incomplete tower, but not a replete extended tower.

\begin{remark}
    As evidenced by Figure \ref{fig:GlobalExample} in Example \ref{ex:GlobalExample}, we can have extended towers that are neither completed nor incomplete, for instance, if not every interior vertex is two-sided but there exists a negative region with no $\bullet$-points in common with any positive region. An extended tower where every interior $\bullet$-point is two-sided is also neither completed nor incomplete. It will follow from our discussion later that in this case the arc $\alpha_0$ is fixable, but we want to detect fixable arc segments rather than arcs and so we exclude this case from our definition of completed extended tower.
\end{remark}

In Figure \ref{fig:TowerExceptions} we can see two extended towers that are neither completed nor incomplete. On the left, $\mathcal{T}_1 = \{ R, R' \}$ is not incomplete because the negative region $R'$ does not have any $\bullet$-point in common with the unique positive region $R$, but is also not completed because there is a $\bullet$-point that is not two-sided and is not of the form $\alpha \cap \varphi(\alpha)$. On the right, $\mathcal{T}_2 = \{ R_1, R'_1, R_2, R'_2 \}$ is not completed nor incomplete because every interior vertex is two-sided, and $R'_2$ does not share a $\bullet$-point with a positive region.

\begin{figure}[htp]
    \centering
    \includegraphics[height=4cm]{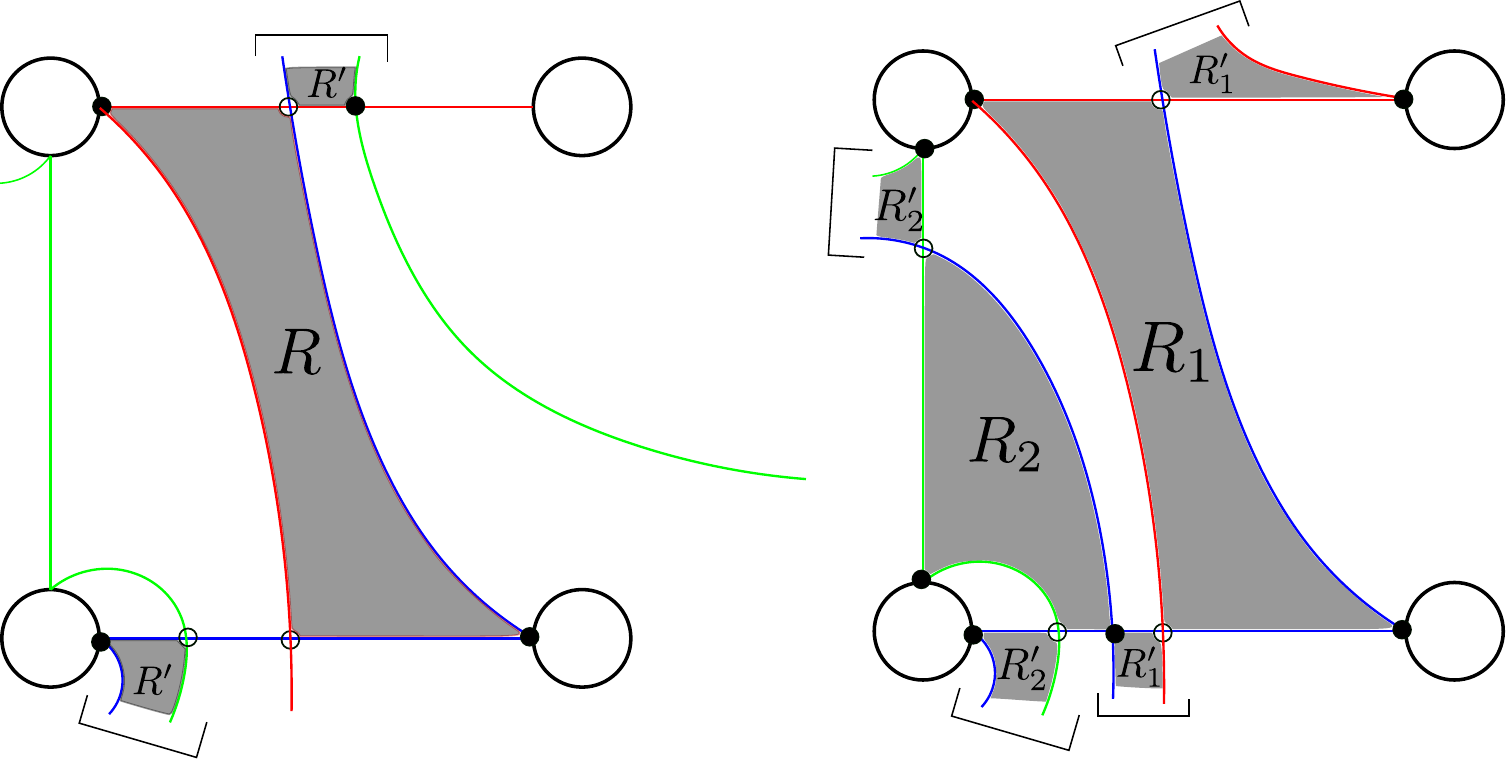}
    \caption[Non-completed, non-incomplete extended towers]{Two examples of extended towers which are neither completed nor incomplete.}
    \label{fig:TowerExceptions}
\end{figure}

\section{Results} \label{Results}

In this section we prove that we can detect a left-veering arc using extended towers. First we show existence of a special type of a left-veering arc that will be easier to detect.

\subsection{Minimal left-veering arcs} \label{Minimal Arcs}

\begin{definition} \label{ShortenedArc}

Let $(\Sigma, \varphi)$ be an open book, and let $\mathcal{B}$ be a basis for $\Sigma$. Let $\gamma$ be a properly embedded arc, which we may assume intersects the basis. This divides $\gamma$ into a collection of arc segments $\gamma_1, \dots , \gamma_n$, labelled and oriented following the orientation of $\gamma$, which intersect $\mathcal{B}$ only on their endpoints. We say $\gamma$ is \emph{shortened with respect to $\mathcal{B}$} if $\gamma_1, \dots, \gamma_{n-1}$ are fixable. If $\gamma$ is left-veering, we call it a \emph{shortened left-veering arc with respect to $\mathcal{B}$}.

\end{definition}

Note that a shortened arc $\gamma$ as in Definition \ref{ShortenedArc} is left-veering if and only if $\gamma_n$ is left-veering. Indeed, since $\gamma_1, \dots, \gamma_{n-1}$ are fixable, we can take a representative of the isotopy class of $\gamma$ so that it is pointwise fixed in $\gamma_1, \dots, \gamma_{n-1}$, and then the veeringness of $\gamma$ is completely determined by the veeringness of $\gamma_n$.

\begin{lemma} \label{ShortenedLV}
Let $(\Sigma, \varphi)$ be an open book, and let $\mathcal{B}$ be a basis for $\Sigma$. Suppose there exists a left-veering arc $\gamma$ in $(\Sigma, \varphi)$. Then there exists a shortened left-veering arc $\gamma'$ with respect to $\mathcal{B}$.

\end{lemma}

\begin{proof}

We may assume $\varphi(\gamma)$ is bigon free with respect to $\mathcal{B} \cup {\gamma}$. We define the arc $\gamma'$ as follows. Let $x \in \gamma \cap \mathcal{B}$ be the first intersection point with the basis such that, after $x$, $\gamma $ and $\varphi(\gamma)$ exit the disc cut out by the basis by intersecting different arcs $\alpha_1$ and $\alpha_2$ respectively, and let the intersection of $\gamma$ with $\alpha_1$ be $y$. Then take the arc $\gamma '$ that is the same as $\gamma$ up to $x$ and ends in the starting point of $\alpha_2$ without having any more intersections with $\mathcal{B}$ or the arc segment of $\gamma$ up to $y$. Clearly $\gamma'$ is shortened with respect to $\mathcal{B}$. To show that it is left-veering, take the arc $\gamma'$ and isotope it slightly so that it lies to the left of $\gamma$. Then its image is fixable and to the left of the image of $\gamma$ up to $x$. Suppose for a contradiction that $\gamma'$ is right-veering. Then $\varphi(\gamma')$ must intersect $\varphi(\gamma)$ so that $\varphi(\gamma)$, $\varphi(\gamma')$, and $\partial \Sigma$ bound a disc. Moreover, this would have to be the image of a disc bounded by $\gamma$, $\gamma'$, and $\partial \Sigma$. However, this gives a contradiction because $\gamma$ and $\gamma' $ are (by construction) disjoint before $\gamma$ intersects $\alpha_1$ so they cannot bound a disc, as we can see in Figure \ref{fig:MinimalArc}.

 \begin{figure} [htp]
    \centering
    \includegraphics[height=5cm]{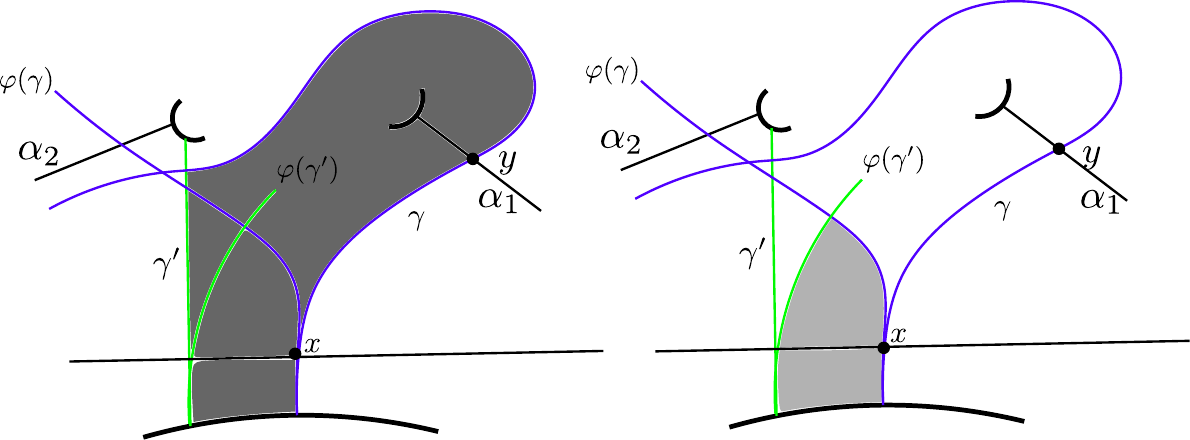}
    \caption[Shortened left-veering arc.]{If $\gamma'$ were right-veering, the image of the darkly shaded subsurface would have to be the lightly shaded one, a contradiction. Note that in this case $\gamma $ is an arc and not an arc image even though it is not represented with a straight line.}
    \label{fig:MinimalArc}
\end{figure}

\end{proof}

\begin{definition} \label{Minimal}
Define a length of an arc $\gamma$ with respect to a basis $\mathcal{B}$ by the (unsigned) number of intersections of $\gamma$ with all the arcs of $\mathcal{B}$. Out of all shortened left-veering arcs, we call one which minimises this length a \emph{minimal left-veering arc} with respect to $\mathcal{B}$.
\end{definition}

Note that a minimal left-veering arc minimises this distance for all left-veering arcs and not just shortened ones.

We will be interested primarily in minimal left-veering arcs. Moreover, when we have a $2n$-gon $P$ cut out by $\{\alpha_i\}_{i=0}^n$ with $\alpha_0$ left-veering, we can assume there are no left-veering arcs contained in $P$, because otherwise we can find a subset of the arcs cutting out a $2m$-gon (for $m<n$) with a left-veering arc, and we can work with this subset instead. In particular, we can assume that the image of $\alpha_0$ leaves the $2n$-gon by intersecting $\alpha_1$. This is because if it leaves $P$ by intersecting some other arc, say $\alpha_k$, then the arc $\beta_0$ that is obtained by consecutive arcslides of $\alpha_0$ over $\alpha_1, \dots, \alpha_{k-1}$ must be left-veering, by the same argument we made in Lemma \ref{ShortenedLV}, and then we can focus on the $2(n-k)$-gon cut out by $\alpha_{k}, \alpha_{k+1}, \dots, \alpha_n$, and $\beta_0$, which will have the property we want.

\subsection{Base Case} \label{Base Case}

In this subsection we show using methods from Section \ref{Preliminaries} that extended towers detect left-veering arcs and fixable arc segments in $6$-gons. This will be the base case of our induction.

\begin{proposition} \label{TriangleLV}
Let $(\Sigma, \varphi)$ be an open book. Let $\alpha_0, \alpha_1, \alpha_2$ be properly embedded arcs cutting out a $6$-gon $P$, oriented counterclockwise, and assume $\alpha_1$ and $\alpha_2$ are right-veering. Then $\alpha_0$ is left-veering if and only if $\{\alpha_1, \alpha_2 \}$ support an incomplete extended tower that is nested, nice, and replete.

\end{proposition}

\begin{proof}
    First suppose that $\alpha_0$ is left-veering. Then Proposition \ref{InitialLV} gives a (positive) region $R$ whose interior is disjoint from $\{ \alpha_1, \alpha_2 \}$. Since again by Proposition \ref{InitialLV} there are no negative regions, $\mathcal{T} = \{ R \}$ forms an incomplete extended tower that moreover is replete and nice, and $R$ is on level zero since the $\bullet$-points of $R$ are on the boundary, so $\mathcal{T}$ is nested. \medskip

    Now suppose that there exists an incomplete extended tower $\mathcal{T}$ supported in $\{\alpha_1, \alpha_2\}$. We want to show that $\mathcal{T} = \{ R \}$, with $R$ the region from Proposition \ref{InitialLV}, which shows that $\alpha_0$ is left-veering. There must be a region in $\mathcal{T}^+_0$, and for $\mathcal{T}$ to be nice it must be the region $R$ from Proposition \ref{InitialLV}. Now, if there exists a negative region $R'$ in $\mathcal{T}^-_0$, then again by Proposition \ref{InitialLV} every point on $\alpha_2$ belongs to a region, so there can be no more positive regions. But now $R'$ does not have any $\bullet$-points in common with a positive region, so $\mathcal{T}$ is not incomplete --a contradiction. So there does not exist such a negative region, and thus $\mathcal{T} = \{ R \}$, and then by Proposition \ref{InitialLV} $\alpha_0 $ is left-veering.

 \end{proof}

\begin{proposition} \label{TriangleFixed}
Let $(\Sigma, \varphi)$ be an open book. Let $\alpha_1, \alpha_2, \alpha_0$ be properly embedded strictly right-veering arcs cutting out a $6$-gon $P$, oriented counterclockwise. Let $\gamma$ be an arc segment contained in $P$ starting between $\alpha_2$ and $\alpha_0$ and ending in the interior of $\alpha_1$. Then $\gamma$ is fixable by $\varphi$ if and only if $\{\alpha_1, \alpha_2\}$ support a completed extended tower that is nested, nice and replete, and whose connecting vertex coincides with $\gamma \cap \alpha_1$.
\end{proposition}

\begin{proof}

First suppose that $\gamma$ is fixable. Then we have the regions $R$ and $R'$ from Proposition \ref{InitialFixed} forming the splitting pair, and we can see that they form a completed extended tower which is nested, replete and nice, and the unique connecting vertex is $\gamma \cap \alpha_1$.

Conversely, suppose that there exists a completed extended tower $\mathcal{T}$ which is nice and replete. By the same reasoning as in Proposition \ref{TriangleLV}, the positive region $R$ from Proposition \ref{InitialFixed} must be in $\mathcal{T}$. Since $\mathcal{T}$ is completed, there must be a negative region $R'$ with the same $\circ$-points as $R$. As one of the $\circ$-points in this negative region is on the basepoint triangle on $\alpha_2$, one of the $\bullet$-points of $R'$ is on the boundary (because there can be no other intersection points between the $\circ$-point on the basepoint triangle and the boundary as our arc collections are bigon free). Now every point of $\alpha_2$ belongs to a region, so there can be no more regions in $\mathcal{T}$. This means that for $\mathcal{T}$ to be completed the other $\bullet$-point of $R'$ must be an interior point $\alpha_1 \cap \varphi(\alpha_1)$, which means that the regions in $\mathcal{T}$ are the regions from Proposition \ref{InitialFixed} (i.e the splitting pair), so $\gamma$ is fixable.

\end{proof}

\subsection{Inductive Step} \label{Inductive Step}

We now have that a left-veering arc is detected by an incomplete extended tower if it cuts out a $6$-gon (with the correct orientation) with two arcs from the basis. We want to extend this by induction to the case where the left-veering arc cuts out an $n$-gon with arcs from the basis. Similarly, we have that completed towers detect fixable arc segments when the arc segment is contained in a $6$-gon cut out by three arcs, two of which are from our basis, and we want to extend to the case where the arc segment is contained in a $n$-gon, with $n-1$ arcs in our basis. To show this, let us first introduce some notation to be used throughout this subsection. \medskip

Let $\alpha_0$, $\alpha_1$, and $\alpha_2$ be properly embedded arcs that cut out a $6$-gon $P$, where the arcs are labelled counterclockwise. Also let $\Gamma$ be an arc collection such that there exists an arc $\beta$ disjoint from $\Gamma$ with $ \mathcal{C} = \Gamma \cup \{ \beta \}$ cutting out a disc $P'$ with disjoint interior with $P$, and $\alpha_0 \in \Gamma$. Again, we will orient this arc collection with the counterclockwise orientation. Moreover, assume that $\alpha_0$ is not the first arc in $\Gamma$ (i.e the next one to $\beta$ as we go counterclockwise through the boundary of the disc cut out by $\mathcal{C}$). This is because we want to detect fixable arcs with an endpoint on the first arc of the collection $\Gamma$, so we do not want this point to change. Let $\Gamma' = (\Gamma \setminus \{ \alpha_0 \} ) \cup \{ \alpha_1, \alpha_2 \}$. Then $\mathcal{C}' = \Gamma' \cup \{ \beta \}$ cuts out a disc $ P \cup P'$. Orient the arcs in this collection again counterclockwise (this agrees with the previous orientation). \medskip

The idea is that, given an extended tower $\mathcal{T}$ in $\Gamma$, we can slide its regions over $\alpha_0$ to $\alpha_1$ and $\alpha_2$ to obtain an extended tower $\mathcal{T}'$ supported in $\Gamma'$, that will have the same properties as $\mathcal{T}$. We can see a simple example with completed extended towers in Figure \ref{fig:SlideExample}.

 \begin{figure}[htp]
    \centering
    \includegraphics[height = 4cm]{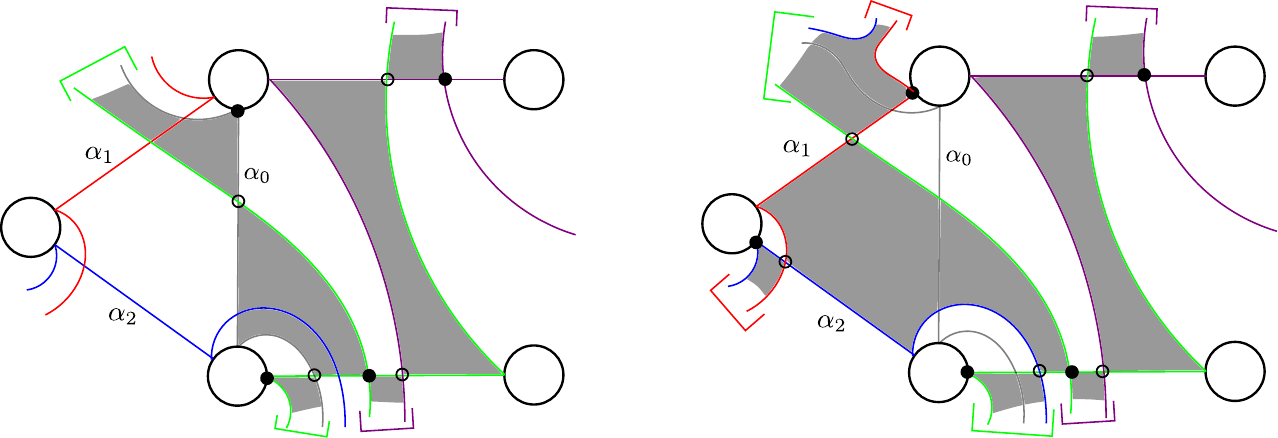}
    \caption[Sliding regions.]{The regions for the new extended tower are obtained by sliding the regions from the old extended tower}
    \label{fig:SlideExample}
\end{figure}

To make this operation more precise we now define two maps, which we will call \emph{slide maps} and denote by $s^{\pm}$. Given an extended tower $\mathcal{T}$ supported in $\Gamma$, where $\Gamma$ is an arc collection as above, these maps will send the set of vertices of the positive (respectively negative) regions of $\mathcal{T}$ (denoted by $V(\mathcal{T}^{\pm})$) to intersection points $ \Gamma' \cap \varphi(\Gamma')$. These points will determine regions that form the extended tower $\mathcal{T}'$ supported in $\Gamma'$ that will have the same properties as $\mathcal{T}$. Note that most points will be two-sided and thus will be both in $\mathcal{T}^+$ and $\mathcal{T}^-$, which means that they will have an image under $s^+$ and an image under $s^-$. Most of the time these will agree. In fact, the only time they will not agree will be when a component of $P \cap \varphi(P)$ is a $6$-gon.

\begin{definition} \label{SlidePositive}
    Let $x \in \textrm{V}(\mathcal{T}^+)$. Then $s^+(x)$ is defined as follows.
    \begin{enumerate}
    
        \item If $x $ does not lie on $\alpha_0$ or $\varphi(\alpha_0)$, $s^+(x) = x$.
        
        \item If $x$ lies on the intersection of $\varphi(\alpha_0)$ with some other arc $\beta$ from $\Gamma$, then $s^+(x)$ is the intersection point $y$ of $\beta$ with $\varphi(\alpha_1)$ or $\varphi(\alpha_2)$ such that the segment between $x$ and $y$ is contained in $\beta \cap \varphi(P)$.
        
        \item If $x$ lies on the intersection of $\alpha_0$ with the image of some other arc $\beta$ from $\Gamma$, then $s^+(x)$ is the intersection point $y$ of $\varphi(\beta)$ with $\alpha_1$ or $\alpha_2$ such that the segment between $x$ and $y$ is contained in $\varphi(\beta) \cap P$.
        
        \item If $x$ lies on the intersection of $\alpha_0$ with its image, then $s^+(x)$ is the intersection point $y$ of $\alpha_m$ with $\varphi(\alpha_l)$ (where $m$ and $l$ can be $1$ or $2$ and not necessarily equal), obtained by first going along $\alpha_0$ to $\varphi(\alpha_l)$, and then along $\varphi(\alpha_l)$ to $y$, such that this path is contained in $P \cap \varphi(P)$.
        
    \end{enumerate}
    We illustrate the different cases in Figure \ref{fig:SlideCases}, where we can see that, while the definition may seem arbitrary, for an intersection point $x \in \Gamma \cap \varphi(\Gamma)$ we are essentially choosing ``the closest point'' to $x$ that belongs to $\Gamma' \cap \varphi(\Gamma')$.
\end{definition}

\begin{proposition} \label{WellDefined}
    The map $s^+$ is well defined.
\end{proposition}  

\begin{proof}
   Case 1 is immediate. Case 2 is well defined because if $\varphi(\alpha_0)$ intersects an arc then either $\varphi(\alpha_1)$ or $\varphi(\alpha_2)$ must also intersect that arc because $\varphi(P)$ is a disc. Moreover, there is a unique segment from $x$ to $y$ contained in $\beta \cap \varphi(P)$. Case 3 is the same as Case 2 but with the roles of the arcs and arc images reversed. Finally, in Case 4, the same argument as for Case 2 shows that either $\varphi(\alpha_1)$ or $\alpha_2$ intersect $\alpha_0$, and there is a unique segment between $x$ and $\varphi(\alpha_m)$ contained in $P \cap \varphi(P)$. Then, $\varphi(\alpha_m)$ must exit $P$ by intersecting either $\alpha_1$ or $\alpha_2$, and going along $\varphi(\alpha_m)$ gives $y$.
\end{proof}

 \begin{figure}[htp]
    \centering
    \includegraphics[width=6cm]{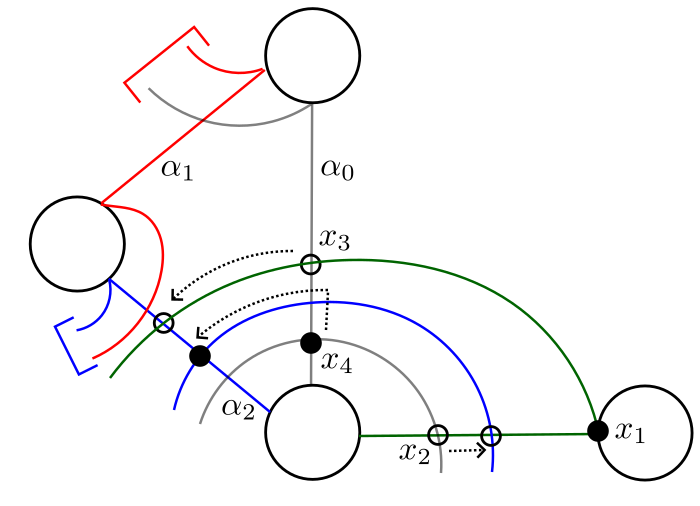}
    \caption[Cases for $s^+$.]{The different cases for $s^+$, where the point $x_i$ is of case $i$ in Definition \ref{SlidePositive}. The dashed arrows indicate the action of $s^+$.}
    \label{fig:SlideCases}
\end{figure}

Observe that for the first two cases, both $x$ and $s^+(x)$ belong to an arc that is not in $P$, and for the last two cases $x$ belongs to $\alpha_0$ and $s^+(x)$ belongs to either $\alpha_1$ or $\alpha_2$. We now define $s^-$ similarly.

\begin{definition}

Let $x \in \textrm{V}(\mathcal{T}^-)$. Then $s^-(x)$ is defined as:

\begin{enumerate}
        \item  If $x $ does not lie on $\alpha_0$ or $\varphi(\alpha_0)$, $s^-(x) = x$.
        
        \item If $x$ lies on the intersection of $\alpha_0$ with the image of some other arc $\beta$ from $\Gamma$, then $s^-(x)$ is the intersection point $z$ of $\varphi(\beta)$ with $\alpha_1$ or $\alpha_2$ such that the segment between $x$ and $z$ is contained in $P$.

        \item If $x$ lies on the intersection of $\varphi(\alpha_0)$ with some other arc $\beta$ from $\Gamma$, then $s^-(x)$ is the intersection point $z$ of $\beta$ with $\varphi(\alpha_1)$ or $\varphi(\alpha_2)$ such that the segment between $x$ and $z$ is contained in $\beta \cap \varphi(P)$.

        \item If $x$ lies on the intersection of $\alpha_0$ with its image, then $s^-(x)$ is the intersection point $z$ of $\alpha_m$ with $\varphi(\alpha_l)$ (where $m$ and $l$ can be $1$ or $2$ and not necessarily equal), obtained by first going along $\varphi(\alpha_0)$ to $\alpha_m$, and then along $\alpha_m$ to $z$, such that this path is contained in $P \cap \varphi(P)$.
        
    \end{enumerate}

\end{definition}

Observe that if we reverse the roles of the arcs and arc images, that is, we take our arc collections to be $\varphi(\Gamma)$ and $\varphi(\Gamma')$, and their images to be $\varphi^{-1}(\varphi(\Gamma))$ and $\varphi^{-1}(\varphi(\Gamma'))$, and we also reverse their orientation (so that the negative regions become positive regions), then the definition of $s^-$ is the same as the definition of $s^+$ using the original arc collections. This also means, by Proposition \ref{WellDefined}, that $s^-$ is well defined. \medskip

Also note that, away from $P \cup \varphi(P)$, the slide map does not change the intersection point. Moreover, if $x$ is a positive (respectively negative) intersection point then $s^+(x)$ and $s^-(x)$ will also be positive (respectively negative). \medskip

Finally, the slide maps are injective, so they give a bijection onto their image, and then we can refer to the inverse of these maps. We will use this to show that extended towers in $\Gamma'$ also induce extended towers in $\Gamma$.

Now we want to show that using the maps $s^+$ and $s^-$ we can construct an extended tower $\mathcal{T}'$ supported in $\Gamma'$ that has the same properties as $\mathcal{T}$. The properties that will be preserved will be being  nested, being completed/incomplete (or neither), being replete, and being nice. We will focus on the local effect of the slide maps on a neighbourhood of $P$, and a neighbourhood of $\varphi(P)$ (because outside these neighbourhoods the slide maps do not change anything). In particular, this means that if a positive (respectively negative) region $R$ is supported in $\Gamma$ and the image of each of its vertices under $s^+$ (resp. $s^-$) is itself, then $R$ is also supported in $\Gamma'$. If $R$ is not supported in $\Gamma'$, the local effect of the slide map on the vertices will induce one (or more) regions supported in $\Gamma'$. \medskip

We will need to check several things. First, that the induced regions form an extended tower, and then, that the properties of being nested, replete, nice, and completed or incomplete (or neither) are preserved. \medskip

First we will separate two cases, when $\alpha_2$ is $\varphi$-contained in $\alpha_1$ and when $\alpha_0$ is $\varphi$-contained in $\alpha_2$. The reason for this is that in these cases the image under the slide maps of a boundary point is an interior point (a boundary point is never two sided so we need to consider this separately). We will not consider the case where $\alpha_1$ is $\varphi$-contained in $\alpha_0$, because then the slide maps send an interior point to two boundary points (which are never two-sided). However, we do not need this case.

\begin{lemma} \label{Contained1}
    Suppose that $\alpha_2$ is $\varphi$-contained in $\alpha_1$. Then if $\mathcal{T}$ is an extended tower in $\Gamma$, there is an extended tower $\mathcal{T}'$ supported in $\Gamma'$ with the same properties as $\mathcal{T}$. Conversely, if $\mathcal{T}'$ is an extended tower supported in $\Gamma'$, there is an extended tower $\mathcal{T}$ supported in $\Gamma$ with the same properties as $\mathcal{T}$.

\end{lemma}

\begin{proof}
    First let $\mathcal{T}$ be an extended tower in $\Gamma$. We will now see how each region in $\mathcal{T}$ induces a region supported in $\Gamma'$

    First, away from $P$ and $ \varphi(P)$ any region $R \in \mathcal{T}$ is unchanged since it is already supported in $\Gamma'$, and so the region induced by the slide maps is $R$ itself. For any region with an edge on the interior of $\alpha_0$, observe that an arc image intersecting $\alpha_0$ must leave $P$ by intersecting $\alpha_1$, and so the images of any vertex on $\alpha_0$ under the slide maps coincide, and is on $\alpha_1$. Then, we obtain the region $R'$ by simply adding or removing rectangles. Similarly, for a region with an edge on $\varphi(\alpha_0)$, we can see reversing the role of arcs and arc images that any arc intersecting $\varphi(\alpha_0)$ must also intersect $\varphi(\alpha_1)$, and so the image under the slide maps of a vertex on $\varphi(\alpha_0)$ lies on $\varphi(\alpha_1)$, and again the induced region is obtained by simply adding or removing rectangles. We can see this in Figure \ref{fig:ArcSlideEquivalent}. Observe that each region $R_i \in \mathcal{T}$ corresponds to a unique region $R_i'$ in $\Gamma'$ (sometimes $R_i = R_i'$), so we define $\mathcal{T}' = \{ R_i' \mid R_i \in \mathcal{T} \} \cup \{ R_1, R_2 \} $, where $\{ R_1, R_2 \}$ is the splitting pair of $\{ \alpha_1, \alpha_2 \} $.

    \begin{figure}[htp]
    \centering
    \includegraphics[height=5cm]{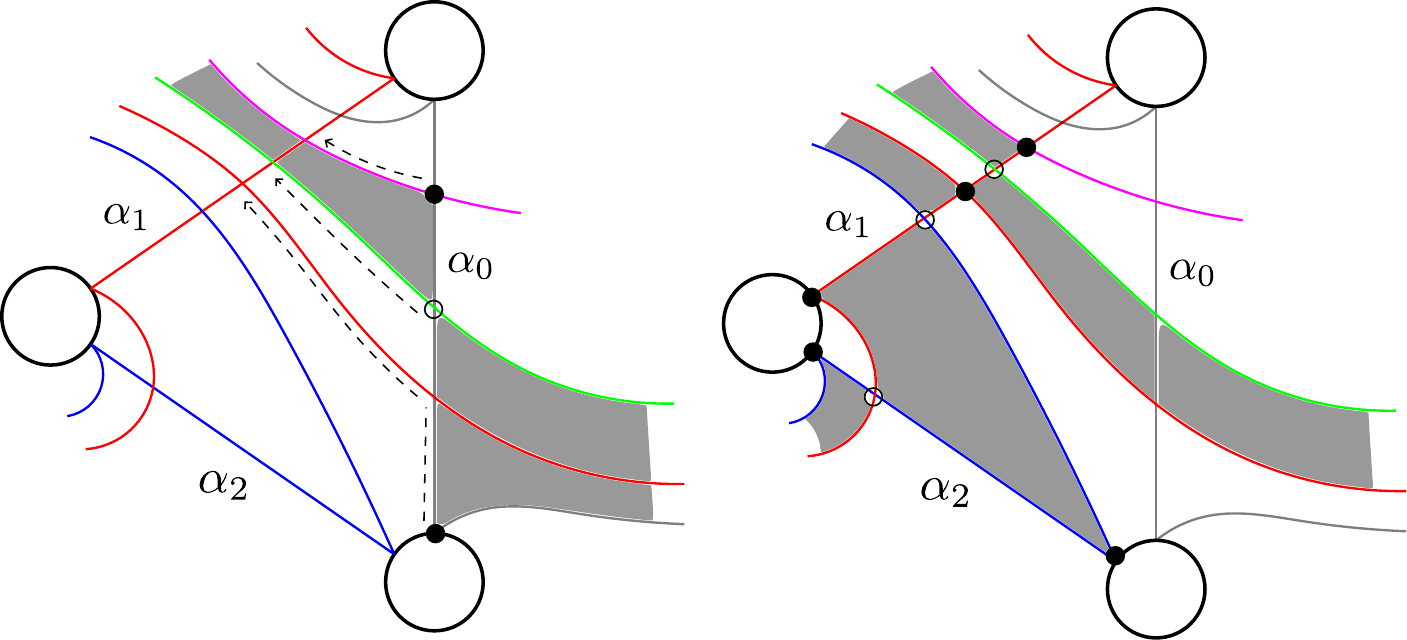}
    \caption[Arc-slide equivalent regions]{The regions obtained by adding or removing rectangles, when $\alpha_2 $ is $\varphi$-contained in $\alpha_1$, and the splitting pair. The dashed arrows indicate the action of the slide maps.}
    \label{fig:ArcSlideEquivalent}
\end{figure}

    Let us now check the properties of $\mathcal{T}'$. The positive regions are by construction contained in $P$ and thus disjoint from $\Gamma'$ in their interior, and we can see by switching the role of arcs and arc images that the interior of the negative regions is disjoint from $\varphi(\Gamma')$, so the extended tower is nice. It is also replete, because we cannot add any negative regions with $\circ$-points in $\textrm{Circ}(\mathcal{T}') \cup \textrm{Circ}_{\partial}(\Gamma)$, since $\mathcal{T}$ is replete and we have already used the $\circ$-point in the basepoint triangle of $\alpha_1$ and $\alpha_2$ for the splitting pair of regions.

    The positive region from the splitting pair is a level $0$ region, and so is the negative region from the splitting pair. This will imply that some of the regions may go up one level (for example, if there was a level $0$ region with a $\bullet$-point on $\alpha_0(1)$, the induced region is now a level $1$ region because it as an interior $\bullet$-point that is also a vertex of a level $0$ negative region), but if $\mathcal{T}$ was nested then so is $\mathcal{T}'$, as we can construct it level by level from the levels of $\mathcal{T}$.

    Finally, being completed, incomplete, or neither comes from which vertices are two-sided. For interior points that are the image of an interior point $x$, they will be two-sided if and only if $x$ is two-sided. So we only need to check the boundary point $y = \alpha_0(1)$ (recall $\alpha_0$ is given its orientation from being an arc of $\Gamma$) because its image $z = s^+(y)$ is an interior point (the image of the other endpoint $\alpha_0(0)$ is not an interior point so we do not need to check if it is two-sided). But $z$ is two-sided because we have included the positive splitting pair of regions in $\{ \alpha_1, \alpha_2 \}$ in $\mathcal{T}'$. Thus if $\mathcal{T}$ is completed then $\mathcal{T}'$ is completed, and if $\mathcal{T}$ is incomplete then so is $\mathcal{T}'$.

    For the converse, notice that the only points in an extended tower that do not have an inverse image are in the splitting pair of $\{\alpha_1, \alpha_2 \}$. Now, if we assume that $\mathcal{T}'$ is not supported in just $\{ \alpha_1, \alpha_2 \}$, then $ z= s^+(\alpha_0(1))$ is an interior $\bullet$-point that must be two-sided, so there must be a region $R'$ that is not supported in $\{ \alpha_1, \alpha_2 \}$. But then we get a region $R$ supported in $\Gamma$. Moreover, note that every point in $\alpha_2$ belongs to a region, and so every other region in $\mathcal{T}'$ has vertices that have an inverse image on $\Gamma$. Moreover, every vertex is two-sided if and only if the preimage is. Notice that we do get an extended tower by taking the regions induced by these points because the preimage of $z$ is on the boundary, so even though the negative region which has $z$ as a vertex does not induce a negative region in $\mathcal{T}$, it does not have to, precisely because $z$ is on the boundary.
\end{proof}

\begin{lemma} \label{Contained2}
    Suppose that $\alpha_0$ is $\varphi$-contained in $\alpha_2$. Then if $\mathcal{T}$ is an extended tower in $\Gamma$, there is an extended tower $\mathcal{T}'$ supported in $\Gamma'$ with the same properties as $\mathcal{T}$. Conversely, if $\mathcal{T}'$ is an extended tower supported in $\Gamma'$, there is an extended tower $\mathcal{T}$ supported in $\Gamma$ with the same properties as $\mathcal{T}$.

\end{lemma}

\begin{proof}
    The argument is the same as for Lemma \ref{Contained1}, except now the arc images intersecting $\alpha_0$ all intersect $\alpha_2$, and the arcs intersecting $\varphi(\alpha_0)$ all intersect $\varphi(\alpha_2)$. Now we need to check the boundary point $y = \alpha_0(0)$, which is the one that is sent to an interior point. But now we the regions that we include in $\mathcal{T}'$ are the negative splitting pair of $\{\alpha_1, \alpha_2 \}$ (if the $\bullet$-point $\alpha_0(0)$ belongs to a region) so its image is two-sided, and it does not affect the properties of the extended tower.  Moreover, the regions from the splitting pair now belong to higher levels (specifically, if the positive region containing $\alpha_0(0)$ is on level $i$, then the regions from the splitting pair are on level $i+1$. Indeed, the positive region has a $\bullet$-point on the boundary and another one on a level $i$ region so it is on level $i+1$, and the negative region has a $\circ$-point on $\textrm{Circ}_{\partial}(\Gamma')$ and another one on a level $i+1$ positive region). The rest of the regions are on the same levels as before, so $\mathcal{T}'$ is also nested if $\mathcal{T}$ is nested. \medskip

    For the converse again the argument is the same using the negative splitting pair instead of the positive one.

\end{proof}

Now suppose none of the arcs $\{\alpha_0, \alpha_1, \alpha_2 \}$ is $\varphi$-contained in any of the others, that is, the image of $\alpha_0$ leaves $P$ by intersecting $\alpha_1$ and the image of $\alpha_2$ leaves $P$ by intersecting $\alpha_0$. Again the slide maps are the identity away from a neighbourhood of $P$, and a neighbourhood of $\varphi(P)$. Therefore, we will distinguish three cases, depending on how $\varphi(P)$ intersects $P$. 

As before, $\mathcal{T}'$ will be a collection of regions induced by the regions in $\mathcal{T}$. However, in the previous cases every region in $\mathcal{T}$ corresponded to a unique region supported in $\Gamma'$ that was obtained by adding and removing rectangles. Now, the slide maps might ``break up'' regions into several other regions as vertices connected by an edge in $\alpha_0$, or $\varphi(\alpha_0)$, could be mapped to different arcs, or arc images. In this case the following definition will be useful.

\begin{definition}
    Let $\mathcal{T}$ be an extended tower in $(\Sigma, \varphi, \Gamma)$. Two regions $R_1,R_2 \in \mathcal{T}^{\pm}$ are said to be \emph{connected by a region $R \in \mathcal{T}^{\mp}$} if  there exist points $x \in \textrm{Dot}(R) \cap \textrm{Dot}(R_1)$ and $y \in \textrm{Circ}(R) \cap \textrm{Circ}(R_2)$. We will then refer to the region $R$ as a \emph{connecting region}.
\end{definition}

Connecting regions supported in $ \{ \alpha_1, \alpha_2 \}$ might not have all of their vertices be images of a vertex in $\mathcal{T}$ (so they are not induced by regions in $\mathcal{T}$ the way the regions in the previous Lemmas were), but we will include them in some cases to preserve properties of $\mathcal{T}$. Similarly, we might need to add negative regions (whose $\bullet$-points are not images under the slide maps of vertices in $\mathcal{T}$) to ensure the resulting extended tower is replete.

\begin{lemma} \label{NoIntersection}
    Suppose that $P \cap \varphi(P)$ is just the basepoint triangles. Using the maps $s^+$ and $s^-$ we can construct regions giving an extended tower $\mathcal{T}'$ supported in $\Gamma'$ that has the same properties as $\mathcal{T}$. 
\end{lemma}

\begin{proof}
    Unlike before, given a region $R \in \mathcal{T}$, the image of its vertices under the slide maps might not induce a unique region in $\Gamma'$. We construct the extended tower $\mathcal{T}'$ as follows. For any region $R \in \mathcal{T} ^ {\pm}$, if the set of vertices $s^{\pm}(V(R))$ induces a unique region $R'$ supported in $\Gamma'$, set $R' \in \mathcal{T}'$. So suppose that it does not, that is, points from $s^{\pm}(V(R))$ lie on two different regions $R'_1$ and $R'_2$ connected by a third region $R'_3$. First assume that $R$ is a negative region. Then if there are two-sided $\bullet$-points $x,y \in V(R)$ such that $s^-(x) \in R'_1$ and $s^-(y) \in R'_2$, we set $R'_1, R'_2, R'_3 \in \mathcal{T}'$. If every two sided point in $R$ has its image in $R'_1$, then set $R'_1 \in \mathcal{T}'$ (but not $R'_2$ or $R'_3$). Now assume that $R$ is a positive region. We do the same as before but with $\circ$-points. If there are two-sided $\circ$-points $x,y \in V(R)$ such that $s^+(x) \in R'_1$ and $s^+(y) \in R'_2$, we set $R'_1, R'_2, R'_3 \in \mathcal{T}'$. If every two sided $\circ$-point in $R$ has its image in $R'_1$, then set $R'_1 \in \mathcal{T}'$, but not $R'_2$ or $R'_3$. 
    
    Finally, if after having done this there is a negative region $R'_1$ supported in $\Gamma'$ such that all its $\circ$-points are vertices of positive regions in $\mathcal{T}'$ or $\circ$-points on a basepoint triangle, set $R'_1 \in \mathcal{T}'$. This ensures that $\mathcal{T}'$ is replete. Further, if there exists a positive region $R'_3$ with its $\bullet$-points being on $\alpha_1(1)$ or $\bullet$-points of $R'_1$, set $R'_3 \in \mathcal{T}'$. This last case only happens if there is a $\circ$-point $x \in \alpha_0$ in a positive region $R \in \mathcal{T}$ that is not two-sided, but its image $s^+(x)$, which is a vertex of a positive region $R'$, is two sided after applying this rule. We will see later (Figure \ref{fig:WeirdTwosided}) that this ensures that if $\mathcal{T}$ is incomplete then so is $\mathcal{T'}$.  \medskip

    We now show what the induced regions are.

    If $P \cap \varphi(P)$ is just the basepoint triangles, for every interior intersection point $x \in \mathcal{T}$ on $\alpha_0$ (which will also be on an arc image $\varphi(\beta)$ with $\beta \neq \alpha_0)$), we have that $s^+(x) = s^-(x)$ and moreover the image under the slide map is obtained by going along $\varphi(\beta)$ until it leaves $P$. We can then see that, for a positive region $R$ with vertices $x,y$ on $\alpha_0$, there are two options. 
    
    The first option is that $s^+(x)$ and $s^+(y)$ lie on the same arc, which means that the local effect of the slide map on $R$ is just extending it by a rectangle.

    The second option is that $s^+(x)$ lies on $\alpha_1$ and $s^+(y)$ lies on $\alpha_2$. Then the local effect of the slide map on $R$ is extending it by a $6$-gon, where the extra sides are given by $\alpha_1$ from $s^+(x)$ to its endpoint, an edge of the basepoint triangle, and $\alpha_2$ from the $\circ$-point on the basepoint triangle to $s^+(y)$, see Figure \ref{fig:No_Intersection}. Moreover, observe that this will only happen once. \medskip 

    For a negative region $R$ with vertices $x,y$ on $\alpha_0$, there are again two options.

    The first option is that $s^-(x)$ and $s^-(y)$ lie on the same arc, which means that the local effect of the slide map on $R$ is just removing a rectangle. We can see this, together with the positive regions, in Figure \ref{fig:No_Intersection}.  Observe that in all these first cases the levels are preserved.

    \begin{figure}[htp]
    \centering
    \includegraphics[width=7cm]{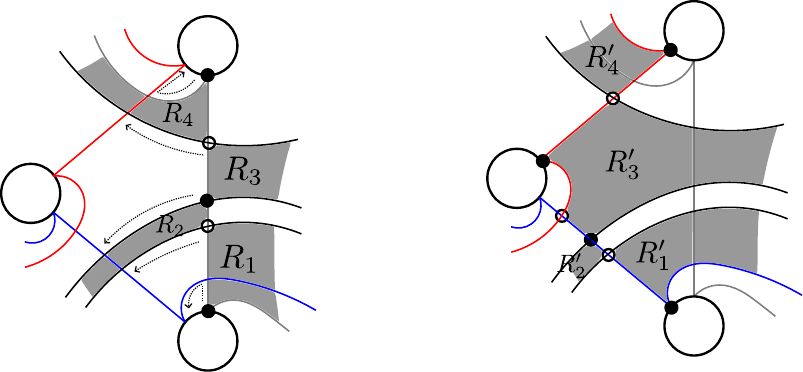}
    \caption[Inductive step for no intersection case.] {Obtaining the new regions by adding or removing rectangles (or a $6$-gon in the case of $R_3$). The dashed arrows indicate the action of $s^{\pm}$. Here we do not a priori know if the $\circ$-point on the basepoint triangle is two-sided, we deal with this case later.}
    \label{fig:No_Intersection}
\end{figure}

     The second option is that $s^-(x)$ lies on $\alpha_1$ and $s^-(y)$ lies on $\alpha_2$. In this case, $R$ must have boundary along another arc image intersecting $\alpha_1$ and $\alpha_2$ so that $R$ is a disc. Moreover, because $s^-(x)$ lies on $\alpha_1$ and $s^-(y)$ lies on $\alpha_2$, we do not have a unique induced region now, but two regions $R_1'$ and $R_2'$, one with an edge on $\alpha_2$ between $s^{-}(y)$ and a $\bullet$-point $b$ and another one with an edge on $\alpha_1$ between a $\circ$-point $a$ and $s^{-}(x)$, connected by a positive rectangle $R_3'$ with vertices $a,b$, the $\circ$-point on the basepoint triangle, and the boundary point on $\alpha_1$. We can see this in Figure \ref{fig:DividingNegative}. Notice that $a$, $b$, and the point on the basepoint triangle are not images of any point in $\mathcal{T}$ under the slide maps but these, together with the analogous case in $\varphi(P)$ instead of $P$, will be the only cases of this. 
     
     Because the rectangle $R'_3$ connects $R_1'$ and $R_2'$, both $a$ and $b$ are two-sided. Again, note that there can only be one region $R \in \mathcal{T}$ of this form, otherwise we would have vertices in the interior an edge of of another region, contradicting the definition of extended tower.  We now consider the levels of the regions. Assume the region $R$ is on level $i$. Then, since not all $\circ$-points of $R'_1$ come from vertices of $R$, $R'_1$ will be on level $j$ with $j \leq i$. Then, $R'_3$ will be on level $j+1$, and $R'_2$ will be on level $\max \{ j+1, i \}$.

     \begin{figure}[htp] 
    \centering
    \includegraphics[width=7cm]{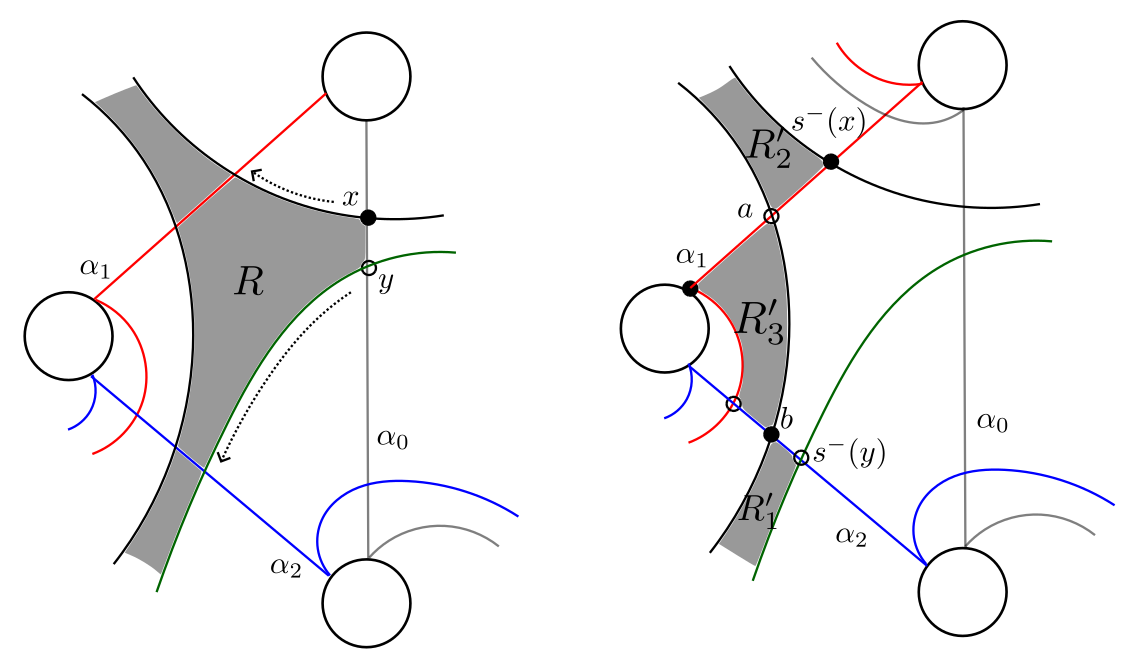}
    \caption[Dividing a negative region.]{Dividing a negative region into two negative regions $R'_1$ and $R'_2$ connected by a positive rectangle $R'_3$.}
    \label{fig:DividingNegative}
\end{figure}

    As for the neighbourhood of $\varphi(P)$, we obtain regions in the same way by reversing the role of the arcs and arc images, and changing the orientation of the arcs. Note that this switches the roles of $\alpha_1$ and $\alpha_2$ (in particular, when we are extending a region by a $6$-gon the point on the boundary that we use is on $\alpha_2$ and not $\alpha_1$, but note that neither of these points comes as the image under the slide maps). \medskip
    
    Finally, if there is a vertex $x \in \alpha_0$ in a positive region $R \in \mathcal{T}$ that is not two sided, but $s^+(x)$ is two sided because otherwise $\mathcal{T}'$ would not be replete, then $s^+(x)$ is a vertex of a negative region $R'_1$ supported in $\Gamma'$. Then we must have that $s^+(x)$ lies on $\alpha_2$, and $P$ divides what would be a region (but is not) making $x$ two-sided, and $R'_1$ is a region that makes $s^+(x)$ two-sided, see Figure \ref{fig:WeirdTwosided}. Then we also have that the region $R'_3$ has as its $\bullet$-points the point $\alpha_1(1)$, which is on the boundary, and the point $b$, which is a $\bullet$-point on a negative region of $\mathcal{T}'$, so $R'_3 \in \mathcal{T}'$. Notice that here the region $R'_1$ is on the same level as $R'$, and the region $R'_3$ is on the level above $R'_1$.

        \begin{figure}[htp]
    \centering
    \includegraphics[width=7cm]{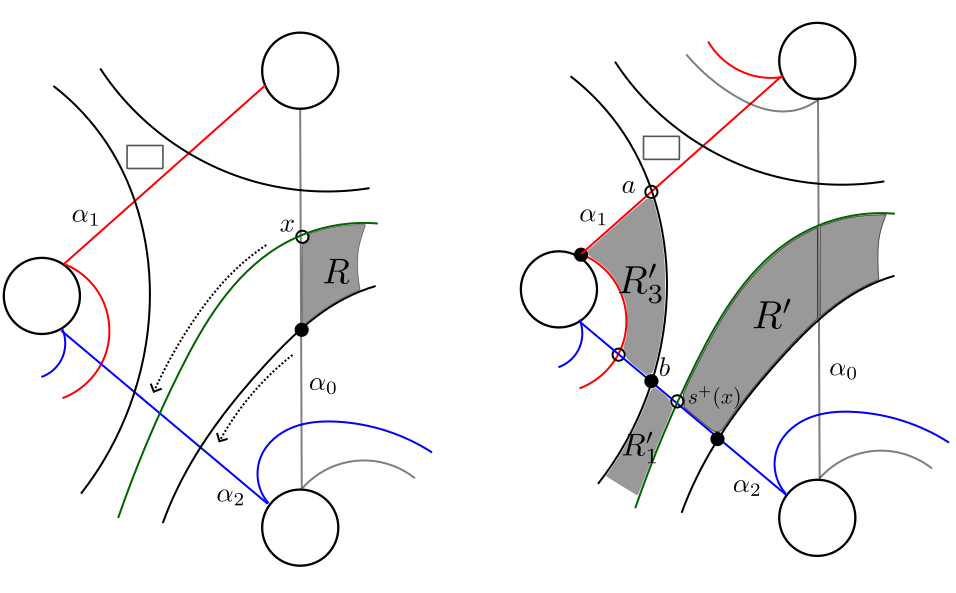}
    \caption[Two-sided points.]{The point $x$ is not two-sided but $s^+(x)$ is, because $R'_1$ is a region. However, this induces a positive region $R'_3$ and now $b$ is a two-sided $\bullet$-point in $R'_1$.}
    \label{fig:WeirdTwosided}
\end{figure} 
    
    We will now see how the properties of $\mathcal{T}$ are preserved. 

    As before, because of the way the induced regions are constructed, the result is an extended tower, which moreover is nice. It is also replete by construction.

    Also, our discussion on the levels of the regions implies that if $\mathcal{T}$ is nested then so is $\mathcal{T}'$. Being completed (or incomplete or neither) will depend on which vertices are two-sided.  

    Let us now see how the vertices being two-sided (or not) determines whether their images are two-sided. Again we will focus on a neighbourhood of $P$ and the analogous result for a neighbourhood of $\varphi(P)$ will follow from reversing the role of arcs and arc images. 

    First suppose that $\mathcal{T}$ is completed. This means that every interior vertex is two-sided, except for the connecting vertex $y_0$. But by construction this means that the images of these interior vertices in $\Gamma'$ are two-sided, except from $y_0$ (we imposed that $\alpha_0$ is not the first arc in $\Gamma$ so $s^-(y_0) = y_0$). Moreover the vertices $a$ and $b$ from the connecting region, if there is one, are also two-sided.

    There only remains to show that the $\circ$-point on the basepoint triangle formed by $\varphi(\alpha_1)$ and $\alpha_2$ is two-sided. Since $\mathcal{T}$ is completed, $\alpha_0(1)$ is the vertex of a positive region in $\mathcal{T}$, which means that $\alpha_2(1) = s^+(\alpha_0(1))$ is the vertex of a positive region in $\mathcal{T}'$. Note that not every arc image intersecting $\alpha_0$ leaves $P$ by intersecting $\alpha_2$, since initially $\varphi(\alpha_0)$ intersects $\alpha_1$ by hypothesis. Since every point in $\alpha_0$ belongs to a region, because $\mathcal{T}$ is completed, there must be a region where the arc images forming the edge on $\alpha_0$ must leave $P$ by intersecting different arcs. If this region is negative, it must split into two negative regions in $\mathcal{T}'$ connected by a (positive) rectangle which has the $\circ$-point on the basepoint triangle as one of its vertices (see Figure \ref{fig:DividingNegative}). If the region is positive, then it is extended by a $6$-gon as in Figure \ref{fig:No_Intersection}, and one of the vertices is immediately the $\circ$-point on the basepoint triangle. Therefore the $\circ$-point on the basepoint triangle is a vertex of a positive region. To see that it is also a vertex of a negative region, reverse the roles of arcs and arc images. Since $\mathcal{T}$ is completed, every point in $\varphi(\alpha_0)$ belongs to a region. But then there must be a region where the arcs forming the edge on $\varphi(\alpha_0)$ leave $\varphi(P)$ by intersecting different arc images. If this region is positive, it must split into two positive regions connected by a (negative) rectangle which has the $\circ$-point on the basepoint triangle as one of its vertices, see Figure \ref{fig:PositiveSplitRegion}. 
    
       \begin{figure}[htp]
    \centering
    \includegraphics[width=8cm]{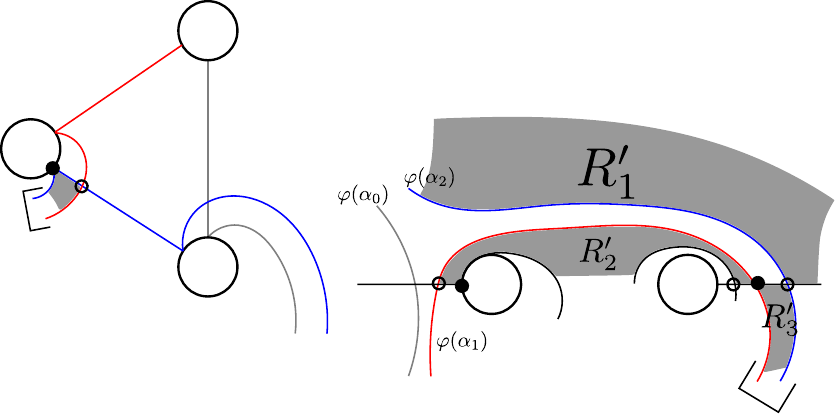}
    \caption[Using basepoint triangle.]{When there is a positive region such that the arcs forming the edge on $\alpha_0$ leave $\varphi(P)$ by intersecting different arc images, there are three induced regions $R'_1, R'_2, R'_3$ and the connecting region $R'_3$ uses the basepoint triangle.}
    \label{fig:PositiveSplitRegion}
\end{figure}

    If the region is negative, then it is extended by a $6$-gon, and one of the vertices is immediately the $\circ$-point on the basepoint triangle. Therefore the $\circ$-point on the basepoint triangle is a vertex of a negative region.

       \begin{figure}[htp]
    \centering
    \includegraphics[width=8cm]{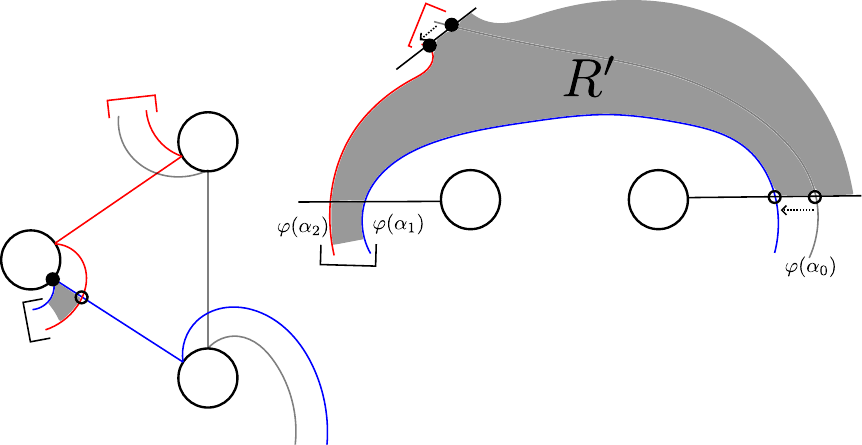}
    \caption[Using basepoint triangle.]{When there is a negative region such that the arcs forming the edge on $\alpha_0$ leave $\varphi(P)$ by intersecting different arc images, the induced region $R'$ uses the basepoint triangle.}
    \label{fig:NegativeExtendedRegion}
\end{figure}

    Now assume that $\mathcal{T}$ is incomplete. Let $R \in \mathcal{T}^-$. Then there exists a $\bullet$-point $x$ in $R$ that is two-sided. If the vertices of $R$ induce a unique region in $\mathcal{T}'$ then $s^-(x)$ is a two-sided $\bullet$-point. If they induce two regions $R'_1, R'_2$ connected by a positive region $R'_2$, as in Figure \ref{fig:DividingNegative}, there are three cases to consider. 

    First, if all the $\bullet$-points in $R$ that are two-sided have their images in $R'_1$, by construction we do not add $R'_2$ or $R'_3$ to $\mathcal{T}'$ and so the property that negative regions have a $\bullet$-point that is two-sided is preserved. The result is still a replete extended tower because the vertex $a$ is not a vertex of a positive region anymore, so we are not forced to add $R'_2$ to make $\mathcal{T}'$ replete, and there are no interior $\bullet$-points in $R'_2$ that are only vertices of positive regions, by hypothesis.

    Second, if some $\bullet$-points in $R$ have their images in $R'_1$ and some in $R'_2$, the property that negative regions have a $\bullet$-point that is two-sided is immediately preserved.

    Third, if all the $\bullet$-points in $R$ that are two-sided have their images in $R'_2$, note that the vertex $b$ which is a $\bullet$-point in $R'_1$ is now two-sided, and so the property that negative regions have a $\bullet$-point that is two-sided is preserved.

    Finally, assume that a $\circ$-point in $\alpha_0$ is not two-sided, but its image is, as in Figure \ref{fig:WeirdTwosided}. Then recall that $R'_1, R'_3 \in \mathcal{T}'$. Therefore, again the point $b$ is a $\bullet$-point in the negative region that is two-sided, so the property that negative regions have a $\bullet$-point that is two-sided is preserved.

    Therefore if $\mathcal{T}$ is incomplete $\mathcal{T}'$ is incomplete, and we are done.

\end{proof}

\begin{lemma} \label{RectangleIntersection}
    Suppose that the intersection of $P \cap \varphi(P)$ is the basepoint triangles and a collection of rectangles. Using the maps $s^+$ and $s^-$ we can construct regions giving an extended tower $\mathcal{T}'$ supported in $\Gamma'$ that has the same properties as $\mathcal{T}$. 
\end{lemma}

\begin{proof}
    This case is done in exactly the same way as Lemma \ref{NoIntersection}. Notice that now in the case where $x \in \varphi(\alpha_0)$, $s^{\pm}(x)$ are not obtained by following an arc image to its intersection with $\alpha_1$ or $\alpha_2$ but by following two sides of a rectangle (which is part of the intersection $P \cap \varphi(P)$). However, locally this only amounts to adding or removing this rectangle from the regions. Moreover, note that in this case we still have that $s^+(x) = s^-(x)$ for all points where both maps are defined.
\end{proof}

\begin{lemma}
    Suppose that the intersection $P \cap \varphi(P)$ contains a $6$-gon. Using the maps $s^+$ and $s^-$ we can construct regions giving an extended tower $\mathcal{T}'$ supported in $\Gamma'$ that has the same properties as $\mathcal{T}$. 
\end{lemma}

\begin{proof}
    The construction of $\mathcal{T}'$ is done in the same way as Lemma \ref{NoIntersection}. Let $A$ be the $6$-gon contained in $P \cap \varphi(P)$. We only need to focus on the regions given by the intersection points in $A$, because all the others have been covered by Lemmas \ref{NoIntersection} and \ref{RectangleIntersection}. There are two cases; either $\varphi(\alpha_0)$ intersects $\alpha_0$ as an edge of $A$, or it does not. \medskip

    First assume it does, and call this intersection point $x$. Then $s^+(x) \neq s^-(x)$ (if they are both defined), but for a region $R$ (either positive or negative) we will see that the induced region $R'$ is obtained by adding and removing rectangles, see Figure \ref{fig:SlideMap} for an example.

     \begin{figure}[htp]
    \centering
    \includegraphics[width=8cm]{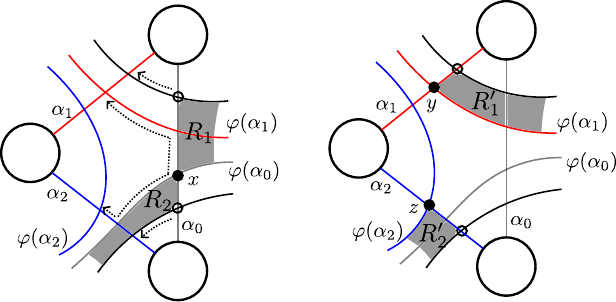}
    \caption[The case where $s^+(x) \neq s^-(x)$]{The case where $s^+(x) \neq s^-(x)$. Here $y = s^+(x) $ and $z = s^-(x)$. The dashed arrows indicate the action of $s^{\pm}$. The regions $R'_1, R'_2$ are obtained from $R_1, R_2$ by adding and removing rectangles.}
    \label{fig:SlideMap}
    \end{figure}

    Now, this case is further divided into two cases. First, if $\varphi(\alpha_1)$ also intersects $\alpha_0$, then the orientation of the arc images forces $x$ to be a $\bullet$-point, as we can see in Figure \ref{fig:6-gonIntersection}. This means that $x$ is a vertex of a negative region $R_2$, and we get an induced negative region $R'_2$ supported in $\Gamma'$.

    Now if $x$ is two-sided, it is also a vertex of a positive region $R'_1$. We also have a positive rectangle $R'_3$ using $\alpha_1$, $\alpha_2$, a side of the basepoint triangle, and either $\varphi(\alpha_1)$ or $\varphi(\alpha_2)$ (depending on which one does not intersect $\alpha_0$). Moreover, one of the vertices of this region is $s^-(x)$. This region can be completed with a rectangle $R'_4$ using the part where $\varphi(\alpha_1)$ and $\varphi(\alpha_2)$ are parallel, and one of the vertices of this region is $s^+(x)$. This means that both $s^+(x)$ and $s^-(x)$, together with the new points we have introduced (including the $\circ$-point on the basepoint triangle) are two-sided. In particular if $\mathcal{T}$ was completed, so is $\mathcal{T}'$, because all new points introduced are two-sided. Similarly, if $\mathcal{T}$ was incomplete, so is $\mathcal{T}'$, because both the negative regions that we have introduced have $\bullet$-points in common with a positive region. We can see these regions in Figure \ref{fig:6-gonIntersection}.\medskip

    \begin{figure}[htp]
    \centering
    \includegraphics[width=7cm]{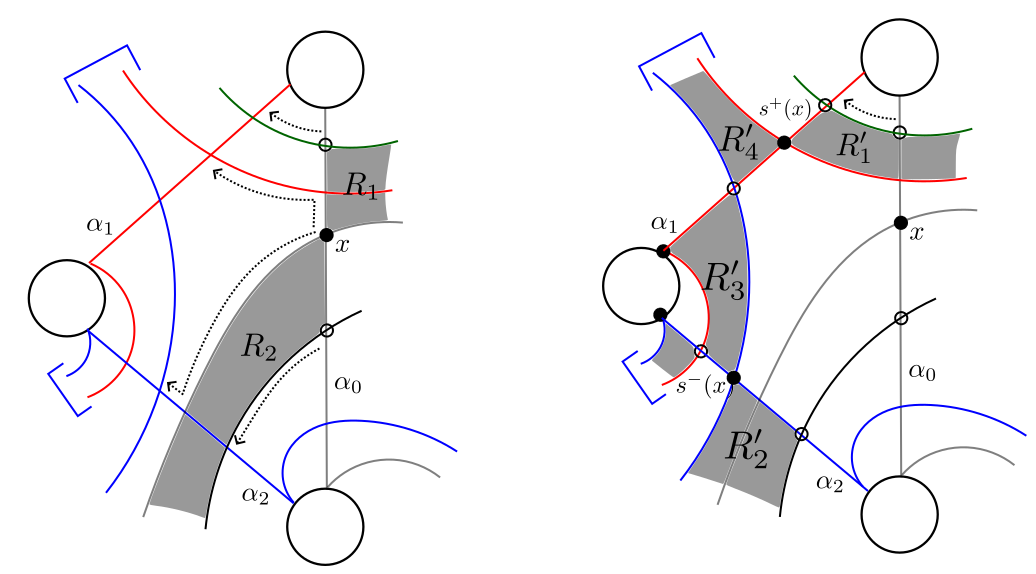}
    \caption[Inductive step for the first $6$-gon intersection case.]{The regions induced by a positive region $R_1$ and a negative region $R_2$ with a vertex on $x$. The dashed arrows indicate the action of $s^{\pm}$.}
    \label{fig:6-gonIntersection}
    \end{figure}

    If $x$ is not two-sided, then $\mathcal{T}$ cannot have been completed. So assume it was incomplete. Then $R_2$ has some $\bullet$-point $y$ in common with a positive region, which means that $R'_2$ has some $\bullet$-point in common with a positive region, and now by construction we do not include $R'_3$ and $R'_4$ in $\mathcal{T}'$, because $R'_4$ does not have any $\bullet$-points in common with a positive region in $\mathcal{T}'$. However the result is still an extended tower, which is nice and replete, using the same argument as in Lemma \ref{NoIntersection}. Moreover, the property that negative regions have a $\bullet$-point that is two-sided is preserved.

    To see that $\mathcal{T}'$ is nested, suppose that $R_2$ is on level $i$. Then $R_1$ is on level $j$, with $j > i$ (because one of its $\bullet$-points is a vertex of a level $i$ region, but it might have other vertices belonging to regions on higher levels). Then, $R'_2$ is also on level $i$, $R'_3$, is on level $i+1$, and $R'_4$ is also on level $i+1$. Therefore, $R'_1$ is on level $\max \{j, i+2 \}$. \medskip

    Second, if $\varphi(\alpha_2)$ intersects $\alpha_0$, then the orientation of the arc images forces $x$ to be a $\circ$-point. This means that it is a vertex of a positive region $R_1$, which induces a positive region $R'_1$ supported in $\Gamma'$. Now, $s^+(x)$ is two-sided because we can use the negative region $R'_4$ where $\varphi(\alpha_1)$ and $\varphi(\alpha_2)$ are parallel, and this in turn gives a positive rectangle $R'_3$ as before, where one of the vertices is $s^-(x)$. This means that the interior $\bullet$-point that we introduced in $R'_4$ is two-sided, so the property that negative regions have a $\bullet$-point that is two-sided is preserved, so if $\mathcal{T}$ is incomplete then so is $\mathcal{T}'$. Moreover, if $\mathcal{T}$ is completed, every interior vertex, in particular $x$, is two-sided, which means that it is a vertex of a negative region $R_2$, which induces a negative region $R'_2$, which means that $s^-(x)$ is two-sided, and $\mathcal{T}'$ is completed. We can see this case in Figure \ref{fig:6-gon2}. Similarly as before we can find the levels of these new regions, and thus $\mathcal{T}' $ is nested.

    \begin{figure}[htp]
    \centering
    \includegraphics[width=9cm]{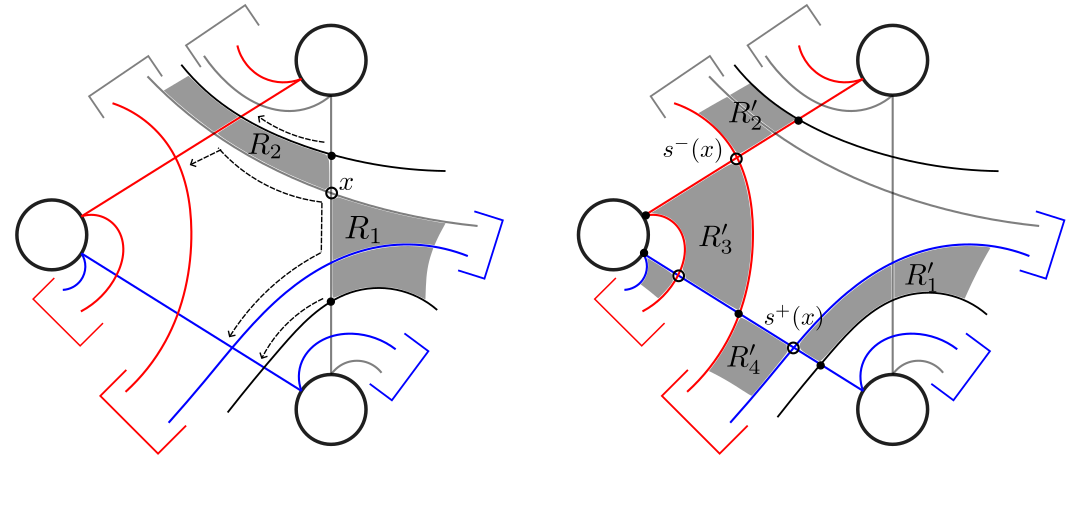}
    \caption[The second $6$-gon intersection case.]{The regions induced by a positive region $R_1$ and a negative region $R_2$ with a vertex on $x$. The dashed arrows indicate the action of $s^{\pm}$.}
    \label{fig:6-gon2}
    \end{figure}

    So now assume that $\varphi(\alpha_0)$ does not intersect $\alpha_0$ as an edge of $A$. This means that it intersects $\alpha_1$ and $\alpha_2$. Moreover, the segment of $\varphi(\alpha_0)$ that is an edge of $A$ can only be a part of a negative region $R$. But since this region cannot intersect $\varphi(P)$, the intersection with $P$ must be a rectangle. But since there is no intersection points with $\alpha_0$ there is no action of the slide maps, and the induced region is simply given by extending using a $6$-gon. We can see this region in Figure \ref{fig:LastCase}. There only remains to show that, if $\mathcal{T}$ is completed, the $\circ$-point on the basepoint triangle is two-sided, that is, it is the vertex of a positive region. But if $\mathcal{T}$ is completed, every point on $\alpha_0$ is part of a region, and thus there must be a region in $\mathcal{T}$ with a vertex between $\alpha_0(0)$ and the intersection between $\varphi(\alpha_2)$ and $\alpha_0$, and a vertex between the intersection point of $\varphi(\alpha_1)$ and $\alpha_0$ and $\alpha_0(1)$. Now notice that, because the region $R$ intersects $P$, such a region must necessarily be positive, otherwise we would have a corner of a region in the interior of a region, contradicting the definition of extended tower. But then the induced region supported in $\Gamma'$ uses the $\circ$-point on the basepoint triangle, and we are done. 
    
    Again as before we can assign a level to these regions, so $\mathcal{T}'$ is nested.

\begin{figure}[htp]
    \centering
    \includegraphics[width=7cm]{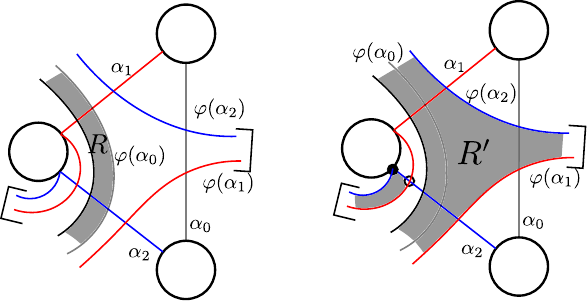}
    \caption[Extending a region with a $6$-gon]{If $\varphi(\alpha_0)$ does not intersect $\alpha_0$ we can extend this region using a $6$-gon (like we did for positive regions in $P$, see Figure \ref{fig:No_Intersection}).}
    \label{fig:LastCase}
    \end{figure}
    
\end{proof}

We are only left to show that if we have an extended tower $\mathcal{T}'$ supported in $\Gamma'$, there is an extended tower $\mathcal{T}$ supported in $\Gamma$ with the same properties as $\mathcal{T}'$.

Recall that the slide maps do not provide a bijection between vertices in $\mathcal{T} $ and vertices in $\mathcal{T}'$, because there are new vertices in $\mathcal{T}'$ that we have added. However, both $s^+$ and $s^-$ are injective, and so form a bijection with their image. Moreover, the vertices without an inverse are those that lie on endpoints of a segment disjoint from $\alpha_0$ contained in an arc image, or a segment disjoint from $\varphi(\alpha_0)$ contained in an arc. Also, in the definition of the slide maps the extended tower $\mathcal{T}$ is only used to specify the domain (we only consider vertices of regions in $\mathcal{T}$). 

Thus, given an extended tower $\mathcal{T}'$ supported in $\Gamma'$, we can define a set of intersection points of $\Gamma \cap \varphi(\Gamma)$ as follows.

Let $R' \in \mathcal{T}'^+$ and $x$ a vertex of $R'$. If the inverse of the slide map makes sense, that is, there exists a vertex $y \in \Gamma \cap \varphi(\Gamma)$ such that $s^+(y) = x$, then define $y = (s^+)^{-1}(x)$ the \emph{preimage} of $x$. We proceed analogously with negative regions and $s^-$. Now consider the set of all such preimages, i.e $\mathcal{V}(\mathcal{T}') = \{ y \in \Gamma \cap \varphi(\Gamma) \mid s^{+}(y) \textrm{ or } s^-(y) \in V(\mathcal{T} \}$.

\begin{lemma}
    Let $\mathcal{T}'$ be a completed (respectively incomplete) extended tower in $\Gamma'$, which is not just supported in $\{ \alpha_1, \alpha_2 \}$. Then the set $\mathcal{V}(\mathcal{T}')$ induces an extended tower $\mathcal{T}$ supported in $\Gamma$ that is completed (respectively incomplete).

\end{lemma}

\begin{proof}
    For regions whose entire set of vertices has a preimage we are done in the same way as the previous lemmas. So we only need to focus on regions that have one or more vertices without a preimage under $s^{\pm}$. We distinguish several different cases.

    The first case is when the region is supported in $\{ \alpha_1, \alpha_2 \}$. There are two options for this. The first one is when the region connects two regions which are not just supported in $\{ \alpha_1, \alpha_2 \}$, we will deal with this more generally in the second case. If it does not,  then the same reasoning as in Lemma \ref{TriangleFixed}  implies that it must be part of a splitting pair (otherwise $\mathcal{T}$ would not be nested), but Lemma \ref{TriangleFixed} already shows how to get the extended tower $\mathcal{T}$ in this case.

    The second case is a vertex belonging to an arc image (different from $\varphi(\alpha_1)$ and $\varphi(\alpha_2)$) that intersects $\alpha_1$ and $\alpha_2$. In this case we see that the arc image provides an edge for a positive region as in the previous case. Moreover, this means that the positive region must be connecting two negative regions, and the points of these regions that do have preimages will give a negative region in $\Gamma$. Essentially we are merging the two negative regions, which is the opposite operation to dividing a region into two connected by a third region as we had done in the previous Lemmas. The level of the new region will be the highest of the levels of the two original regions (this might also increase the level of positive regions with $\bullet$-points in the merged region) so $\mathcal{T}'$ is still nested.
    
    If there are not two regions but just one, then we have a vertex that is not two-sided, but then the preimage of a vertex from the positive region is not two-sided as in Figure \ref{fig:WeirdTwosided} (in particular neither of the extended towers can be completed, so assume that $\mathcal{T}'$ is incomplete). Moreover this means that we do not add a negative region to $\mathcal{T}$, so this would not affect whether every negative region in $\mathcal{T}$ has a two-sided $\bullet$-point. This could now result in a positive region with $\bullet$-points that do not belong to a negative region. In this case, we also do not add this region to $\mathcal{T}$ to make sure $\mathcal{T}$ is an extended tower. Now this could result in a negative region with $\circ$-points that do not belong to a positive region. We can carry on this procedure, but it must (at the latest) terminate when we reach level $0$, where we would have an extended tower $\mathcal{T} = \{R \}$, with $R$ a positive region, so $\mathcal{T}$ is nested, replete, nice, and incomplete.

    The third case is analogous to the second one, and happens when both $\varphi(\alpha_1)$ and $\varphi(\alpha_2)$ intersect an arc $\beta_j$. Observe that, as before, we can relate this case to the second one by interchanging the roles of $\Gamma'$ and $\varphi(\Gamma')$.

    Because all the vertices are inverse images under the slide maps, the resulting collection of regions $\mathcal{T}$ is an extended tower which moreover must be nice. To see whether it is replete, let $A \in \mathcal{R}(\Sigma, \varphi, \Gamma)$ such that $\textrm{Circ}(A) \subset \textrm{Circ}(\mathcal{T}) \cup \textrm{Circ}_{\partial}(\Gamma)$, and $\mathcal{T} \cup A$ is again an extended tower. Take $s^-(V(A))$. These vertices must also give negative region(s) in $\Gamma'$, so $V(A)$ is the inverse image of vertices of $\mathcal{T}'$ and thus is in $\mathcal{T}$, so $\mathcal{T}$ is replete.

    Now suppose that $\mathcal{T}'$ is completed. Then all its interior vertices are two-sided, but this implies that all interior vertices of $\mathcal{T}$ are two-sided, so $\mathcal{T}$ is completed. 
    
    So suppose that $\mathcal{T}'$ is incomplete, and let $R$ be a negative region in $\mathcal{T}$. We want to show that it has a two-sided $\bullet$-point. Take $s^-(V(R))$. These points are by construction vertices of regions in $\mathcal{T}'$. 
    The only case where a $\bullet$-point of a region in $\Gamma'$ being two-sided does not imply that its preimage is two-sided is when $P \cap \varphi(P)$ contains a $6$-gon. So suppose we have a negative region $R$ supported in $\Gamma$ with a $\bullet$-point $x \in \alpha_0 \cap \varphi(\alpha_0)$ that is not two-sided, and an induced region $R'$ supported in $\Gamma'$ such that $ y = s^-(x)$ that is two-sided. This means that $y$ is the vertex of a connecting region $R'_1$ that connects $R'$ to a negative region $R'_2$, which is a rectangle where $\varphi(\alpha_1)$ and $\varphi(\alpha_2)$ are parallel. But now, because $x$ is not two-sided, $R'_2$ has no two-sided $\bullet$-points (it only has two, a boundary point, and the point $z$ that would be $s^+(x)$ if $x$ were two-sided, and none of them can be two-sided). This means that either $\mathcal{T}$ was not incomplete, or $R'_2$ is not in $\mathcal{T}'$. But if $R'_2$ is not in $\mathcal{T}'$, since $\mathcal{T}'$ is replete $R'_1$ also cannot be in $\mathcal{T}'$. But this means that $s^-(x)$ is not two-sided --a contradiction. We can see this in Figure \ref{fig:IncompleteGoingBack}. 
    
    \begin{figure}[htp]
        \centering
        \includegraphics[height=4cm]{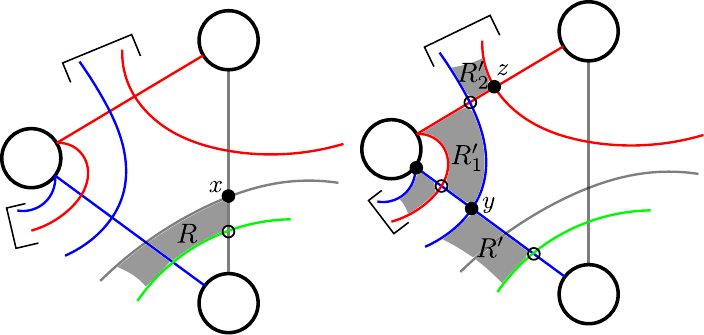}
        \caption[Going back.]{If $x$ is not two-sided but $y = s^-(x)$ is, then $\mathcal{T}$ cannot have been incomplete as $R'_2$ does not have any $\bullet$-points in common with any positive region.}
        \label{fig:IncompleteGoingBack}
    \end{figure}

    There only remains to show that this procedure does not yield an empty extended tower. For a contradiction, assume that it does. Then, there is a positive region $R$ in $\mathcal{T}'$ that does not induce a positive region in $\Gamma$. At least one vertex of $R$ must not have an inverse image under $s^+$. Moreover, $R$ cannot be a connecting region contained in $P$, in this case $R$ does not come from a positive region in $\Gamma$ (it connects two regions induced by a negative region in $\Gamma$) but one of those is in a lower level than $R$ so we can find a positive region in a lower level that also cannot induce a positive region in $\Gamma$ (if we assume that the induced extended tower is empty). Thus the only way this can happen is if we have $R$ and a negative region $R'$ that is not a connecting region, that is, there is not another positive region connected to $R$ by $R'$ (notice that we have encountered this case before with the roles of positive and negative regions reversed). For this to happen, we must have that a vertex $x$ of this region lies either on $\varphi(\alpha_1)$ or $\varphi(\alpha_2)$ and the intersection of the discs cut out by $\Gamma \cup \{ \beta \}$ and $\varphi(P)$ contains a $6$-gon, see Figure \ref{fig:EmptyTower} for an example.

    Suppose $x$ is a $\circ$-point. Then necessarily it lies on $\varphi(\alpha_2) \cap \beta_i$, for some $\beta_i \in \Gamma$. But now the negative region $R'$ does not have any common $\bullet$-points with any positive region, as it has its other $\bullet$-point on the boundary. Indeed, if it does not, we can show $\mathcal{T}'$ is not nested. Suppose that the other $\bullet$-point is not on the boundary. Then either it is not two-sided, in which case we are done because $\mathcal{T}$ is not incomplete, or it is the vertex of a positive region $R'_1$. Notice that the other $\circ$-point of $R'$ must also be a vertex of a positive region $R'_2$. Then, there must be another negative region $R'_3$ where $\varphi(\alpha_1)$ and $\varphi(\alpha_2)$ are parallel, until the boundary (if it is not until the boundary we can repeat this argument). Now let us look at the levels of the regions. Let $i$ be the level of $R'$. Then $R'_1$ is on level $j$ with $j > i$, as they share a $\bullet$-point. Similarly, $R'_2$ is on level $k$ with $k \leq i$, as they share a $\circ$-point. Then the same argument shows that $R'_3$ is on level $l$ with $l > k$, and $j \leq l$. But then we have $k \leq i < j \leq l < k$, a contradiction. Thus, the other $\bullet$-point of $R'$ is on the boundary. We can see this situation in Figure \ref{fig:Levels}.

    \begin{figure}[htp]
        \centering
        \includegraphics[height=6cm]{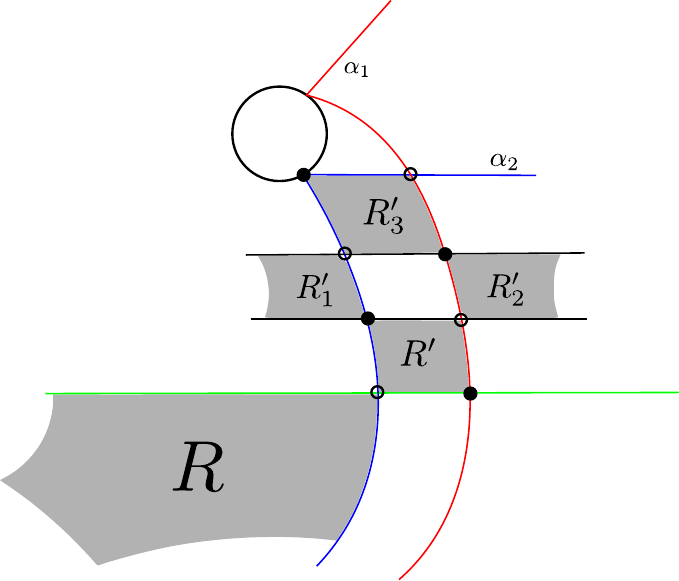}
        \caption[Empty extended tower, multiple regions.]{We cannot have multiple regions where $\varphi(\alpha_1)$ and $\varphi(\alpha_2)$ are parallel because then the extended tower would not be nested. Going from $R'$ to $R'_1$ and then $R'_3$ would increase the level, but going to $R'_2$ and then $R'_3$ would decrease it..}
        \label{fig:Levels}
    \end{figure}

    This means that $R'$ has no common $\bullet$-point with any positive region so $\mathcal{T}'$ is not incomplete. Moreover, if $\mathcal{T}$ is completed, then it must be supported in $\{ \alpha_1, \alpha_2 \}$ because there is a point in $\varphi(\alpha_1)$ that is not two-sided --a contradiction. We can see this in Figure \ref{fig:EmptyTower}.

     \begin{figure}[htp]
        \centering
        \includegraphics[height=4cm]{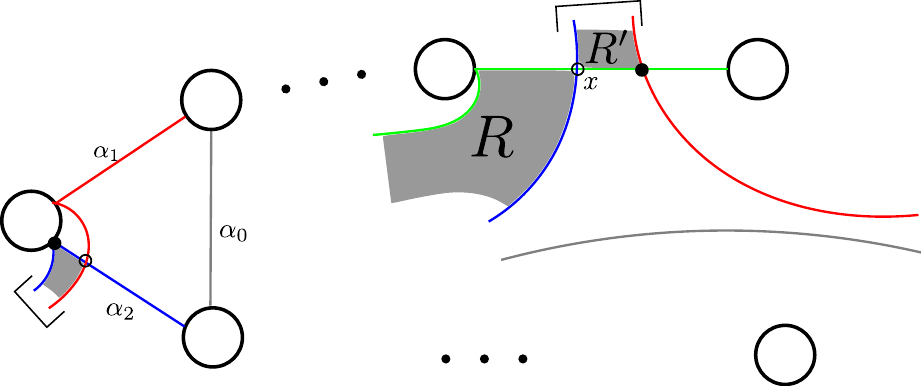}
        \caption[Empty extended tower, $\circ$-point.]{If $\mathcal{T}'$ induces an empty extended tower, then it cannot have been incomplete because $R'$ has no common $\bullet$-point with any positive region but also cannot have been completed unless it is supported in $\{ \alpha_1, \alpha_2 \}$ because there is a point in $\varphi(\alpha_1)$ that is not two-sided.}
        \label{fig:EmptyTower}
    \end{figure}

    Now suppose that $x$ is a $\bullet$-point. Then the negative region $R'$ has a $\circ$-point $y$ that is not two-sided (and not on a basepoint triangle), contradicting the definition of extended tower, see Figure \ref{fig:EmptyTower2}.

     \begin{figure}[htp]
        \centering
        \includegraphics[height=4cm]{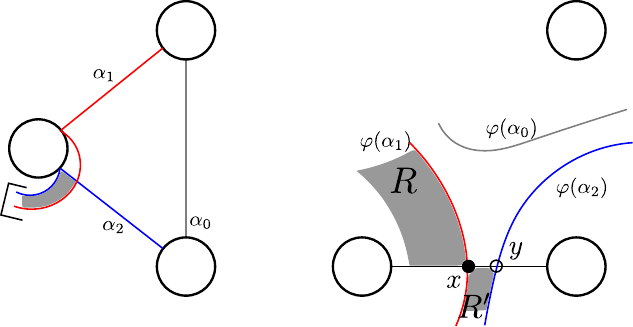}
        \caption[Empty extended tower, $\bullet$-point.]{The point $y$ is a $\circ$-point of a negative region but not a positive region.}
        \label{fig:EmptyTower2}
    \end{figure}

    \end{proof}

Recall that our setup is an arc collection $\Gamma$ such that there exists an arc $\beta$ with $\Gamma \cup \beta$ cutting out a disc from the basis. Now let $\alpha_0$ be an arc in $\Gamma$ that is not the next one to $\beta$ as we go along the boundary of the disc cut out by $\Gamma \cup  \{ \beta \}$, and $\alpha_1, \alpha_2$ arcs disjoint from this disc such that together with $\alpha_0$ they cut out a $6$-gon $P$.
Combining the previous Lemmas, we have proved the following Propositions.

\begin{proposition} \label{Induction1}
Let $\mathcal{T}$ be an extended tower in $\Gamma$. If $\alpha_1$ is not $\varphi$-contained in $\alpha_0$, then there is a (nice and replete) extended tower $\mathcal{T}'$ in $\Gamma' = (\Gamma \setminus \{ \alpha_0 \}) \cup \{ \alpha_1, \alpha_2 \}$ which is completed (respectively incomplete) if $\mathcal{T}$ is.
 \end{proposition}

\begin{proposition} \label{Induction2}
Let $\mathcal{T}'$ be an extended tower in $\Gamma'$. If $\alpha_1$ is not $\varphi$-contained in $\alpha_0$, then there is a (nice and replete) extended tower $\mathcal{T}$ in $\Gamma $ which is completed (respectively incomplete) if $\mathcal{T}'$ is.
\end{proposition}

\subsection{Main Results}

Using the base cases and the inductive step we can now show that a collection of arcs $\Gamma$ detects a left-veering arc $\beta$ if $\Gamma \cup \{\beta\}$ cuts out a disc (with the correct orientation), by which we mean that $\Gamma$ supports an incomplete extended tower if and only if $\beta $ is left-veering. Similarly, $\Gamma$ also detects fixable arc segments.

\begin{theorem} \label{TowerLV}
Let $ \{\alpha_i\}_{i=0} ^ n$ be a collection of properly embedded arcs cutting out a $(2n + 2)$-gon $P$, oriented counterclockwise, and assume $ \{\alpha_i\}_{i=1} ^ n$ are right-veering. Moreover, suppose no arc contained in $P$ is left-veering. Then $\alpha_0$ is left-veering if and only if $ \Gamma = \{\alpha_i\}_{i=1} ^ n$ supports a replete and incomplete extended tower $\mathcal{T}$.
\end{theorem}

\begin{proof}

We argue by induction on the number of arcs in our collection. The case $n=3$ is given by Proposition \ref{TriangleLV}. Now assume the result is true for $k$ arcs, with $k<n$. In the $2n$-gon at least one arc image $\varphi(\alpha_{i_0})$ will leave $P$ by intersecting $\alpha_{i_0+1}$ (because each arc image cuts a smaller subsurface inside $P$ so we can apply an innermost disc argument), creating a basepoint triangle. If $\alpha_{i_0} \neq \alpha_1$ then we can apply Propositions \ref{Induction1} and \ref{Induction2}, and there exists an incomplete extended tower $\mathcal{T}$ supported by $\Gamma$ if and only if there exists an incomplete extended tower $\mathcal{T}$ supported by $\{\alpha_1, \dots, \alpha_{n}\}\setminus (\{ \alpha_{i_0},\alpha_{i_0+1} \}) \cup \{ \beta \}$, that is nice and replete, where $\beta $ is the arc-sum of $\alpha_{i_0}$ and $\alpha_{i_0+1}$. But by induction this happens if and only if $\alpha_0$ is left-veering. 

So now suppose that $\alpha_{i_0} = \alpha_1$, and there are no other cases where the arc images create a basepoint triangle. Then every arc image $\varphi(\alpha_i)$ must leave $P$ by intersecting $\alpha_1$ (an arc image intersecting another arc would cut a smaller disc that does not contain $\alpha_1$ and we could apply our innermost disc argument there). Then $\Gamma$ supports an extended tower $\mathcal{T}$ whose regions are all rectangles as follows. The level $0$ positive region $R_1$ and its completion $R'_1$ come from the fact that $\alpha_2$ is $\varphi$-contained in $\alpha_1$ (because otherwise the arc-slide of $\alpha_1$ and $\alpha_2$ would be left-veering by Proposition \ref{TriangleLV}. Then $\varphi(\alpha_1)$ enters $P$ again and must exit by intersecting $\alpha_3$, because $\varphi(\alpha_3)$ leaves $P$ by intersecting $\alpha_1$. This forms another rectangle $R_2$, which must be completed by a rectangle $R'_2$. To see this, suppose for a contradiction that $R_2$ is not completed. Then, $\{ R_1, R'_1, R_2 \}$ would form an incomplete extended tower supported in $\{ \alpha_1, \alpha_2, \alpha_3 \}$, and then by induction their arc-sum (strictly speaking, the arc-sum with opposite orientation) would be left-veering, which contradicts the assumption that no arc contained in $P$ is left-veering. The rest of the rectangles are obtained in the same fashion. This extended tower is clearly nice and replete. Indeed, by construction positive regions are contained in $P$ and as for the negative regions, since we get one positive region per edge, we get one negative region as well. Moreover, in each negative region one edge lies on $\varphi(\alpha_1)$, and another in the corresponding edge of some $\varphi(\alpha_i)$ that goes to the boundary, see Figure \ref{fig:LastLV}. Therefore no arc image goes through the interior of the negative regions, so the extended tower is nice.
It is replete because every positive region, except the last one, already has a completion, and so there cannot be more negative regions added to the extended tower (by hypothesis, the last positive region does not have a completion because the extended tower is incomplete).

To see that it is nested, observe that $R_1$ is on level $0$ (and so is $R'_1$), and then the level increases by $1$ with each positive region.

However, if the extended tower is incomplete, the last (positive) rectangle cannot have a completion (otherwise the tower would be completed), and the same argument as in Proposition \ref{InitialLV} shows that $\alpha_0$ is left-veering. Conversely, if $\alpha_0$ is left-veering, suppose for a contradiction that the extended tower is completed, that is, the last positive rectangle does have a completion. But then the edge on $\alpha_1$ must be restricted. However the fact that $\alpha_0$ is left-veering means that its image must intersect this edge --a contradiction. We can see this in Figure \ref{fig:LastLV}.

\begin{figure} [htp]
    \centering
    \includegraphics[width=12cm]{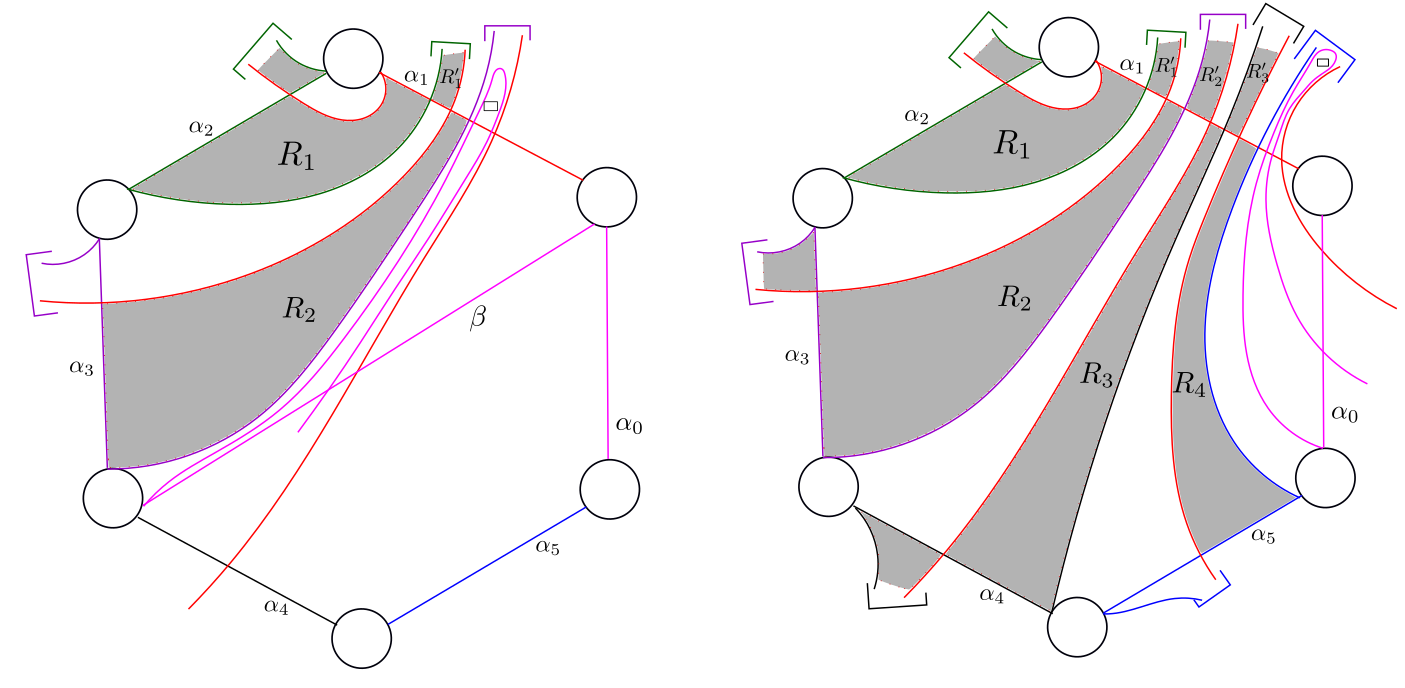}
    \caption[Incomplete extended tower from left-veering arc.]{On the left, if the incomplete tower is supported in a smaller collection of arcs then their arc-sum $\beta$ (with opposite orientation) is left-veering by induction. On the right, the incomplete extended tower. Observe we get one positive region and one negative region per edge $\alpha_2, \cdots \alpha_5$, with the positive regions contained in $P$ and the negative regions disjoint from $\varphi(\Gamma)$ in their interior.}
    \label{fig:LastLV}
\end{figure}

\end{proof}

\begin{theorem} \label{TowerFixedArc}
Let $ \{\alpha_i\}_{i = 0} ^n$ be a collection of properly embedded right-veering arcs cutting out a $(2n + 2)$-gon $P$, oriented and indexed counterclockwise, and suppose no arc contained in $P$ is left-veering. Let $\gamma$ be an arc segment contained in $P$ starting between $\alpha_n$ and $\alpha_0$ and ending in the interior of $\alpha_1$. Then $\gamma$ is fixable by $\varphi$ if and only if $\Gamma =\{\alpha_i\}_{i = 1} ^n $ supports a completed extended tower $\mathcal{T}$ whose connecting vertex coincides with $\gamma \cap \alpha_1$.
\end{theorem}

\begin{proof}
We argue by induction on the number of arcs in our collection. The case $n=3$ is given by Proposition \ref{TriangleFixed}. Now assume the result is true for $k$ arcs, with $k < n$. In the $2n$-gon at least one arc image $\varphi(\alpha_{i_0})$ will leave $P$ by intersecting $\alpha_{i_0+1}$, creating a basepoint triangle (because each arc image cuts a smaller subsurface inside $P$). If $\alpha_{i_0} \neq \alpha_1$ then we can apply Propositions \ref{Induction1} and \ref{Induction2}, and there exists a completed extended tower $\mathcal{T}$ supported by $\Gamma$ if and only if there exists a completed extended tower $\mathcal{T}$ supported by $\{\alpha_1, \dots \alpha_{n}\}\setminus (\{ \alpha_{i_0},\alpha_{i_0+1} \}) \cup \{ \beta \}$, that is replete, where $\beta $ is the arc-slide of $\alpha_{i_0}$ and $\alpha_{i_0+1}$. But by induction this happens if and only if $\gamma$ is fixable by $\varphi$. 

So now suppose that $\alpha_{i_0} = \alpha_1$, and there are no other cases where the arc images create a basepoint triangle. Then every arc image $\varphi(\alpha_i)$ must leave $P$ by intersecting $\alpha_1$. Then the extended tower is a collection of rectangles, obtained as in Theorem \ref{TowerLV}, with the difference that now if the extended tower is completed then $\gamma$ must be fixable because it cannot go to either right or left, and conversely if $\gamma$ is fixable then the extended tower must be completed. We can see this last case in Figure \ref{fig:LastInductionCase}.

\begin{figure} [htp]
    \centering
    \includegraphics[width=8cm]{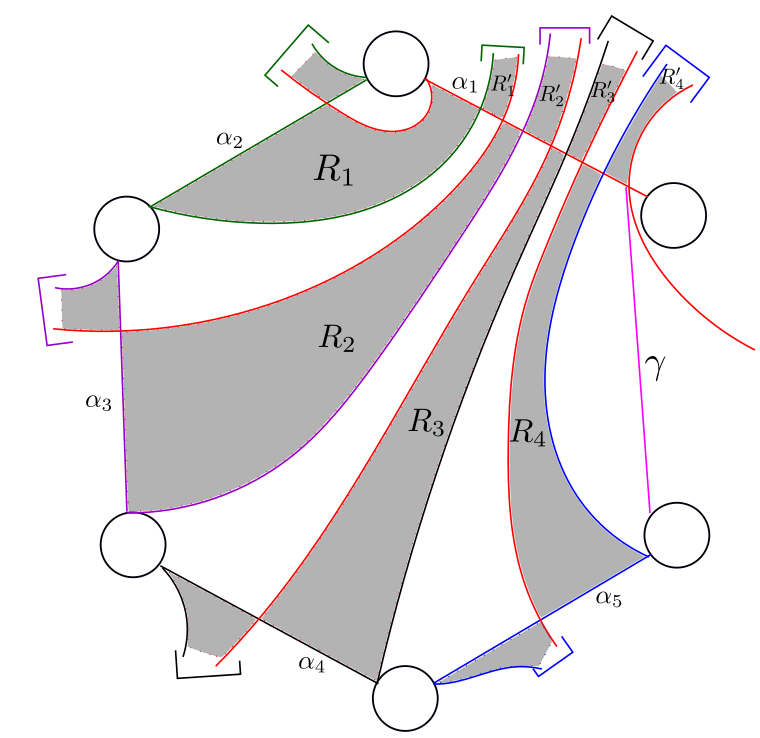}
    \caption[Completed extended tower from fixable arc.]{The extended tower when the only basepoint triangle is formed by $\alpha_1$ and $\alpha_2$.}
    \label{fig:LastInductionCase}
\end{figure}

\end{proof}

Our aim is to detect a left-veering arc with a collection of extended towers, each of which will detect a segment of the arc. However, in the setup we have so far, we only detect arcs (or arc segments) with a starting point on the boundary. To get around this, let $ \mathcal{C} = \{\alpha_0, \dots ,\alpha_n\}$ be an arc collection cutting out a $(2n + 2)$-gon $P$, oriented and labelled counterclockwise. Now assume that there is a point $x \in \alpha_n$ that is the endpoint of a fixed arc segment disjoint from $P$, so by Theorem \ref{TowerFixedArc} there exists an extended tower with $x$ as its connecting veretx. Let $\alpha'_n$ be the (oriented) arc segment between $\alpha_n(0) $ and $x$, and let $\alpha'_0$ be the (oriented) arc segment that goes from $x$ to $\alpha_1(0)$. Then $\mathcal{C}' = \{\alpha_0', \alpha_1, \dots ,\alpha_{n-1}, \alpha_n'\}$ is a collection of arc segments that cut out a $(2n+1)$-gon $P'$. Moreover, at $x$, the tangent vector of $\alpha'_n$ followed by the tangent vector of $\varphi(\alpha'_n)$ define the orientation of $\Sigma$ (because $x$ is the connecting vertex of a completed extended tower) so in a slight abuse of notation we can say that the arc segment $\alpha'_n$ is right-veering, and we can adapt the terminology and methods of extended towers to $\mathcal{C}'$ (since the only properties we use in the results is that the arcs bound a disc and are disjoint and right-veering). In particular Theorems \ref{TowerLV} and \ref{TowerFixedArc} still hold. We can see this situation in Figure \ref{fig:PartialTower}.

\begin{definition} \label{PartialTower}
Let $\mathcal{T}$ be an extended tower supported in the collection $\mathcal{C}' \setminus \{\alpha_0'\}$. We say $\mathcal{T}$ is a \emph{partial extended tower}, and $x$ its \emph{starting point}.
\end{definition}

We will also say, in a slight abuse of notation, that such a partial extended tower $\mathcal{T}$ is \emph{supported in $\Gamma = \{\alpha_1, \dots \alpha_n \}$}.

\begin{figure} [htp]
    \centering
    \includegraphics[width=5cm]{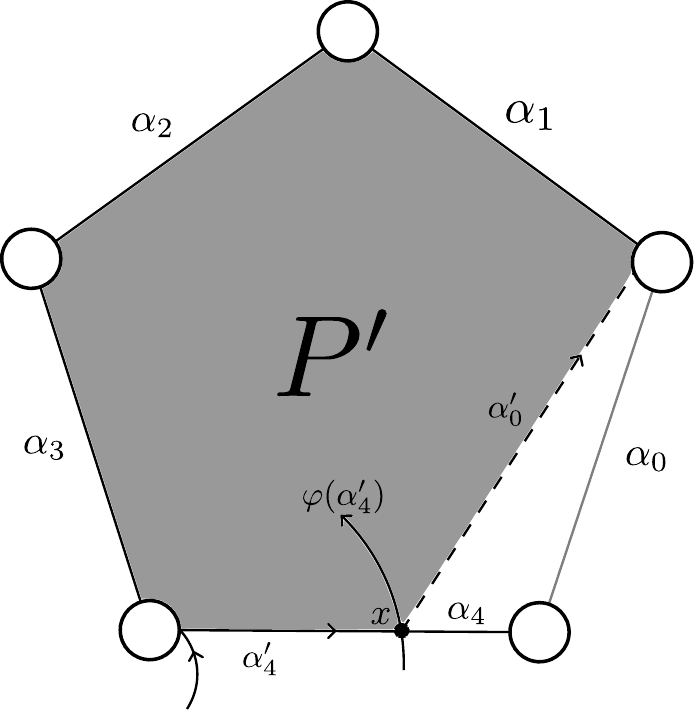}
    \caption[Partial extended towers.]{The setup for a partial extended tower supported in $  \{ \alpha_1, \alpha_2, \alpha_3, \alpha'_4 \}  $. The fact that $x$ is the connecting vertex of an extended tower means that at $x$ the tangent vector of $\alpha'_4$ followed by the tangent vector of  $\varphi(\alpha'_4)$ define the orientation of $\Sigma$.}
    \label{fig:PartialTower}
\end{figure}

We are now almost ready to prove that we can detect the existence of a left-veering arc from a basis of the surface, we just need one more definition. Notice that for our results to work we need the arcs cutting out a disc to be distinct. However, if we simply take a basis as our collection of arcs, this does not necessarily happen. We solve this issue by duplicating every arc from the basis. This has the effect that when a left-veering arc intersects the basis it actually intersects two (isotopic) arcs $\alpha$ and $\beta$. Moreover, the segment of the left-veering arc before this intersection will cut out a disc with a collection containing one of the arcs (say, $\alpha$) and the segment after the intersection will cut out a disc with a collection containing the other arc (say, $\beta$). Then an extended tower in the first collection will have a connecting vertex on $\alpha \cap \varphi(\alpha)$ and an extended tower in the second collection will have a starting point on $\beta \cap \varphi(\beta)$. Because we want to construct the left-veering arc from the extended towers, we want to relate these two points. 

\begin{remark}
    We want to distiguish $\alpha$ and $\beta$ even though they are isotopic because $\alpha$ could also be an arc in the second collection, and we want each arc from the collection supporting an extended tower to be distinct. Note though that we only need to duplicate each arc from a basis because the disc cut by a basis has exactly two copies of each arc, so each extended tower will have at most two isotopic arcs.
\end{remark}

\begin{definition}
    Let $\alpha$ and $\beta$ be two isotopic properly embedded arcs. Then $\varphi(\alpha)$ and $\varphi(\beta)$ are always parallel and, for an intersection $x \in \alpha \cap \varphi(\alpha)$, then we have a small rectangle contained in the intersection of the thin strip between $\alpha$ and $\beta$ with its image which has $x$ as a vertex. We call the vertex of this rectangle $y \in \beta \cap \varphi(\beta) $ the \emph{adjacent point to $x$}.
\end{definition}

\begin{figure} [htp]
    \centering
    \includegraphics[width=4cm]{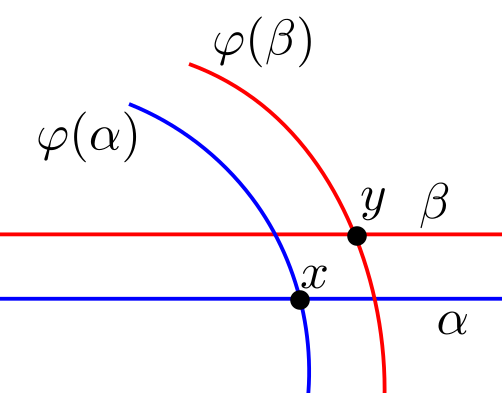}
    \caption[Adjacent points.]{The point $y$ that is adjacent to the point $x$.}
    \label{fig:AdjacentPoint}
\end{figure}

We can now prove our main result, which we restate for the convenience of the reader.

\main*

\begin{proof}
Cut $\Sigma $ along the arcs $\Gamma$, making a disc, and orient them counterclockwise. Then $\gamma$ is fixable until it intersects one of the arcs (if it is disjoint from the basis then it will cut out a disc with a subcollection of arcs from $\Gamma$ and then Theorem \ref{TowerLV} gives an incomplete extended tower). Then Theorem \ref{TowerFixedArc} gives the first completed extended tower $\mathcal{T}_1$. Then $\gamma$ is again fixable from the adjacent point to this point until the next intersection with $\Gamma$ (and also in the small rectangle between the adjacent points), and now modifying Theorem \ref{TowerFixedArc} for the case where we have an arc segment with fixed endpoints gives the first completed partial extended tower. Repeat until $\varphi(\gamma)$ goes to the left, and then modifying Theorem \ref{TowerLV} gives the incomplete partial extended tower. 

For the converse, observe that both Theorem \ref{TowerFixedArc} and Theorem \ref{TowerLV} are if and only if statements, and the left-veering arc is constructed by joining all the fixed arc segments given by the completed extended towers and the left-veering arc segment given by the incomplete one. Observe that the small arc segments between adjacent points necessary to connect all of the arc segments given by Theorems \ref{TowerFixedArc} and \ref{TowerLV} are fixable because their endpoints are fixed points and they lie in the thin strip between isotopic arcs.
\end{proof}

Notice that the number of extended towers $N$ in a collection detecting a left-veering arc in the construction described by Theorem \ref{TowerCollection1} coincides with the number of intersections of the arc with the basis. Moreover, each of these points corresponds to a point $\alpha \cap \varphi(\alpha)$ for some arc $\alpha$ in the basis. Once we fix a basis, the number of such points is finite and gives an upper bound for $N$. Moreover, each extended tower is also a finite collection of regions. Therefore, Theorem \ref{TowerCollection1} implies the existence of an algorithm that takes as input a basis of arcs and their images and either produces a collection of extended towers giving a left-veering arc or terminates in a finite number of steps, which means that the monodromy is right-veering.

\section{Examples} \label{Example}

Theorem \ref{TowerCollection1} does not make any assumptions on the number of extended towers needed to detect a left-veering arc, and a natural question is whether multiple extended towers are always needed, and, if the answer is affirmative, whether there is an upper bound on the number of extended towers that does not depend on the choice of basis. Example \ref{LargeIntersection} shows that indeed some cases require multiple extended towers and the number of extended towers required can be made arbitrarily large. Recall that the arc segments that form a left-veering arc $\gamma$ are determined by the intersections of $\gamma$ with a basis, and each segment is detected with an extended tower. Therefore for any natural number $n$ we construct an open book which is not right-veering and a basis such that each arc has more than $n$ intersections with every left-veering arc.

\begin{example} \label{LargeIntersection}
    Let $\Sigma$ be a planar surface with $4$ boundary components $\{ C_i\}_{i = 1}^4$, and $\varphi = \tau_1\tau_2\tau_3\tau_a\tau_b^{-1}$ (where $\tau_i$ represents a positive Dehn twist around the boundary component $C_i$) as shown in Figure \ref{fig:Example}. Then $(\Sigma, \varphi)$ is not a right-veering open book, as the arc $\gamma$ in Figure \ref{fig:Example} is left-veering. However, any left-veering arc has to start in the boundary component $C_4$, and moreover has to intersect $b$ before it intersects $a$. In particular, this means that it has to intersect the arc $\delta$ going from $C_1$ to $C_3$, and thus the curve $c$ that separates $C_1$ and $C_3$ from $C_2$ and $C_4$. We choose a basis $\mathcal{B}$ of right-veering arcs such that all arcs intersect $c$. For every natural number $n$, let $\mathcal{B}_n = \tau_c^n(\mathcal{B})$. Then $\mathcal{B}_n$ is a basis which is also right-veering and every arc of $\mathcal{B}_n$ intersects every left-veering arc more than $n$ times. But this in turn implies that more than $n$ extended towers are needed for $\mathcal{B}_n$ to detect a left-veering arc.

    \begin{figure} [htp]
    \centering
    \includegraphics[width=8cm]{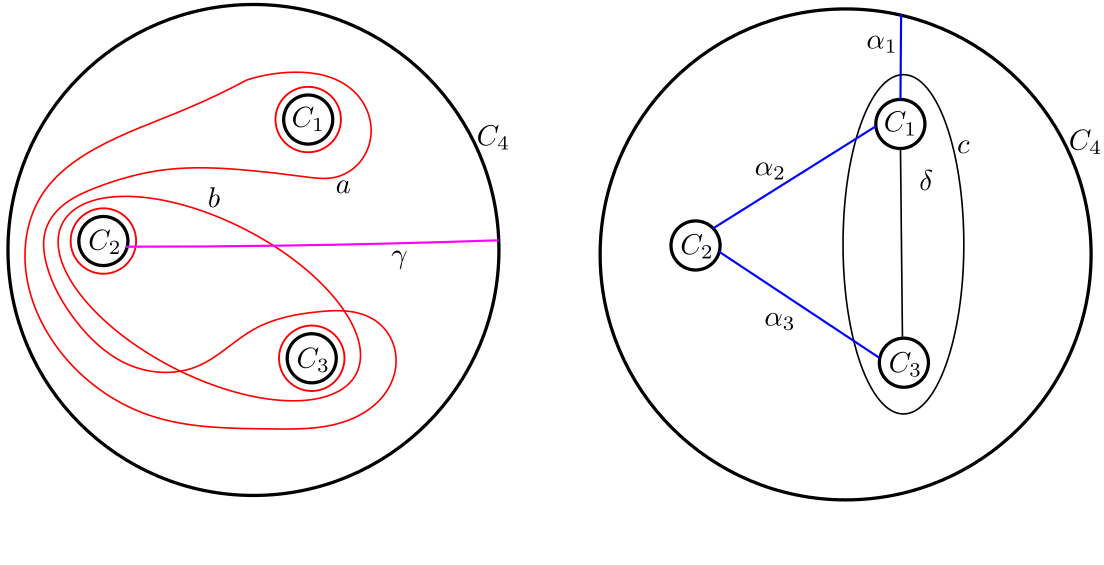}
    \caption[Arbitrarily many extended towers needed.]{On the left, the surface $\Sigma$ with the curves involved in the monodromy, and a left-veering arc $\gamma$. On the right, a basis $\mathcal{B} = \{ \alpha_1, \alpha_2, \alpha_3 \}$ such that every arc in $\mathcal{B}$ intersects the curve $c$.}
    \label{fig:Example}
    \end{figure}

\end{example}

We note, however, that this procedure is not very efficient, since it produces an extended tower for each intersection of the left-veering arc with the basis, and it would in some cases be possible to reduce the number of extended towers needed to detect a left-veering arc.

\begin{example}
    Let $(\Sigma, \varphi)$ be the open book from Example \ref{LargeIntersection}, and consider the basis $\{ \alpha_1, \alpha_2, \beta_3   \}$ on the left-hand side of Figure \ref{fig:Inefficient}. The left-veering arc $\gamma$ has an intersection point $x$ with the basis. 
    
    By the procedure we have described, we double the arcs from the basis, and find fixable arc segments and left-veering arc segments disjoint from the arcs in the (duplicated) basis except at their endpoints. Thus, the fact that $\gamma$ is left-veering is detected in two steps. First we detect that $\gamma_1$ is a fixable arc segment with a completed extended tower $\mathcal{T}_1$ supported in $\{ \alpha_1, \alpha_2  \}$, because they together with the dashed arc cut out a disc $P_1$ from $\Sigma$. This extended tower has $x$ as its connecting vertex. Then, we detect that $\gamma_2$ is a left-veering arc segment with an incomplete extended tower $\mathcal{T}_2$ supported in $\{ \beta_1, \beta_2, \beta_3 \}$, because they together with $\gamma_2$ cut out a disc $P_2$ from $\Sigma$. This (partial) extended tower has $y$, which is is the adjacent point to $x$, as its starting point. Thus the arc $\gamma$, which is the union of $\gamma_1$ and $\gamma_2$, is left-veering. We can see this on the left hand side of Figure \ref{fig:Inefficient}. 

    However, we can also detect the fact that $\gamma$ is a left-veering arc directly with a single extended tower. This is because it cuts out a disc $P$ together with $\alpha_1, \beta_1, \beta_3$, so these $3$ arcs support an incomplete extended tower. However, the procedure we described will not find it because the interior of $P$ is not disjoint from the basis arcs. We can see this on the right hand side of Figure \ref{fig:Inefficient}.

    \begin{figure} [htp]
    \centering
    \includegraphics[width=8cm]{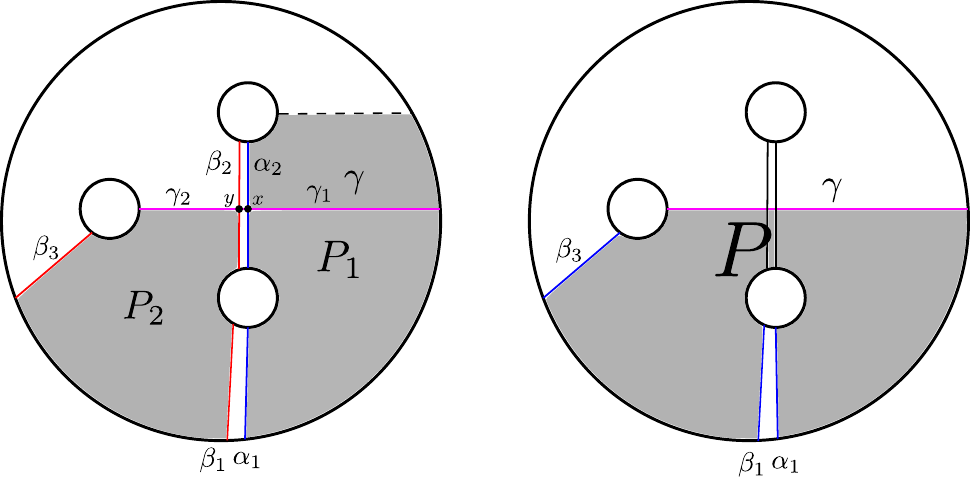}
    \caption[The algorithm is inefficient.]{On the left, the procedure we described uses two extended towers to detect the fact the $\gamma$ is left-veering. However, the disc $P$ on the right implies that only one extended tower is needed.}
    \label{fig:Inefficient}
    \end{figure}

\end{example}

We end with a detailed example of a collection of two extended towers detecting a left-veering arc.

\begin{example}
    Let $\Sigma$ be a planar surface with $7$ boundary components, and $\varphi$ the mapping class determined by the image of the basis $\mathcal{B}$ as shown in Figure \ref{fig:LastExample}, where the straight lines denote arcs and curved lines denote images of the arcs. At first inspection, it is unclear whether this monodromy is right-veering.

    \begin{figure} [htp]
    \centering
    \includegraphics[width=9cm]{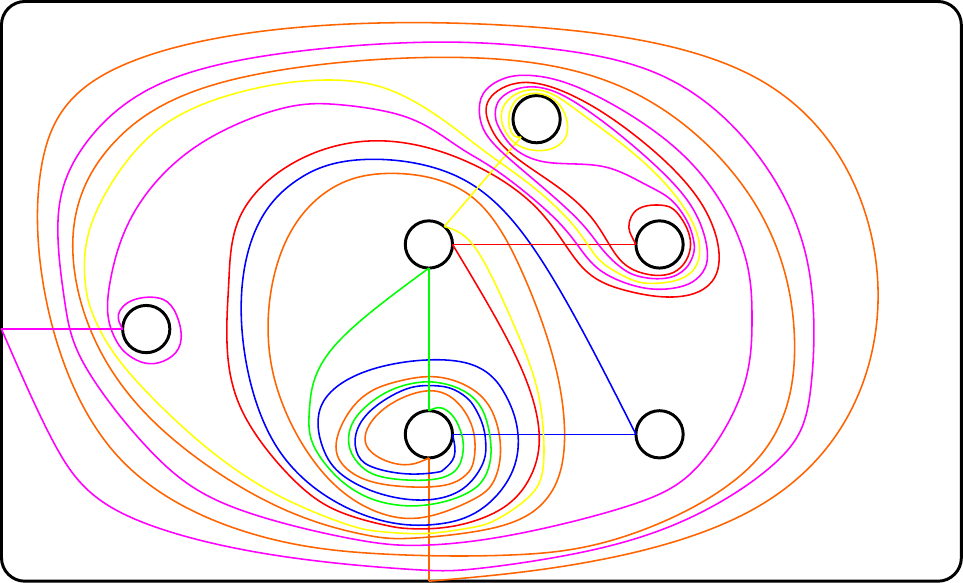}
     \caption[The monodromy.]{The monodromy of the open book is determined by the image of the arcs in the chosen basis --here the arcs are straight lines, and curved ones are their images.}
    \label{fig:LastExample}
    \end{figure}

    However, using Theorem \ref{TowerCollection1} we find two extended towers $\mathcal{T}_1, \mathcal{T}_2$ that construct a left-veering arc. 
    
    First, we take the arcs $\alpha_1$, $\alpha_2$, and $\alpha_3$ in Figure \ref{fig:FirstTower}. There is a positive region $R_1$ using $\alpha_1$ and $\alpha_3$ which is on level $0$ because its $\bullet$-points are on the boundary. Then, the region $R'1$ is a negative region on level $0$ with the same $\circ$-points as $R_1$. One of its $\bullet$-points is the point $x$ that we want to use for the construction of the left-veering arc, however the other one is still not on the boundary, so we need to keep going to obtain a replete extended tower. For that, we take the positive region $R_2$, which is on level $1$ because it has a $\bullet$-point on a level $0$ region. This gives us the negative level $1$ region $R'_2$, with the same $\circ$-points as $R_2$. Moreover, all its $\bullet$-points are on the boundary. Therefore, $\mathcal{T}_1 = \{R_1, R'_1, R_2, R'_2\}$ is a replete extended tower that is completed. It is immediate from Figure~\ref{fig:FirstTower} that it is also nice. Indeed, $R_1$ and $R_2$ are contained in the square cut out by $\alpha_1, \alpha_2, \alpha_3$, and their arc sum, and $R'_1, R'_2$ are disjoint from $\varphi(\alpha_i)$ in their interior for $i = 1,2,3$. Then, by Theorem~\ref{TowerFixedArc}, the arc segment $\gamma_1$, which starts at the endpoint of $\alpha_1$ and ends in $x$; is fixable, see bottom right of Figure \ref{fig:FirstTower}. 
    
    \begin{figure} [htp]
    \centering
    \includegraphics[width=9cm]{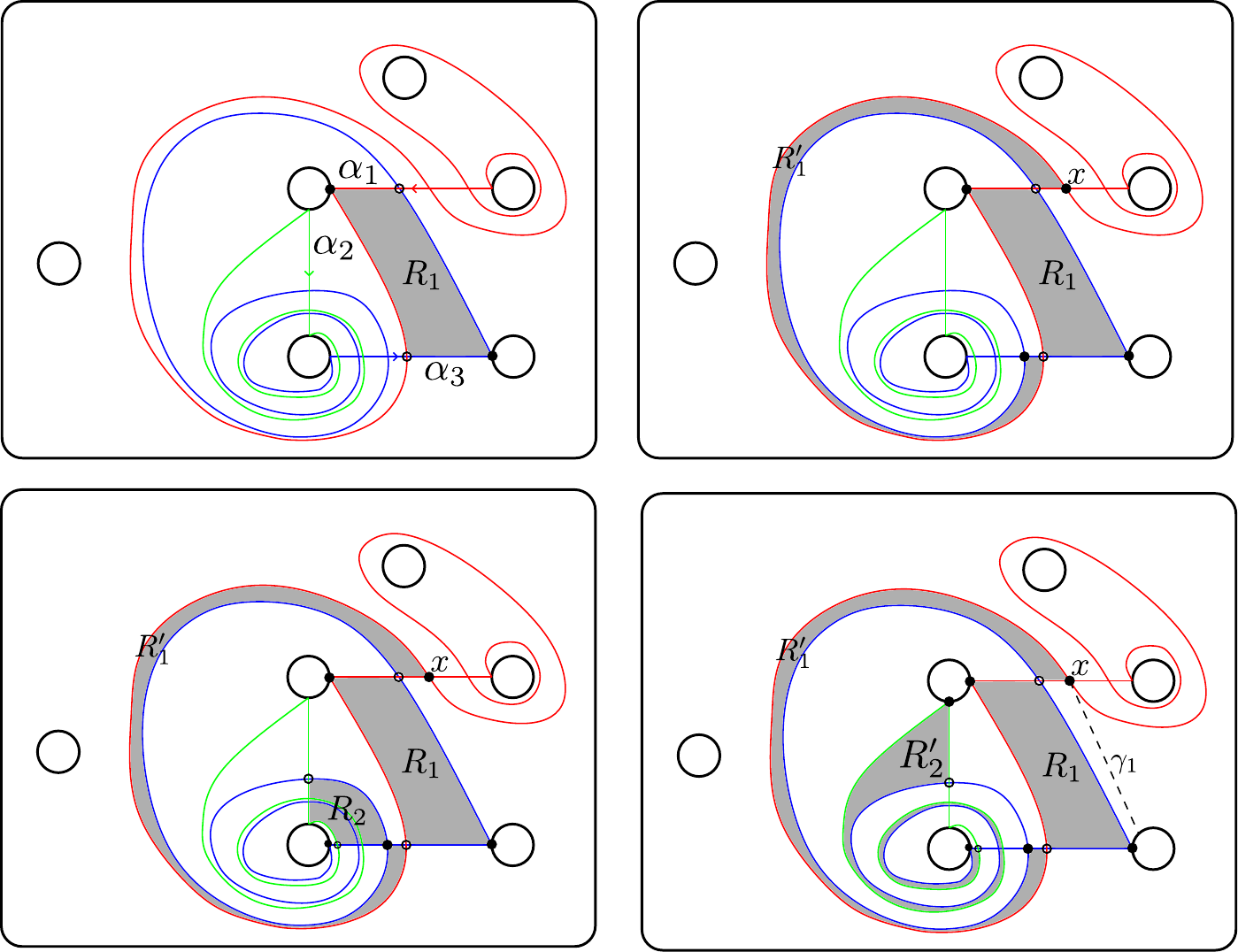}
    \caption[The first extended tower.]{On the top left, the positive level $0$ region $R_1$, which is completed by the negative level $0$ region $R'_1$, shown on the top right. Then, on the bottom left, the positive level $1$ level region $R_2$, which is completed by the negative level $1$ region $R'_2$ on the bottom right (here we removed $R_2$ from the figure for clarity, as $R_2$ and $R'_2$ overlap). Then, the arc $\gamma_1$ is fixable. Note that the rest of the arcs in the basis have also been removed for clarity.}
    \label{fig:FirstTower}
    \end{figure}
    Second, we take $\alpha'_1$ and $\alpha_4$, where $\alpha'_1$ is the segment of $\alpha_1$ (now with its orientation reversed so that it induces the correct orientation in the regions) up to $x$. These two arcs support a positive level $0$ region $R_3$ which is incomplete, see the left hand side of Figure \ref{fig:SecondTower}.
    Then, the (partial) extended tower $\mathcal{T}_2 = \{ R_3 \}$, is incomplete by construction. Moreover, it is clearly nice, and it is replete because there are no more possible regions supported in $\{\alpha'_1, \alpha_4\}$. Thus by Theorem \ref{TowerLV} (indeed Proposition \ref{InitialLV} is enough as there is a unique region) the segment $\gamma_2$ is left-veering. 
    
    This implies that the arc $\gamma = \gamma_1 \cup \gamma_2$ is left-veering, which we can see on the right hand side of Figure \ref{fig:SecondTower}.

    \begin{figure} [htp]
    \centering
    \includegraphics[width=9cm]{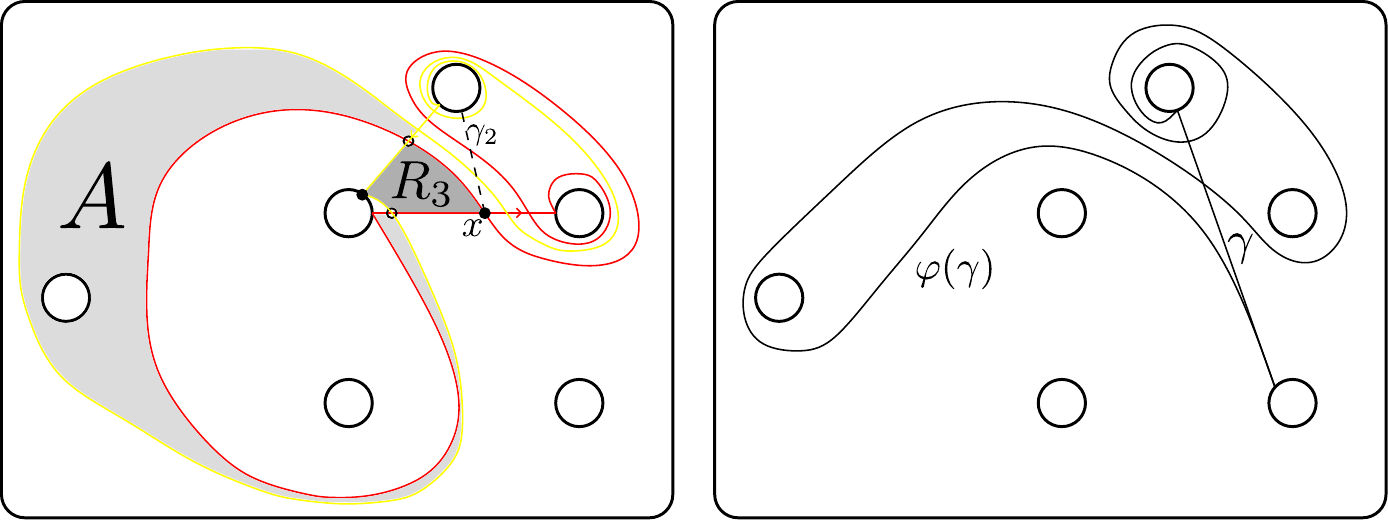}
    \caption[The second extended tower.]{On the left, the positive level $0$ region $R_3$, which cannot be completed as the shaded part $A$ of the surface where the completion would have to be is not a disk. This implies that $\gamma_2$ is left-veering. Joining $\gamma_1$ and $\gamma_2$ gives us the left-veering arc $\gamma$, which we see on the right, together with its image.}
    \label{fig:SecondTower}
    \end{figure}

    We remark that there are other left-veering arcs in this open book, such as $\widetilde{\gamma}$ in Figure \ref{fig:OtherArc}, which is detected by a single extended tower supported in the arcs $\alpha_1, \alpha_2, \alpha_3, \alpha_4, \widetilde{\alpha_1}, \widetilde{\alpha_4}$, where $\widetilde{\alpha_1}, \widetilde{\alpha_4}$ are copies isotopic to $\alpha_1$ and $\alpha_4$. However, this extended tower is more complicated in the sense that it has more regions and moreover it requires doubled arcs. In contrast, the previous one does not, as each extended tower requires a single copy of each arc ($\alpha_1$ is used twice, but in different extended towers, so for practical purposes we can draw it just once). Therefore, while it would be ``easier" to find for a computer, it might be harder to find for a human\footnote{Indeed, for the author it was harder to find than $\gamma$.}. Finally, we note that the extra left-veering arcs may be removed by post-composing with positive twists along curves that miss $\varphi(\gamma)$, leaving us with only left-veering arcs that intersect the basis.

    \begin{figure} [htp]
    \centering
    \includegraphics[width=9cm]{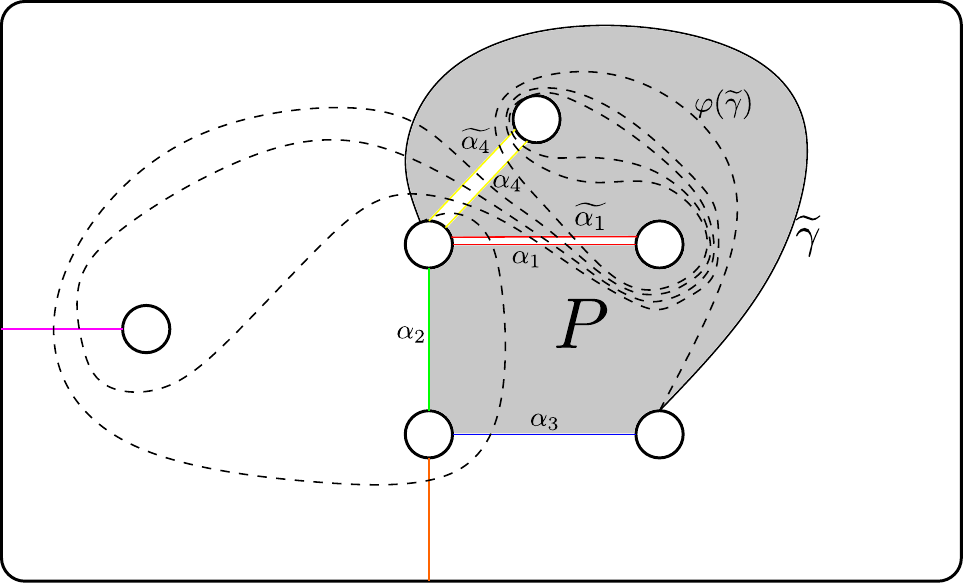}
    \caption[Another left-veering arc.]{The arc $\widetilde{\gamma}$ is also left-veering, and we could detect it using the arcs $\alpha_1, \alpha_2, \alpha_3, \alpha_4, \widetilde{\alpha_1}, \widetilde{\alpha_4}$, as they, together with  $\widetilde{\gamma}$, cut out a disk $P$ from the surface.}
    \label{fig:OtherArc}
    \end{figure}
\end{example}

\bibliographystyle{myamsalpha} 
\bibliography{main}

\end{document}